\documentclass[10pt]{amsart}

\usepackage{amsmath,amssymb,amscd,amsthm,amsxtra}

\usepackage{graphicx, mathrsfs}
\usepackage{amssymb,amsmath,amsthm}
\usepackage{mathrsfs,dsfont}

\makeatletter
\DeclareRobustCommand\widecheck[1]{{\mathpalette\@widecheck{#1}}}
\def\@widecheck#1#2{%
   \setbox\z@\hbox{\m@th$#1#2$}%
   \setbox\tw@\hbox{\m@th$#1%
      \widehat{%
         \vrule\@width\z@\@height\ht\z@
         \vrule\@height\z@\@width\wd\z@}$}%
   \dp\tw@-\ht\z@
   \@tempdima\ht\z@ \advance\@tempdima2\ht\tw@ \divide\@tempdima\thr@@
   \setbox\tw@\hbox{%
      \raise\@tempdima\hbox{\scalebox{1}[-1]{\lower\@tempdima\box\tw@}}}%
   {\ooalign{\box\tw@ \cr \box\z@}}}
\makeatother

\newtheorem{theorem}{Theorem} [section]

\newtheorem{lemma}[theorem]{Lemma}
\newtheorem{proposition}[theorem]{Proposition}
\newtheorem{remark}[theorem]{Remark}

\newtheorem{corollary}[theorem]{Corollary}

\begin{document}

\title[Bounds on Sobolev norms for NLS on $S^1$]{Bounds on the growth of high Sobolev norms of solutions to Nonlinear Schr\"{o}dinger Equations on $S^1$}
\author{Vedran Sohinger}
\address{Massachusetts Institute of Technology, Department of Mathematics}
\email{\tt vedran@math.mit.edu}

\begin{abstract}
We consider Nonlinear Schr\"{o}dinger type equations on $S^1$. In this paper, we obtain polynomial bounds on the growth in time of high Sobolev norms of their solutions. The key is to derive an iteration bound based on a frequency decomposition of the solution, which is different than the iteration bound first used by Bourgain in \cite{B2}. We first look at the NLS equation with nonlinearity of degree $\geq 5$. For $q=5$, Bourgain in \cite{B4} derives stronger bounds using different techniques. However, our approach works for higher nonlinearities, where the techniques from \cite{B4} don't seem to apply.
Furthermore, we study non-integrable modifications of the cubic NLS, among which is the Hartree Equation, with sufficiently regular convolution potential. For most of the equations obtained this way, we obtain better bounds than for the other equations, due to the fact that we can use \emph{higher modified energies}, as in work of the I-Team \cite{CKSTT5,CKSTT3}.
\end{abstract}


\maketitle

\section{Introduction.}

\subsection{Setup of the Problem.}

Given $k \in \mathbb{N}$ and $s \in \mathbb{R}$ with $s \geq 1$, let us consider the 1D defocusing periodic nonlinear Schr\"{o}dinger initial value problem:

\begin{equation}
\label{eq:NLS}
\begin{cases}
iu_t + \Delta u=|u|^{2k}u, x\in S^1, t \in \mathbb{R}\\
u(x,0)=\Phi(x) \in H^s(S^1).
\end{cases}
\end{equation}

The nonlinear Schr\"{o}dinger equation arises naturally in geometric optics and in Bose-Einstein Condensates \cite{Sch,SuSu}. In the latter context, the equation is obtained in the appropriate scaling limit as $N \rightarrow \infty$ of $N-body$ Bose systems. Heuristically, the power of the nonlinearity comes from how many particles can interact at once.
\vspace{3mm}

If we start with initial data $\Phi \in H^1(S^1)$, we obtain a global solution $u$ to (\ref{eq:NLS}) for which the following
quantities are conserved:

\begin{equation}
\label{eq:massconservation}
M(u(t)):=\int_{S^1} |u(x,t)|^2 dx\,\,(\mbox{\emph{Mass}}).
\end{equation}

\begin{equation}
\label{eq:energyconservation}
E(u(t)):=\frac{1}{2} \int_{S^1}|\nabla u(x,t)|^2 dx + \frac{1}{2k+2}\int_{S^1}|u(x,t)|^{2k+2}\,\,(\mbox{\emph{Energy}}).
\end{equation}

Here we are using that the problem is \emph{globally well-posed in $H^1$}
\cite{B,B3}. Hence energy and mass conservation imply:

\begin{equation}
\label{eq:H1uniformbound}
\|u(t)\|_{H^1}\leq C(\Phi), \forall t \in \mathbb{R}.
\end{equation}
Furthermore, $\|u(t)\|_{H^1}$ can be bounded by a continuous function of energy and mass.

\vspace{3mm}

We are interested in the problem of bounding $\|u(t)\|_{H^s}$ for $s>1$. In this case, we can in general no longer get a priori bounds
coming from conservation laws. This problem has physical significance since it quantifies the \emph{Low-to-High frequency cascade}, i.e. how much of the support of $|\widehat{u}|^2$ has shifted from the low to the high frequencies. Namely, we observe that the $H^s$ norms weigh the higher frequencies more, especially as $s$ becomes large. Hence, the growth of high Sobolev norms, gives us a quantitative estimate on the \emph{Low-to-High frequency cascade} \footnote{We observe that, from conservation of energy, not all of the support of $|\widehat{u}|^2$ can move to the high frequencies. If a low-to-high frequency cascade occurs, then a part of $|\widehat{u}|^2$ must concentrate near the low frequencies, to counterbalance a movement of $|\widehat{u}|^2$ towards the high frequencies. The growth of high Sobolev norms quantitatively describes the latter part of the process.}.  The phenomenon of such a cascade in a dispersive wave model was first studied in the 1960s, for instance in \cite{BN,Has,Zak68}.

\vspace{2mm}

As was noted in \cite{FadTak,ZM}, the equation $(\ref{eq:NLS})$ is completely integrable when $k=1$. Hence, if we start from smooth initial data, all the Sobolev norms of a solution will be uniformly bounded in time.
We consider several modifications of the cubic NLS in which we break the complete integrability.
The first modification we consider is the Hartree equation on $S^1$:

\begin{equation}
\label{eq:Hartree}
\begin{cases}
iu_t + \Delta u=(V*|u|^2)u,\,\,x\in S^1, t \in \mathbb{R}\\
u(x,0)=\Phi(x) \in H^s(S^1).
\end{cases}
\end{equation}

The assumptions that we have on $V$ are:

\begin{enumerate}
\item[(i)] $V \in L^1(S^1)$
\item[(ii)] $V \geq 0$
\item[(iii)] $V$ is even.
\end{enumerate}

The Hartree Equation arises naturally in the dynamics of large quantum systems. It occurs in the context of the Mean-Field limit of $N$-body dynamics when we take $V$ to be the interaction potential \cite{Sch}.
Under the latter assumptions on $V$, we will see in Section 4.1. that the Hartree equation also has global solutions with a priori control on the $H^1$ norm, so we can consider the same question as we did before.

\vspace{3mm}

The analogous setup holds for the following two modifications of the cubic NLS, namely for the modification:

\begin{equation}
\label{eq:potentialcubicnls}
\begin{cases}
iu_t + \Delta u=|u|^2u + \lambda u,\,\,x\in S^1, t \in \mathbb{R}\\
u(x,0)=\Phi(x) \in H^s(S^1).
\end{cases}
\end{equation}

\vspace{3mm}

Here, we are assuming:

\begin{enumerate}
\item[(i)] $\lambda \in C^{\infty}(S^1)$
\item[(ii)] $\lambda$ is real-valued.
\end{enumerate}

\vspace{3mm}

We also consider the modification:

\begin{equation}
\label{eq:inhomogeneouscubicnls}
\begin{cases}
iu_t + \Delta u=\lambda |u|^2u,\,\,x\in S^1, t \in \mathbb{R}\\
u(x,0)=\Phi(x) \in H^s(S^1).
\end{cases}
\end{equation}

\vspace{2mm}

Here, the inhomogeneity $\lambda=\lambda(x)$ satisfies:

\begin{enumerate}
\item[(i)] $\lambda \in C^{\infty}(S^1)$
\item[(ii)] $\lambda \geq 0$.
\end{enumerate}

\subsection{Statement of the Main Results.}

Given a real number $r$, we denote by $r+$ the number $r+\epsilon$, where we take $0<\epsilon \ll 1$.
The number $r-$ is defined analogously as $r-\epsilon$.
With this notation, the results that we prove are:

\begin{theorem}
\label{Theorem 1}
Let $k\geq 2$ be an integer and let $s\geq 1$ be a real number.  Let $u$ be a global solution to
(\ref{eq:NLS}). Then, there exists a continuous function $C$, depending on $(s,k,E(\Phi),M(\Phi))$ such that, for all $t \in \mathbb{R}:$

\begin{equation}
\label{eq:nlsbound2}
\|u(t)\|_{H^s}\leq
C(s,k,E(\Phi),M(\Phi))(1+|t|)^{2s+}\|\Phi\|_{H^s}.
\end{equation}

\end{theorem}

For the modifications of the cubic NLS, we can prove the following results:

\begin{theorem}
\label{Theorem 2}

Let $s\geq 1$ and let $u$ be a global solution of $(\ref{eq:Hartree})$. Then, there exists a function $C$ as above, such that for all $t \in \mathbb{R}:$

\begin{equation}
\label{eq:hartreebound}
\|u(t)\|_{H^s}\leq
C(1+|t|)^{\frac{1}{2}s+}\|\Phi\|_{H^s}.
\end{equation}

\end{theorem}

\begin{theorem}
\label{Theorem 3}

Let $s\geq 1$ and let $u$ be a global solution of $(\ref{eq:potentialcubicnls})$. Then, there exists a function $C$ as above, such that for all $t \in \mathbb{R}:$

\begin{equation}
\label{eq:potentialcubicnlsbound}
\|u(t)\|_{H^s}\leq
C(1+|t|)^{s+}\|\Phi\|_{H^s}.
\end{equation}

\end{theorem}

\begin{theorem}
\label{Theorem 4}

Let $s\geq 1$ and let $u$ be a global solution of $(\ref{eq:inhomogeneouscubicnls})$. Then, there exists a function $C$ as above, such that for all $t \in \mathbb{R}:$

\begin{equation}
\label{eq:inhomogeneouscubicnlsbound}
\|u(t)\|_{H^s}\leq
C(1+|t|)^{2s+}\|\Phi\|_{H^s}.
\end{equation}

\end{theorem}

\vspace{3mm}

It makes sense to consider the case $k=1$ in Theorem \ref{Theorem 1}, as long as we are taking $s$ which is not an integer, and if we are assuming only $\Phi \in H^s(S^1)$. It turns out that we can get a better bound, which is the same as the one obtained for $(\ref{eq:Hartree})$. This result will be clear from the proof of Theorem \ref{Theorem 2}. This, however, doesn't allow us to recover the uniform bounds on the integral Sobolev norms of a solution, as we observed, up to a loss of $t^{0+}$ in the non-periodic case in \cite{SoSt1}. The question of bounding the growth of fractional Sobolev norms of solutions to the 1D periodic and non-periodic cubic NLS was posed on \cite{DW}.

\vspace{3mm}

Analogous results hold for focusing-type equations, except that then we need to consider initial data which is sufficiently small in an appropriate norm. As we will see in the proof, the only reason why we are looking at defocusing equations is that we have global existence in $H^1$, and the a priori bound on the $H^1$ norm, as is given by (\ref{eq:H1uniformbound}).

\vspace{3mm}

We can obtain the same conclusion for the defocusing variant of $(\ref{eq:NLS})$ if $\|\Phi\|_{H^1}$ is sufficiently small. On the other hand, in the case of the Hartree equation $(\ref{eq:Hartree})$, we can change the second assumption on $V$ to just assume that $V$ is real-valued, as long as we suppose that $\|\Phi\|_{L^2}$ is sufficiently small. For such initial data, the conclusion of Theorem \ref{Theorem 2} will still hold. Under an analogous $L^2$-smallness assumption on the initial data, we can consider $(\ref{eq:potentialcubicnls})$ with focusing nonlinearity, and $(\ref{eq:inhomogeneouscubicnls})$ with $\lambda$ which is assumed to be real-valued, but not necessarily non-negative. The conclusions of Theorem \ref{Theorem 3} and Theorem \ref{Theorem 4} will still hold then. We will henceforth consider only the defocusing-type equations.

\subsection{Previously known results.}

Suppose that $u$ is a solution of $(\ref{eq:NLS})$.
One can immediately obtain exponential bounds on the growth of Sobolev norms by iterating the local well-posedness scheme. The main reason is that the increment time coming from local well-posedness is determined by the conserved quantities of the equation. More precisely, one recalls from \cite{B,B3,Tao} that there exist $\delta,C>0$ depending only on the initial data such that for all times $t_0$:

\begin{equation}
\label{eq:exponentialiteration}
\|u(t_0+\delta)\|_{H^s} \leq C \|u(t_0)\|_{H^s}.
\end{equation}

We iterate (\ref{eq:exponentialiteration}) to obtain the exponential bound:

\begin{equation}
\label{eq:exponentialbound}
\|u(t)\|_{H^s}\lesssim_{s,\Phi} e^{A|t|}.
\end{equation}

It is, however, possible to obtain polynomial bounds. This was achieved for other nonlinear Schr\"{o}dinger equations
in \cite{B2,CatW,S,S2,Z}. These arguments can also be adapted to (\ref{eq:NLS}).
The main idea in these papers was to modify (\ref{eq:exponentialiteration}) to obtain an improved iteration bound by which there exists a constant $r \in (0,1)$ depending on $k,s$ and
$\delta,C>0$ depending also on the initial data such that for all times $t_0$:

\begin{equation}
\label{eq:polynomialiteration}
\|u(t_0+\delta)\|_{H^s}^2 \leq \|u(t_0)\|_{H^s}^2 + C \|u(t_0)\|_{H^s}^{2-r}.
\end{equation}

In \cite{B2}, (\ref{eq:polynomialiteration}) is proved by the \emph{Fourier multiplier method},  whereas in \cite{S,S2}, this bound is proved by using fine multilinear estimates.
The key to the latter approach was in the use of smoothing estimates similar to those used in \cite{KPV3}. A slightly different approach, based on the analysis from \cite{BGT}, is used to obtain the same iteration bound in \cite{CatW,Z}.

One can show that (\ref{eq:polynomialiteration}) implies:

\begin{equation}
\label{eq:firstpolynomialbound}
\|u(t)\|_{H^s}\lesssim_{s,\Phi} (1+|t|)^{\frac{1}{r}}.
\end{equation}

Other places where the idea of frequency decomposition was used to estimate the growth of Sobolev norms are  \cite{B5,B6,W}. Here, the authors are considering the periodic linear Schr\"{o}dinger Equation with real potential $V=V(x,t)$:

\begin{equation}
\label{eq:potentialequation}
i u_t +\Delta u= V u.
\end{equation}

Under rather restrictive smoothness assumptions on $V$ (for instance, in \cite{B5}, $V$ is taken to be jointly smooth in $x$ and $t$ with uniformly
bounded partial derivatives with respect to both of the variables), it is shown that solutions to (\ref{eq:potentialequation}) satisfy for all $\epsilon>0$ and all $t \in \mathbb{R}$:

\begin{equation}
\label{eq:tepsilon}
\|u(t)\|_{H^s}\lesssim_{s,\Phi,\epsilon} (1+|t|)^{\epsilon}
\end{equation}
in \cite{B5}, and, for some $r>0$
\begin{equation}
\label{eq:tepsilon'}
\|u(t)\|_{H^s}\lesssim_{s,\Phi} log(1+|t|)^r
\end{equation}
in \cite{B6,W}. The latter result requires even stronger assumptions on $V$.

\vspace{2mm}

The idea of the proof of (\ref{eq:tepsilon}),(\ref{eq:tepsilon'}) is to reduce the problem to one that is periodic in time and then to use localization of eigenfunctions of a certain linear differential operator together with separation properties of the eigenvalues of the Laplace operator on $S^1$. These separation properties can be deduced by elementary means on $S^1$. In \cite{B5}, the bound $(\ref{eq:tepsilon})$ is also proved on $S^d$, for $d\geq 2$. In this case, the separation properties are proved by a more sophisticated number theoretic argument.

Let us also note that recently, a new proof of $(\ref{eq:tepsilon})$ was given in \cite{De}. The argument given in this paper is based on an iterative change of variable. In addition to recovering the result $(\ref{eq:tepsilon})$ on any $d$-dimensional torus, the same bound is proved for the linear Schr\"{o}dinger equation on any Zoll manifold, i.e. on any compact manifold whose geodesic flow is periodic.

Arguing as in \cite{S}, one can obtain polynomial bounds for solutions to (\ref{eq:NLS}).
There doesn't seem to be an easy way to substitute $V=|u|^{2k}$ into (\ref{eq:potentialequation}) and bootstrap these polynomial bounds by applying the technique from \cite{B5} to obtain better bounds. The reason is that the reduction to the problem which is periodic in time doesn't work as soon as one has some growth in time of a fixed finite number of Sobolev norms.

The problem of Sobolev norm growth was also recently studied in \cite{CKSTT6}, but in the sense of bounding the growth from below. In this paper, the authors exhibit the existence of smooth solutions of the cubic defocusing nonlinear Schr\"{o}dinger equation on $\mathbb{T}^2$ whose $H^s$ norm is arbitrarily small at time zero and is arbitrarily large at some large finite time.

Let us note that the same method that was used in our paper can be applied in the context of the Hartree equation on $\mathbb{R},\mathbb{T}^2$, and $\mathbb{R}^2$. Furthermore, we can improve bounds for the cubic NLS on $\mathbb{R}^2$, which were previously obtained in \cite{CDKS}. Finally, we can use an additional modification of the argument and apply it to the cubic NLS on $\mathbb{R}$. For the cubic NLS on $\mathbb{R}$, which is completely integrable, we derive bounds that recover uniform bounds on integral Sobolev norms, up to a factor of $t^{0+}$. These results will be presented in our forthcoming papers \cite{SoSt1,SoSt2}.

\subsection{Techniques of the Proof.}

The main idea of the proof of Theorem \ref{Theorem 1} is to obtain a good iteration bound. We will use the idea, used in \cite{B5,B6,W}, of estimating the high-frequency part of the solution. Let $E^1$ denote an operator which, after an appropriate rescaling, essentially adds the square $L^2$ norm of the low frequency part and the square $H^s$ norm of the high frequency part of a function. The threshold between the low and high frequencies is the parameter $N>1$. With this definition, we show that there exist $\delta,C>0$ depending only on $\Phi$ such that for all times $t_0$:

\begin{equation}
\label{eq:Diterationboundkgeq2}
E^1(u(t_0+\delta))\leq(1+\frac{C}{N^{\frac{1}{2}-}})E^1(u(t_0)).
\end{equation}

\vspace{2mm}

One observes that (\ref{eq:Diterationboundkgeq2}) is more similar to (\ref{eq:exponentialiteration}) than
to (\ref{eq:polynomialiteration}). The key fact to observe is that, due to the present decay factor, iteration of (\ref{eq:Diterationboundkgeq2})
$O(N^{\frac{1}{2}-})$ times doesn't cause exponential growth in $E^1(u(t))$, as it did for $\|u(t)\|_{H^s}$ in (\ref{eq:exponentialbound}).

\vspace{2mm}

The crucial point is to obtain the decay factor in (\ref{eq:Diterationboundkgeq2}). The reason why this is difficult is that we are working in the periodic setting in which we don't have the improved bilinear Strichartz estimates proved in \cite{B7,CKSTT}. In \cite{CKSTT3}, one could fix this problem by rescaling the circle to add more dispersion and reproving the estimates in the rescaled setting. Finally, one could scale back to the original circle, keeping in mind the relationship between the scaling parameter, the time interval on which one is working, and the threshold between the ``high'' and the ``low'' frequencies. This approach is unsuccessful in our setting since it is impossible to scale back, because the time on which we can obtain nontrivial bounds tends to zero as the rescaling factor tends to infinity \footnote{We note that this is not the same phenomenon that occurs for super-critical equations. The reason why the rescaling here doesn't give the result is that there are too many constraints on all of the parameters.}.

\vspace{2mm}

We take:

\begin{equation}
\label{eq:definitionofE1}
E^1(f):=\|\mathcal{D}f\|_{L^2}^2.
\end{equation}

\vspace{2mm}

Here $\mathcal{D}$ is an appropriate Fourier multiplier. In this paper, we take the $\mathcal{D}$-operator to be an \emph{upside-down I-operator}, corresponding to high regularities. The idea of using an \emph{upside-down I-operator} first appeared in \cite{CKSTT2}, but in the low regularity context. The purpose of such an operator is to control the evolution of a Sobolev norm which is higher than the norm associated to a particular conserved quantity. This is the opposite from the standard \emph{I-operator}, which was first developed in \cite{CKSTT,CKSTT5,CKSTT2,CKSTT3,CKSTT4}.

\vspace{2mm}

We then want to estimate:

\begin{equation}
\label{eq:derivativeofnormsquared}
\int_I \frac{d}{dt}\|\mathcal{D}u(t)\|_{L^2}^2 dt.
\end{equation}
over an appropriate time interval $I$ whose length depends only on the initial data.

Similarly as in the papers by the I-Team, the multiplier $\theta$ corresponding to the operator $\mathcal{D}$ is not a rough cut-off. Hence, in frequency regimes where certain cancelation occurs, we can symmetrize the expression \label{eq:derivativeofnormsquared} and see how the cancelation manifests itself in terms of $\theta$, as in \cite{CKSTT2}. If there is no cancelation in the symmetrized expression, we need to look at the spacetime Fourier transform. Arguing as in \cite{BGT,Z}, we decompose our solution into components whose spacetime Fourier transform is localized in the parabolic region $\langle \tau + n^2 \rangle \sim L$. In each of the cases, we obtain a satisfactory decay factor. The mentioned symmetrizations and localizations allow us to compensate for the absence of an \emph{improved Strichartz estimate}.

\vspace{2mm}

It is important to note that the iteration bound we obtain in (\ref{eq:Diterationboundkgeq2}) doesn't depend on the nonlinearity. The reason for this is that, by Sobolev embedding, one has: $X^{\frac{1}{2}+,\frac{1}{2}+}\hookrightarrow L^{\infty}_{t,x}$. On the other hand, if one uses (\ref{eq:polynomialiteration}), the bounds one obtains become progressively worse as we increase $k$ since $r$ becomes smaller as $k$ grows. The use of (\ref{eq:Diterationboundkgeq2}) gives us better bounds than the use of (\ref{eq:polynomialiteration}) already in the case $k=2$. We can show that (\ref{eq:polynomialiteration}) holds for $r=\frac{1}{18(s-1)}-$, from where we deduce the bound:

\begin{equation}
\label{eq:firstbound}
\|u(t)\|_{H^s}\lesssim_{s,\Phi} (1+|t|)^{18(s-1)+}.
\end{equation}
This is a worse bound than (\ref{eq:nlsbound2}).

It should be noted that a better bound for the quintic equation than the one given by Theorem \ref{Theorem 1} was noted by Bourgain in the appendix of \cite{B4}. The techniques sketched out in this paper are completely different and come from dynamical systems. In \cite{B4}, the author uses an appropriate normal form which reduces the nonlinearity to its essential part, i.e. to the frequency configurations which are close to being resonant. The result in \cite{B4} is mentioned only for the quintic equation. As we will note, due to the fact that it uses Besov-type spaces, which don't embed into $L^{\infty}_{t,x}$, we can't seem to modify this method to apply it to (\ref{eq:NLS}) with $k>2$.

\vspace{2mm}

The proofs of Theorems \ref{Theorem 2}, \ref{Theorem 3}, and \ref{Theorem 4} are based on similar techniques. For $(\ref{eq:Hartree})$, and $(\ref{eq:potentialcubicnls})$, we can use the method of \emph{higher modified energies} as in \cite{CKSTT5,CKSTT2}, i.e. we can find an approximation $E^2(u)$ of $\|u\|_{H^s}^2$ that varies in time slower than $E^1(u)$. $E^2(u)$ is obtained as a multilinear correction of $E^1(u)$. We deduce better iteration bounds than the one in $(\ref{eq:Diterationboundkgeq2})$, from which the results in Theorem \ref{Theorem 2} and Theorem \ref{Theorem 3} follow. The technique of higher modified energies doesn't seem to work for $(\ref{eq:inhomogeneouscubicnls})$. Heuristically, this means that adding an inhomogeneity as in $(\ref{eq:inhomogeneouscubicnls})$ breaks the integrability of the cubic NLS more than adding the convolution potential in $(\ref{eq:Hartree})$, or adding the external potential in $(\ref{eq:potentialcubicnls})$. Let us note that the techniques  sketched in \cite{B4} could in principle be applied to $(\ref{eq:Hartree})$ to obtain the same result. The techniques from \cite{B4} don't seem to apply to $(\ref{eq:potentialcubicnls})$ and $(\ref{eq:inhomogeneouscubicnls})$.

\vspace{3mm}

\textbf{Organization of the paper:}

\vspace{2mm}

The paper is organized as follows: In Section 2, we define the notation we will be using, and we recall some facts from Harmonic Analysis. In Section 3, we prove Theorem \ref{Theorem 1}. In Section 4, we prove Theorems \ref{Theorem 2}, \ref{Theorem 3}, and \ref{Theorem 4}. Two useful facts, one about localization in $X^{s,b}$ spaces, and another about local well-posedness bounds, are proved in Appendix A and in Appendix B.

\vspace{3mm}

\textbf{Acknowledgements:}

\vspace{2mm}

The author would like to thank his Advisor, Gigliola Staffilani, for suggesting this problem, and for her help and encouragement. He would also like to thank Hans Christianson and Antti Knowles for several useful comments and discussions. Furthermore, the author is grateful to the referee for their careful reading and constructive comments. This work was partially supported by the NSF Grant DMS 0602678.

\vspace{3mm}

\section{Some notation.}

\vspace{2mm}

In our paper, we denote by $A \lesssim B$ an estimate of the form $A \leq CB$, for some $C>0$. If $C$ depends on $d$, we write $A \lesssim_d B$. We also write the latter condition as $C=C(d)$.

\vspace{2mm}

We are taking the convention for the Fourier transform on $S^1$ to be:

$$\widehat{f}(n):=\int_{S^1} f(x) e^{- i n x} dx.$$

\vspace{2mm}

On $S^1 \times \mathbb{R}$, we define the spacetime Fourier transform by:

$$\widetilde{u}(n,\tau):=\int_{S^1} \int_{\mathbb{R}} u(x,t) e^{-i n x -i t \tau} dt dx.$$

\vspace{2mm}

Let us take the following convention for the Japanese bracket $\langle \cdot \rangle$ :

$$\langle n \rangle: =\sqrt{1+|n|^2}.$$

\vspace{1mm}

Let us recall that we are working in Sobolev Spaces $H^s=H^s(S^1)$ on the circle, whose norms are defined
for $s \in \mathbb{R}$ by:

$$\|f\|_{H^s}:=\big(\sum_{n}|\widehat{f}(n)|^2 \langle n \rangle^{2s}\big)^{\frac{1}{2}}. $$

where $f:S^1 \rightarrow \mathbb{C}.$

\vspace{1mm}

Let us define $H^{\infty}(S^1):=\bigcap_{s>0}H^s(S^1).$

\vspace{1mm}

An important tool in our work will also be $X^{s,b}$ spaces. We recall that in our context, these spaces come from
the norm defined for $s,b \in \mathbb{R}$:

$$\|u\|_{X^{s,b}}:=\big(\sum_n \int |\widetilde{u}(n,\tau)|^2 \langle n \rangle^{2s} \langle \tau + n^2 \rangle^{2b} d \tau \big)^{\frac{1}{2}}.$$

where $u:S^1 \times \mathbb{R} \rightarrow \mathbb{C}$.

\vspace{2mm}

There are some key facts one should note about $X^{s,b}$ spaces:

By definition, one has:

\begin{equation}
\label{eq:L2tL2x}
\|u\|_{L^2_tL^2_x}=\|u\|_{X^{0,0}}.
\end{equation}

Using Sobolev embedding, one obtains:

\begin{equation}
\label{eq:Linftytx1}
\|u\|_{L^{\infty}_tL^{\infty}_x} \lesssim \|u\|_{X^{\frac{1}{2}+,\frac{1}{2}+}}
\end{equation}

\vspace{1mm}

and:

\vspace{1mm}

\begin{equation}
\label{eq:Linftytx2}
\|u\|_{L^{\infty}_tL^2_x} \lesssim \|u\|_{X^{0,\frac{1}{2}+}}.
\end{equation}

Interpolating between (\ref{eq:L2tL2x}) and (\ref{eq:Linftytx2}), it follows that:

\begin{equation}
\label{eq:L4tL2x}
\|u\|_{L^4_tL^2_x}
\lesssim \|u\|_{X^{0,\frac{1}{4}+}}.
\end{equation}

Two more key $X^{s,b}$ space estimates are the two following Strichartz inequalities:

\begin{equation}
\label{eq:strichartztorus}
\|u\|_{L^4_{t,x}}\lesssim
\|u\|_{X^{0,\frac{3}{8}}}.
\end{equation}

\begin{equation}
\label{eq:strichartztorus2}
\|u\|_{L^6_{t,x}}\lesssim
\|u\|_{X^{0+,\frac{1}{2}+}}.
\end{equation}

For the proof of (\ref{eq:strichartztorus}), one should consult Proposition 2.13. in \cite{Tao}.
A proof of (\ref{eq:strichartztorus2}) can be found in \cite{G}.
It is crucial to observe that both estimates are global in time.

\vspace{2mm}

We give some useful notation for multilinear expressions, which can also be found in \cite{CKSTT}. For $r \geq 2$, an even
integer, we define the hyperplane:

$$\Gamma_r:=\{(n_1,\ldots,n_r)\in \mathbb{Z}^r: n_1+\cdots n_r=0\},$$
endowed with the measure $\delta(n_1+\cdots + n_r)$.

\vspace{2mm}

If we are given a function $M_r=M_r(n_1,\cdots,n_r)$ on $\Gamma_r$, i.e. an
\emph{r-multiplier}, we define the \emph{r-linear functional} $\lambda_r(M_r;f_1,\ldots,f_r)$ by:

$$\lambda_r(M_r;f_1,\ldots,f_r):=\int_{\Gamma_r}M_r(n_1,\ldots,n_r)\prod_{j=1}^r \widehat{f_j}(n_j).$$

As in \cite{CKSTT}, we adopt the notation:

\begin{equation}
\label{eq:lambdar}
\lambda_r(M_r;f):=\lambda_r(M_r;f,\bar{f},\ldots,f,\bar{f}).
\end{equation}

We will also sometimes write $n_{i,j}$ for $n_i+n_j$ and $\theta_i$ for $\theta(n_i)$.

\subsection{An observation about $X^{s,b}$ spaces.}

Throughout the paper, we will need to consider quantities such as
$\|\chi_{[c,d]}(t)f\|_{X^{s,b}}$. We will show the following bound:

\begin{lemma}
\label{Lemma 2.1}

If $b \in (0,\frac{1}{2})$ and $s \in \mathbb{R}$, then, for $c,d \in \mathbb{R}$ such that $c<d$, one has:

\begin{equation}
\label{eq:Xsb localization}
\|\chi_{[c,d]}(t)u\|_{X^{s,b}}\lesssim \|u\|_{X^{s,b+}}
\end{equation}
where the implicit constant doesn't depend on $c,d$.
\end{lemma}

A similar fact was proved in \cite{CKSTT4}, but in slightly different spaces.
Furthermore, let us mention that a stronger statement was mentioned in a remark after Proposition 32 in \cite{CafE}. For completeness, we present the proof of Lemma 2.1. in Appendix A.

\vspace{2mm}

From Lemma \ref{Lemma 2.1}, we deduce that in particular:
\begin{corollary}
For $c,d$ as above, one has:
\begin{equation}
\label{eq:key localization}
\|\chi_{[c,d]}u\|_{X^{0,\frac{3}{8}}}\lesssim \|u\|_{X^{0,\frac{3}{8}+}}
\end{equation}
\end{corollary}
This fact will be used later on.

\section{Quintic and Higher Order NLS.}

In this section, we will define the \emph{upside-down I-operator} $\mathcal{D}$. In order to use this operator effectively, we need to prove appropriate local-in-time bounds.
Finally, we use symmetrization to get good estimates on the growth of $\|\mathcal{D}u(t)\|_{L^2}^2$.

\vspace{2mm}

Throughout the first three parts of the section, we will prove the claim in the case $k=2$, for simplicity of notation. Generalizations to higher nonlinearities are given in the fourth part of this section.

\subsection{Definition of the $\mathcal{D}$ operator.}

Suppose $N>1$ is given. Let $\theta: \mathbb{Z}\rightarrow \mathbb{R}$ be given by:

\begin{equation}
\label{eq:theta}
\theta(n) :=
\begin{cases}
  \big(\frac{|n|}
  {N}\big)^s\,,   \mbox{if }|n| \geq N\\
  \,1,\,  \mbox{if } |n| \leq N
\end{cases}
\end{equation}

Then, if $f:S^1 \rightarrow \mathbb{C}$, we define $\mathcal{D}f$ by:

\begin{equation}
\label{eq:D operator}
\widehat{\mathcal{D}f}(n):=\theta(n)\hat{f}(n).
\end{equation}

We observe that:

\begin{equation}
\label{eq:bound on D}
\|\mathcal{D}f\|_{L^2}\lesssim_s \|f\|_{H^s}\lesssim_s N^s \|\mathcal{D}f\|_{L^2}.
\end{equation}

Our goal is to then estimate $\|\mathcal{D}u(t)\|_{L^2}$ , from which we can estimate  $\|u(t)\|_{H^s}$ by $(\ref{eq:bound on D})$.

\subsection{A local-in-time estimate and an approximation lemma.}

From our proof, we will note the key role of good local-in-time and associated approximation results. Here, we collect the statements of these results, whose proofs we give in Appendix B. The first result we want to show is that there exist $\delta=\delta(s,E(\Phi),M(\Phi)),C=C(s,E(\Phi),M(\Phi))>0$, such that for all $t_0 \in \mathbb{R}$, there exists a globally defined function $v:S^1 \times \mathbb{R} \rightarrow \mathbb{C}$ such that:

\begin{equation}
\label{eq:properties of v1}
v|_{[t_0,t_0+\delta]}=u|_{[t_0,t_0+\delta]}.
\end{equation}

\begin{equation}
\label{eq:properties of v2}
\|v\|_{X^{1,\frac{1}{2}+}}\leq C(s,E(\Phi),M(\Phi))
\end{equation}

\begin{equation}
\label{eq:properties of v3}
\|\mathcal{D}v\|_{X^{0,\frac{1}{2}+}}\leq C(s,E(\Phi),M(\Phi)) \|\mathcal{D}u(t_0)\|_{L^2}.
\end{equation}
Moreover, $\delta$ and $C$ can be chosen to depend continuously on the energy and mass.

\vspace{3mm}

\begin{proposition}
\label{Proposition 3.1}
Given $t_0 \in \mathbb{R}$, there exists a globally defined function $v:S^1 \times \mathbb{R} \rightarrow \mathbb{C}$ satisfying the properties (\ref{eq:properties of v1}),(\ref{eq:properties of v2}),(\ref{eq:properties of v3}).
\end{proposition}

\vspace{3mm}

In the proof of Proposition \ref{Proposition 3.1}, we need to use a ``persistence of regularity'' argument, which relies on the following fact:

\begin{proposition}
\label{Proposition 3.2}
Let $R>0,s\geq 1, B:=\{v:\|v\|_{X^{s,b}}\leq R\}$. Then (B,d) is complete as
a metric space if we take:
\begin{equation}
\label{eq:persistence} d(v,w):=\|v-w\|_{X^{1,b}}.
\end{equation}
\end{proposition}
A related technical fact that we will need to use in the proof of the Theorem \ref{Theorem 1}, and later, in the proofs of the other Theorems is the following:

\begin{proposition}
\label{Proposition 3.3}(Approximation Lemma)

If $u$ satisfies:

\begin{equation}
\begin{cases}
i u_t + \Delta u=|u|^{2k}u,\\
u(x,0)=\Phi(x).
\end{cases}
\end{equation}

and if the sequence $(u^{(n)})$ satisfies:

\begin{equation}
\begin{cases}
i u^{(n)}_t + \Delta u^{(n)}=|u^{(n)}|^{2k}u^{(n)},\\
u^{(n)}(x,0)=\Phi_n(x).
\end{cases}
\end{equation}

where $\Phi_n \in C^{\infty}(S^1)$ and $\Phi_n \stackrel{H^s}{\longrightarrow} \Phi$, then, one has for all $t$:

$$u^{(n)}(t)\stackrel{H^s}{\longrightarrow} u(t).$$

\end{proposition}

\vspace{3mm}

The mentioned approximation Lemma allows us to work with smooth solutions and pass to the limit in the end.
Namely, we note that if we take initial data $\Phi_n$ as earlier, then $u^{(n)}(t)$ will belong to $H^{\infty}(S^1)$ for all $t$. On the other hand, by continuity of mass, energy, and the $H^s$ norm on $H^s$, it follows that:

$$M(\Phi_n) \rightarrow M(\Phi),\,E(\Phi_n) \rightarrow E(\Phi),\,\|\Phi_n\|_{H^s} \rightarrow \|\Phi\|_{H^s}.$$

Suppose that we knew that Theorem \ref{Theorem 1} were true in the case of smooth solutions. Then, it would follow that for all $t \in \mathbb{R}$:

$$\|u^{(n)}(t)\|_{H^s}\leq C(s,k,E(\Phi_n),M(\Phi_n)) (1+|t|)^{2s+}\|\Phi_n\|_{H^s},$$

The claim for $u$ would now follow by applying the continuity properties of $C$ and the approximation Lemma.

\vspace{3mm}

We will henceforth work with $\Phi \in C^{\infty}(S^1)$. This implies that $u(t) \in H^{\infty}(S^1)$ for all $t$.
The claimed result is then deduced from this special case by the approximation procedure given earlier.
As we will see, the analogue of Proposition \ref{Proposition 3.1} holds for $(\ref{eq:Hartree})$,$(\ref{eq:potentialcubicnls})$, and for $(\ref{eq:inhomogeneouscubicnls})$. A similar argument shows that for these equations, it suffices to consider the case when $\Phi \in C^{\infty}$. The advantage of working with smooth solutions is that all the formal calculations will then be well-defined.

\vspace{3mm}

\subsection{Control on the increment of $\|\mathcal{D}u(t)\|_{L^2}^2$.}

For $t \in [t_0,t_0+\delta]$, we can work with $\mathcal{D}v(t)$ instead of with $\mathcal{D}u(t)$, where $v$
is the object we had constructed earlier. By our smoothness assumption, we know $v(t) \in H^{\infty}(S^1)$.

\vspace{2mm}

Now, for $t \in [t_0,t_0+\delta]$, one has \footnote{We are using the fact that $v(t) \in H^{\infty}(S^1)$ in order to deduce that this quantity is finite!}:

$$\frac{d}{dt} \|\mathcal{D}v(t)\|_{L^2}^2=2Re\,\langle \mathcal{D}v_t,\mathcal{D}v \rangle=2Re\, \langle i\mathcal{D}\Delta v - i \mathcal{D}(v\bar{v}v\bar{v}v),\mathcal{D}v\rangle$$

Since  $Re\, \langle i\mathcal{D}\Delta v,\mathcal{D}v \rangle=0$, this expression equals:

$$=-2Re\,\langle i\mathcal{D}(v\bar{v}v\bar{v}v),\mathcal{D}v\rangle.$$

After an appropriate symmetrization, by using notation as in Section 2, and arguing as in \cite{CKSTT2}, we get that this expression equals:
$$\frac{1}{3}i \cdot \lambda_6((\theta(n_1))^2-(\theta(n_2))^2+(\theta(n_3))^2-(\theta(n_4))^2+(\theta(n_5))^2-(\theta(n_6))^2;v(t)).$$
Let us take:

$$M_6(n_1,n_2,n_3,n_4,n_5,n_6):=
(\theta(n_1))^2-(\theta(n_2))^2+(\theta(n_3))^2-(\theta(n_4))^2+(\theta(n_5))^2-(\theta(n_6))^2.$$

We now analyze:

$$\|\mathcal{D}u(t_0+\delta)\|_{L^2}^2-\|\mathcal{D}u(t_0)\|_{L^2}^2=\|\mathcal{D}v(t_0+\delta)\|_{L^2}^2-
\|\mathcal{D}v(t_0)\|_{L^2}^2=$$
\vspace{2mm}
$$=\int_{t_0}^{t_0+\delta} \frac{d}{dt}\,\,\|\mathcal{D}v(t)\|_{L^2}^2dt=$$
\vspace{2mm}
$$=\frac{1}{3}i\,\Big(\sum_{n_1+n_2+n_3+n_4+n_5+n_6=0}\int_{t_0}^{t_0+\delta} M_6(n_1,n_2,n_3,n_4,n_5,n_6)
\hat{v}(n_1)\hat{\bar{v}}(n_2)\hat{v}(n_3)\hat{\bar{v}}(n_4)\hat{v}(n_5)\hat{\bar{v}}(n_6)dt\Big)=$$
\vspace{2mm}
$$=\frac{1}{3}i\,\Big(\sum_{n_1-n_2+n_3-n_4+n_5-n_6=0}$$
\vspace{2mm}
\begin{equation}
\label{eq:2zvijezde}
\int_{t_0}^{t_0+\delta} M_6(n_1,n_2,n_3,n_4,n_5,n_6)\hat{v}(n_1)\overline{\hat{v}(n_2)}\hat{v}(n_3)\overline{\hat{v}(n_4)}\hat{v}(n_5)
\overline{\hat{v}(n_6)}dt\Big)\,\,=:I
\end{equation}

\vspace{2mm}

We want to prove an appropriate decay bound on the increment. The bound that we will prove is:

\vspace{2mm}

\begin{lemma}(Iteration Bound)
\label{Bigstar1}
For all $t_0 \in \mathbb{R}$, one has: $$\big|\|\mathcal{D}v(t_0+\delta)\|_{L^2}^2-\|\mathcal{D}v(t_0)\|_{L^2}^2 \big|\lesssim \frac{1}{N^{\frac{1}{2}-}}\|\mathcal{D}v(t_0)\|_{L^2}^2.$$
\end{lemma}

From the proof, it will follow that the implied constant depends only on $(s,Energy,Mass)$, and hence is uniform in time. We call this constant $C=C(s,Energy,Mass)>0$. In fact, by construction, it will follow that all the implied constants we obtain will depend continuously on energy and mass, and hence will be continuous functions of $\Phi$ w.r.t to the $H^1$ norm. For brevity, we will suppress this fact in our further arguments.

\vspace{2mm}

Let us first observe how Lemma \ref{Bigstar1} implies Theorem \ref{Theorem 1} for $k=2$.
From Lemma \ref{Bigstar1}, and for the $C$ constructed earlier, it follows that:

$$\|\mathcal{D}u(\delta)\|_{L^2}^2 \leq (1+\frac{C}{N^{\frac{1}{2}-}})\|\mathcal{D}\Phi\|_{L^2}^2$$
The same $C$ satisfies:

\begin{equation}
\label{eq:Bigstar}
\forall t_0 \in \mathbb{R}\,, \|\mathcal{D}u(t_0+\delta)\|_{L^2}^2 \leq (1+\frac{C}{N^{\frac{1}{2}-}})\|\mathcal{D}u(t_0)\|_{L^2}^2
\end{equation}
Using (\ref{eq:Bigstar}) iteratively, we obtain that \footnote{
Strictly speaking, we are using $(\ref{eq:properties of v3})$ to deduce that we can get the bound for all such times, and not just those which are a multiple of $\delta$.} $\forall T>1:$ \,

$$\|\mathcal{D}u(T)\|_{L^2}^2 \leq (1+\frac{C}{N^{\frac{1}{2}-}})^{\lceil \frac{T}{\delta} \rceil} \|\mathcal{D}\Phi\|_{L^2}^2 $$
i.e. there exists $\alpha=\alpha(s,Energy,Mass)>0$ s.t. for all $T>1$, one has:

\begin{equation}
\label{eq:Iteration1}
\|\mathcal{D}u(T)\|_{L^2}^2 \leq (1+\frac{C}{N^{\frac{1}{2}-}})^{\alpha T} \|\mathcal{D}\Phi\|_{L^2}^2
\end{equation}
For $\lambda_1,\lambda_2>0$, we know:

\begin{equation}
\label{eq:limit}
\lim_{x \to +\infty} \Big(1+\frac{1}{\lambda_1 x}\Big)^{\lambda_2 x}=e^{\frac{\lambda_2}{\lambda_1}}<\infty
\end{equation}
By using (\ref{eq:Iteration1}) and (\ref{eq:limit}), we can take:

\begin{equation}
\label{eq:Iteration2}
T \sim N^{\frac{1}{2}-}
\end{equation}

\begin{equation}
\label{eq:Iteration3}
\|\mathcal{D}u(T)\|_{L^2}\lesssim \|\mathcal{D}\Phi\|_{L^2}
\end{equation}
Recalling (\ref{eq:bound on D}), and using (\ref{eq:Iteration3}), (\ref{eq:Iteration2}), and the fact that $T>1$, we obtain:

$$\|u(T)\|_{H^s} \lesssim N^s \|\mathcal{D}u(T)\|_{L^2} \lesssim N^s \|\mathcal{D}\Phi\|_{L^2}\lesssim N^s \|\Phi\|_{H^s} $$
\vspace{1mm}
\begin{equation}
\label{eq:QB1}
\lesssim T^{2s+} \|\Phi\|_{H^s}\lesssim (1+T)^{2s+} \|\Phi\|_{H^s}.
\end{equation}
Since for times $t\in [0,1]$, we get the bound of Theorem \ref{Theorem 1} just by iterating the local well-posedness construction, the claim for these times follows immediately. Combining this observation, (\ref{eq:QB1}), recalling the approximation result, and using time-reversibility, we obtain that for all $s \geq 1$, there exists $C=C(s,Energy,Mass)$ such that for all $t \in \mathbb{R}$:
\begin{equation}
\label{eq:QuinticBound}
\|u(t)\|_{H^s}\leq C(1+|t|)^{2s+}\|u(0)\|_{H^s}.
\end{equation}
Moreover, $C$ depends continuously on energy and mass.
This proves Theorem \ref{Theorem 1} when $k=2$. $\Box$

\vspace{3mm}

We now turn to the proof of Lemma \ref{Bigstar1}.

\begin{proof}
Let us consider WLOG the case when $t_0=0$. The general case follows by time translation and by the fact that all of our implied constants are independent of time.
The idea is to localize the factors of $v$ into dyadic annuli in frequency dual to $x$, i.e. to perform the
Littlewood-Paley decomposition. Namely, for each $j$ such that $n_j \neq 0$, we find a dyadic integer $N_j$ such that $|n_j|\sim N_j$. If $n_j=0$, we take the corresponding $N_j$ to be equal to $1$.

\vspace{3mm}

We let $v_{N_j}$ denote the function obtained from $v$ by localizing in frequency
to the dyadic annulus $|n| \sim N_j$. Let $|n_{a}|,|n_{b}|$ denote the largest two elements of the set $\{|n_1|,|n_2|,|n_3|,|n_4|,|n_5|,|n_6|\}.$

\vspace{3mm}

In our analysis of (\ref{eq:2zvijezde}), we have to consider two Big Cases:

\vspace{3mm}

\textbf{$\diamondsuit$Big Case 1:} In the expression for $M_6$, $(\theta(n_a))^2$ and $(\theta(n_b))^2$ appear with the opposite sign.

\vspace{2mm}

\textbf{$\diamondsuit$Big Case 2:} In the expression for $M_6$, $(\theta(n_a))^2$ and $(\theta(n_b))^2$ appear with the same sign.

\vspace{2mm}

As we will see, the ways in which we bound the contributions to $(\ref{eq:2zvijezde})$ coming from the two Big Cases are quite different.

\vspace{2mm}

Let $I^{(1)}$ denote the contribution coming to $I$ (as defined in (\ref{eq:2zvijezde})) from Big Case 1, and let $I^{(2)}$ denote the contribution coming from Big Case 2.

\vspace{4mm}

\textbf{Big Case 1:} We can assume WLOG that $|n_a|=|n_1|$, and $|n_b|=|n_2|$. In the proof of Big Case 1, we will see that the order of the other four frequencies in absolute value doesn't matter. Namely, the order of the four lower frequencies won't affect any of the multiplier bounds (which depend only on $|n_1|$ and $|n_2|$), and the estimates that we will use on the factors of $v$ corresponding to these four frequencies will not depend on complex conjugates.
Hence, it suffices to consider WLOG the case when:

\begin{equation}
\label{eq:localizationn}
|n_1| \geq |n_2| \geq |n_3| \geq |n_4| \geq |n_5| \geq |n_6|.
\end{equation}

We observe that, in this contribution, the $N_j$ satisfy:

\begin{equation}
\label{eq:localizationN1}
N_1 \gtrsim N_2 \gtrsim N_3 \gtrsim N_4 \gtrsim N_5 \gtrsim N_6.
\end{equation}

By definition of $\theta$, we observe that $$M_6(n_1,n_2,n_3,n_4,n_5,n_6)=0\,\,\mbox{if}\,\,|n_1|,|n_2|,|n_3|,|n_4|,|n_5|,|n_6|\leq N.$$

Hence, by construction of $|n_1|$, one has $|n_1|> N$ so we obtain the additional localization:

\begin{equation}
\label{eq:localizationN2}
N_1 \gtrsim N.
\end{equation}

Finally, since $n_1-n_2+n_3-n_4+n_5-n_6=0$, (\ref{eq:localizationn}) and the triangle inequality imply that $|n_1| \sim |n_2|$.

From this fact, we can deduce the localization:

\begin{equation}
\label{eq:localizationN3}
N_1 \sim N_2.
\end{equation}

The expression we wish to estimate is:

$$I_{N_1,N_2,N_3,N_4,N_5,N_6}:=$$
$$\sum_{n_1-n_2+n_3-n_4+n_5-n_6=0;|n_1|\geq \cdots \geq |n_6|}
\int_0^{\delta} M_6 \widehat{v_{N_1}}(n_1)\overline{\widehat{v_{N_2}}(n_2)}\widehat{v_{N_3}}(n_3)\overline{\widehat{v_{N_4}}(n_4)}
\widehat{v_{N_5}}(n_5)\overline{\widehat{v_{N_6}}(n_6)}dt.$$

\vspace{3mm}

Let $\widetilde{I}$ denote the contribution to $I$, as defined in (\ref{eq:2zvijezde}), coming from (\ref{eq:localizationn}). Then $\widetilde{I}$ satisfies:

$$|\widetilde{I}| \lesssim \sum_{N_j \,\mbox{satisfying}\, (\ref{eq:localizationN1}),(\ref{eq:localizationN2}),(\ref{eq:localizationN3})}
|I_{N_1,N_2,N_3,N_4,N_5,N_6}|.$$

\vspace{3mm}

Within Big Case 1, we consider two cases:

\vspace{3mm}

\textbf{$\diamond$Case 1:}$N_3,N_4,N_5,N_6\ll N_1^{\frac{1}{2}}.$

\vspace{2mm}

\textbf{$\diamond$Case 2:}$N_3 \gtrsim N_1^{\frac{1}{2}}.$

\vspace{3mm}

\textbf{Case 1:}

The key step in this case is the following bound on $M_6$, which comes from cancelation.

\begin{equation}
\label{eq:Bound on M_6}
M_6 = O(N_1^{-\frac{1}{2}}\theta(N_1)\theta(N_2)).
\end{equation}
Before we prove (\ref{eq:Bound on M_6}), let us see how it gives us a good bound. Assuming (\ref{eq:Bound on M_6}) for the moment, we observe that:

$$|I_{N_1,N_2,N_3,N_4,N_5,N_6}|=$$

\vspace{2mm}

$$=\Big|\sum_{n_1-n_2+n_3-n_4+n_5-n_6=0;|n_1|\geq \cdots \geq |n_6|}
\int_{\mathbb{R}} M_6 (\chi_{[0,\delta]}v_{N_1})\,\widehat{\,}\,(n_1)\overline{\widehat{v_{N_2}}(n_2)}\widehat{v_{N_3}}(n_3)\overline{\widehat{v_{N_4}}(n_4)}
\widehat{v_{N_5}}(n_5)\overline{\widehat{v_{N_6}}(n_6)}dt\Big|=$$

\vspace{2mm}

$$=\Big|\sum_{n_1-n_2+n_3-n_4+n_5-n_6=0;|n_1|\geq \cdots \geq |n_6|}\int_{\tau_1-\tau_2+\tau_3-\tau_4+\tau_5-\tau_6=0}$$
$$M_6 (\chi_{[0,\delta]}v_{N_1})\,\widetilde{\,}\,(n_1,\tau_1)\overline{\widetilde{v_{N_2}}(n_2,\tau_2)}\widetilde{v_{N_3}}(n_3,\tau_3)
\overline{\widetilde{v_{N_4}}(n_4,\tau_4)}\widetilde{v_{N_5}}(n_5,\tau_5)\overline{\widetilde{v_{N_6}}(n_6,\tau_6)}d\tau_j\Big|\lesssim$$

\vspace{2mm}

$$\lesssim N_1^{-\frac{1}{2}}\theta(N_1)\theta(N_2)\sum_{n_1-n_2+n_3-n_4+n_5-n_6=0;|n_1|\geq \cdots \geq |n_6|}\int_{\tau_1-\tau_2+\tau_3-\tau_4+\tau_5-\tau_6=0}$$
$$\{|(\chi_{[0,\delta]}v_{N_1})\,\widetilde{\,}\,(n_1,\tau_1)||\widetilde{v_{N_2}}(n_2,\tau_2)||\widetilde{v_{N_3}}(n_3,\tau_3)|
|\widetilde{v_{N_4}}(n_4,\tau_4)||\widetilde{v_{N_5}}(n_5,\tau_5)||\widetilde{v_{N_6}}(n_6,\tau_6)|\}d\tau_j\leq$$

\vspace{2mm}

Since the integrand is non-negative, we can eliminate the restriction in the sum that $|n_1|\geq \cdots |n_6|$, so the expression is:

\vspace{2mm}

$$\leq N_1^{-\frac{1}{2}}\theta(N_1)\theta(N_2)\sum_{n_1-n_2+n_3-n_4+n_5-n_6=0}\int_{\tau_1-\tau_2+\tau_3-\tau_4+\tau_5-\tau_6=0}$$
$$\{|(\chi_{[0,\delta]}v_{N_1})\,\widetilde{\,}\,(n_1,\tau_1)||\widetilde{v_{N_2}}(n_2,\tau_2)||\widetilde{v_{N_3}}(n_3,\tau_3)|
|\widetilde{v_{N_4}}(n_4,\tau_4)||\widetilde{v_{N_5}}(n_5,\tau_5)||\widetilde{v_{N_6}}(n_6,\tau_6)|\}d\tau_j.
$$

Let us define:

\begin{equation}
\label{eq:eqF1}
F_1(x,t):=\sum_n \int_{\mathbb{R}} |(\chi_{[0,\delta]}v_{N_1})\,\widetilde{\,}\,(n_1,\tau_1)|e^{i(nx+t \tau)} d\tau.
\end{equation}

For $j=2,3,4,5,6$, we let:

\begin{equation}
\label{eq:eqFj}
F_j(x,t):=\sum_n \int_{\mathbb{R}} |\widetilde{v_{N_j}}(n,\tau)|e^{i(nx+t\tau)} d\tau.
\end{equation}

We now recall a fact from Fourier analysis. For simplicity, let us suppose that $f_1$,\ldots,$f_6$ are
functions on $\mathbb{R}$. Let us suppose that all $\widehat{f_j}$ are real-valued.

\vspace{2mm}

Then one has:

$$\int f_1 \overline{f_2} f_3 \overline{f_4} f_5 \overline{f_6} dx=$$
\begin{equation}
\label{eq:realfourier}
=\int_{\xi_1-\xi_2+\xi_3-\xi_4+\xi_5-\xi_6=0}\widehat{f_1}(\xi_1)\widehat{f_2}(\xi_2)\widehat{f_3}(\xi_3)
\widehat{f_4}(\xi_4)\widehat{f_5}(\xi_5)\widehat{f_6}(\xi_6)d \xi_j.
\end{equation}

\vspace{2mm}

Using the analogue of (\ref{eq:realfourier}) for the spacetime Fourier transform on $S^1 \times \mathbb{R}$, together with (\ref{eq:eqF1})
and (\ref{eq:eqFj}), and the previous bound we obtained on $|I_{N_1,N_2,N_3,N_4,N_5,N_6}|$, we deduce that:

$$|I_{N_1,N_2,N_3,N_4,N_5,N_6}|\lesssim N_1^{-\frac{1}{2}}\theta(N_1)\theta(N_2) \int_{\mathbb{R}}\int_{S^1}
F_1 \overline{F_2} F_3 \overline{F_4} F_5 \overline{F_6} dxdt=$$
$$=N_1^{-\frac{1}{2}}\theta(N_1)\theta(N_2) \Big|\int_{\mathbb{R}}\int_{S^1}
F_1 \overline{F_2} F_3 \overline{F_4} F_5 \overline{F_6} dxdt \Big|\leq$$
\vspace{1mm}
Which by H\"{o}lder's inequality is:
\vspace{1mm}
$$\leq N_1^{-\frac{1}{2}} \theta(N_1)\theta(N_2) \|F_1\|_{L^4_{t,x}}\|\overline{F_2}\|_{L^4_{t,x}}\|F_3\|_{L^4_{t,x}}\|\overline{F_4}\|_{L^4_{t,x}}\|F_5\|_{L^{\infty}_{t,x}}
\|\overline{F_6}\|_{L^{\infty}_{t,x}}=$$
\vspace{1mm}
$$=N_1^{-\frac{1}{2}} \theta(N_1)\theta(N_2) \|F_1\|_{L^4_{t,x}}\|F_2\|_{L^4_{t,x}}\|F_3\|_{L^4_{t,x}}\|F_4\|_{L^4_{t,x}}\|F_5\|_{L^{\infty}_{t,x}}
\|F_6\|_{L^{\infty}_{t,x}}$$
\vspace{2mm}
By using (\ref{eq:strichartztorus}) and (\ref{eq:Linftytx1}), this is:
\vspace{1mm}
$$\lesssim N_1^{-\frac{1}{2}}\theta(N_1)\theta(N_2)\|F_1\|_{X^{0,\frac{3}{8}}}\|F_2\|_{X^{0,\frac{3}{8}}}\|F_3\|_{X^{0,\frac{3}{8}}}\|F_4\|_{X^{0,\frac{3}{8}}}
\|F_5\|_{X^{\frac{1}{2}+,\frac{1}{2}+}}\|F_6\|_{X^{\frac{1}{2}+,\frac{1}{2}+}}=$$
\vspace{1mm}
$$=N_1^{-\frac{1}{2}}\theta(N_1)\theta(N_2)\|\chi_{[0,\delta]}v_{N_1}\|_{X^{0,\frac{3}{8}}}\|v_{N_2}\|_{X^{0,\frac{3}{8}}}\|v_{N_3}\|_{X^{0,\frac{3}{8}}}
\|v_{N_4}\|_{X^{0,\frac{3}{8}}}\|v_{N_5}\|_{X^{\frac{1}{2}+,\frac{1}{2}+}}\|v_{N_6}\|_{X^{\frac{1}{2}+,\frac{1}{2}+}}$$
\vspace{2mm}
By using (\ref{eq:key localization}) to bound the first factor, this expression is:
\vspace{1mm}
$$\lesssim N_1^{-\frac{1}{2}}\theta(N_1)\theta(N_2)\|v_{N_1}\|_{X^{0,\frac{3}{8}+}}\|v_{N_2}\|_{X^{0,\frac{3}{8}}}\|v_{N_3}\|_{X^{0,\frac{3}{8}}}
\|v_{N_4}\|_{X^{0,\frac{3}{8}}}\|v_{N_5}\|_{X^{\frac{1}{2}+,\frac{1}{2}+}}\|v_{N_6}\|_{X^{\frac{1}{2}+,\frac{1}{2}+}}\leq$$
\vspace{1mm}
$$\leq N_1^{-\frac{1}{2}}\theta(N_1)\theta(N_2)\|v_{N_1}\|_{X^{0,\frac{1}{2}+}}\|v_{N_2}\|_{X^{0,\frac{1}{2}+}}\|v_{N_3}\|_{X^{1,\frac{1}{2}+}}
\|v_{N_4}\|_{X^{1,\frac{1}{2}+}}\|v_{N_5}\|_{X^{1,\frac{1}{2}+}}\|v_{N_6}\|_{X^{1,\frac{1}{2}+}}\lesssim$$
\vspace{1mm}
$$\lesssim N_1^{-\frac{1}{2}}\|\mathcal{D}v_{N_1}\|_{X^{0,\frac{1}{2}+}}\|\mathcal{D}v_{N_2}\|_{X^{0,\frac{1}{2}+}}\|v\|_{X^{1,\frac{1}{2}+}}^4\leq N_1^{-\frac{1}{2}}\|\mathcal{D}v\|_{X^{0,\frac{1}{2}+}}^2 \|v\|_{X^{1,\frac{1}{2}+}}^4\lesssim$$
\vspace{1mm}
\begin{equation}
\label{eq:increment bound1}
\lesssim N_1^{-\frac{1}{2}} \|\mathcal{D}\Phi\|_{L^2}^2 \|\Phi\|_{H^1}^4 \lesssim N_1^{-\frac{1}{2}} \|\mathcal{D}\Phi\|_{L^2}^2.
\end{equation}

\vspace{2mm}

In the last two inequalities, we used Proposition 3.1., followed by the uniform bound on the $H^1$ norm of the solution to our equation given by the conservation of energy and mass.

\vspace{3mm}

This is the bound that we can obtain from (\ref{eq:Bound on M_6}). We now prove (\ref{eq:Bound on M_6}).

\vspace{2mm}

We must consider three possible subcases:

\vspace{2mm}

\textbf{Subcase 1:} $|n_2|<N.$

\vspace{1mm}

\textbf{Subcase 2:} $|n_2| \geq N \,\,\mbox{and}\,\,|n_3|<N.$

\vspace{1mm}

\textbf{Subcase 3:} $|n_3| \geq N.$

\vspace{2mm}

\textbf{Subcase 1:}

Here, we have:

$$N_1 \gtrsim N, N_2 \sim |n_2| <N, N_1 \sim N_2.$$

So, one obtains:

$$N_1 \sim N_2 \sim N.$$

Also, we know:

$$n_1-n_2+n_3-n_4+n_5-n_6=0,\,\,\mbox{and}\,\,n_3,n_4,n_5,n_6=O(N_1^{\frac{1}{2}})=O(N^{\frac{1}{2}}).$$

Consequently:

$$|n_1|=N+r_1,|n_2|=N-r_2,\,\,\mbox{where}\,\,r_1,r_2>0,\,\mbox{and}\,\,r_1,r_2=O(N^{\frac{1}{2}}).$$

$$\Rightarrow (\theta(n_1))^2-(\theta(n_2))^2+(\theta(n_3))^2-(\theta(n_4))^2+(\theta(n_5))^2-(\theta(n_6))^2=$$
$$=\frac{|n_1|^{2s}}{N^{2s}}-1+1-1+1-1=\frac{|n_1|^{2s}-N^{2s}}{N^{2s}}=\frac{(N+O(N^{\frac{1}{2}}))^{2s}-N^{2s}}{N^{2s}}=$$
$$=O\Big(\frac{N^{2s-\frac{1}{2}}}{N^{2s}}\Big)=O(N^{-\frac{1}{2}})=O(N_1^{-\frac{1}{2}})=O(N_1^{-\frac{1}{2}}\theta(N_1)\theta(N_2)).$$
\vspace{1mm}
In the last inequality, we used the fact that $\theta(N_1),\theta(N_2) \geq 1.$

\vspace{2mm}

\textbf{Subcase 2:}

\vspace{1mm}

Here: $n_2=n_1+(n_3-n_4+n_5-n_6)$, from where it follows that:

$$n_2=n_1+O(|n_1|^{\frac{1}{2}})$$

We observe:

$$(\theta(n_3))^2-(\theta(n_4))^2+(\theta(n_5))^2-(\theta(n_6))^2=1-1+1-1=0.$$

So:

$$M_6=(\theta(n_1))^2-(\theta(n_2))^2=\frac{|n_1|^{2s}}{N^{2s}}-\frac{|n_2|^{2s}}{N^{2s}}=
\frac{|n_1|^{2s}-|n_1+O(|n_1|^{\frac{1}{2}})|^{2s}}{N^{2s}}=$$
$$=O \Big(\frac{|n_1|^{2s-\frac{1}{2}}}{N^{2s}}\Big)=O\Big(\frac{N_1^{2s-\frac{1}{2}}}{N^{2s}}\Big)=
O\Big(N_1^{-\frac{1}{2}}\theta(N_1)\theta(N_2)\Big).$$

\vspace{2mm}

Here, we used the fact that: $\theta(N_1),\theta(N_2) \sim \frac{N_1^s}{N^s}.$

\vspace{2mm}

\textbf{Subcase 3:}

\vspace{1mm}

In this subcase, we can no longer use the cancelation coming from $$(\theta(n_3))^2-(\theta(n_4))^2+(\theta(n_5))^2-(\theta(n_6))^2.$$

The way one gets around this problem is as follows:

\vspace{1mm}

We first note that:

\vspace{1mm}

$(\theta(n_1))^2-(\theta(n_2))^2=O\Big(\frac{|n_1|^{2s-\frac{1}{2}}}{N^{2s}}\Big)=O\Big(N_1^{-\frac{1}{2}}\theta(N_1)\theta(N_2)\Big),\,$as before.

\vspace{1mm}

Also $|n_3|=O(|n_1|^{\frac{1}{2}})$, so:

$$(\theta(n_3))^2=O \Big(\frac{|n_3|^{2s}}{N^{2s}}\Big)=O \Big(\frac{|n_1|^{s}}{N^{2s}}\Big)=O \Big(\frac{|n_1|^{2s-\frac{1}{2}}}{N^{2s}}\Big).$$

Hence, by monotonicity properties of $\theta$, we deduce:

$$(\theta(n_3))^2,(\theta(n_4))^2,(\theta(n_5))^2,(\theta(n_6))^2=O \Big(\frac{|n_1|^{2s-\frac{1}{2}}}{N^{2s}}\Big).$$

Combining the previous estimates, we obtain:

$$M_6=O \Big(\frac{|n_1|^{2s-\frac{1}{2}}}{N^{2s}}\Big)=O\Big(N_1^{-\frac{1}{2}}\theta(N_1)\theta(N_2)\Big).$$

The estimate (\ref{eq:Bound on M_6}) now follows.

\vspace{3mm}

\textbf{Case 2:}

\vspace{1mm}

We recall that in this case, one has $N_3 \gtrsim N_1^{\frac{1}{2}}$. Here, we don't expect to get cancelation coming from $M_6$, so we just bound:

$$|M_6|=|(\theta(n_1))^2-(\theta(n_2))^2+(\theta(n_3))^2-(\theta(n_4))^2+(\theta(n_5))^2-(\theta(n_5))^2|$$
\begin{equation}
\label{eq:Bound on M_6 2}
\lesssim (\theta(n_1))^2 \lesssim (\theta(N_1))^2 \lesssim \theta(N_1)\theta(N_2).
\end{equation}

With notation as in Case 1, we use (\ref{eq:Bound on M_6 2}) and arguments analogous to those used to derive
(\ref{eq:increment bound1}) to deduce:

$$|I_{N_1,N_2,N_3,N_4,N_5,N_6}|\lesssim$$
\vspace{1mm}
$$\lesssim \theta(N_1) \theta(N_2)
\|v_{N_1}\|_{X^{0,\frac{1}{2}+}}\|v_{N_2}\|_{X^{0,\frac{1}{2}+}}\|v_{N_3}\|_{X^{0,\frac{1}{2}+}}\|v_{N_4}\|_{X^{0,\frac{1}{2}+}}
\|v_{N_5}\|_{X^{\frac{1}{2}+,\frac{1}{2}+}}\|v_{N_6}\|_{X^{\frac{1}{2}+,\frac{1}{2}+}}\lesssim$$
$$\lesssim \|\mathcal{D}v_{N_1}\|_{X^{0,\frac{1}{2}+}}\|\mathcal{D}v_{N_2}\|_{X^{0,\frac{1}{2}+}}(\frac{1}{N_3}\|v_{N_3}\|_{X^{1,\frac{1}{2}+}})\|v_{N_4}\|_{X^{0,\frac{1}{2}+}}
\|v_{N_5}\|_{X^{\frac{1}{2}+,\frac{1}{2}+}}\|v_{N_6}\|_{X^{\frac{1}{2}+,\frac{1}{2}+}}\lesssim$$
\begin{equation}
\label{eq:increment bound2}
\lesssim N_1^{-\frac{1}{2}}\|\mathcal{D}v\|_{X^{0\frac{1}{2}+}}^2 \|v\|_{X^{1,\frac{1}{2}+}}^4 \lesssim N_1^{-\frac{1}{2}} \|\mathcal{D}\Phi\|_{L^2}^2.
\end{equation}
The last bound follows from Proposition \ref{Proposition 3.1}. We note that this is the same bound we obtained in (\ref{eq:increment bound1}). Combining (\ref{eq:increment bound1}) and (\ref{eq:increment bound2}), and recalling that $I^{(1)}$ denotes the contribution of $I$ from Big Case 1, it follows that:

$$|I^{(1)}|\lesssim \sum_{N_j \,\mbox{satisfying}\, (\ref{eq:localizationN1}),(\ref{eq:localizationN2}),(\ref{eq:localizationN3})}
 N_1^{-\frac{1}{2}} \|\mathcal{D}\Phi\|_{L^2}^2 \lesssim$$
\vspace{1mm}
$$\lesssim \sum_{N_j \,\mbox{satisfying}\, (\ref{eq:localizationN1}),(\ref{eq:localizationN2}),(\ref{eq:localizationN3})}
 N_1^{-\frac{1}{2}+}N_2^{-0+}N_3^{-0+}N_4^{-0+}N_5^{-0+}N_6^{-0+} \|\mathcal{D}\Phi\|_{L^2}^2\lesssim$$
\vspace{1mm}
\begin{equation}
\label{eq:Big Case 1}
\lesssim \frac{1}{N^{\frac{1}{2}-}}\|\mathcal{D}\Phi\|_{L^2}^2.
\end{equation}
By construction, the implied constant depends only on $(s,Energy,Mass)$, and is continuous in energy and mass.

\vspace{4mm}

\textbf{Big Case 2:}

\vspace{2mm}

We recall that in this Big Case, in the expression for $M_6$, $(\theta(n_a))^2$ and
$(\theta(n_b))^2$ appear with the same sign. Arguing as in Big Case 1, we observe that the order of the four lower frequencies doesn't matter. Let us reorder the variables so that the hyperplane over which we are summing becomes $n_1+n_2+n_3-n_4-n_5-n_6=0$. It suffices to consider the case when:

$$|n_1| \geq |n_2| \geq |n_3| \geq |n_4| \geq |n_5| \geq |n_6|.$$
The expression we want to bound is:

$$\sum_{n_1+n_2+n_3-n_4-n_5-n_6=0,|n_1| \geq |n_2| \geq |n_3| \geq |n_3| \geq |n_4| \geq |n_5| \geq |n_6|}$$
$$\{\int_0^{\delta} M'_6 \hat{v}(n_1) \hat{v}(n_2) \hat{v}(n_3)
\overline{\hat{v}(n_4)}\,\overline{\hat{v}(n_5)}\,\overline{\hat{v}(n_6)}dt\}.$$

\vspace{1mm}

Here, we are taking: $$M'_6(n_1,n_2,n_3,n_4,n_5,n_6):=
(\theta(n_1))^2+(\theta(n_2))^2+(\theta(n_3))^2-(\theta(n_4))^2-(\theta(n_5))^2-(\theta(n_6))^2.$$

As before, we dyadically localize the factors of $v$ in the Fourier domain.

In this Big Case, we want to estimate:

$$J_{N_1,N_2,N_3,N_4,N_5,N_6}:=\sum_{n_1+n_2+n_3-n_4-n_5-n_6=0,|n_1| \geq |n_2| \geq |n_3| \geq |n_3| \geq |n_4| \geq |n_5| \geq |n_6|}$$
$$\{\int_0^{\delta} M'_6 \widehat{v_{N_1}}(n_1) \widehat{v_{N_2}}(n_2) \widehat{v_{N_3}}(n_3)
\overline{\widehat{v_{N_4}}(n_4)}\,\overline{\widehat{v_{N_5}}(n_5)}\,\overline{\widehat{v_{N_6}}(n_6)}dt\}.$$

One has the additional localizations on the $N_j$'s:

\begin{equation}
\label{eq:localizationN21}
N_1 \gtrsim N_2 \gtrsim N_3 \gtrsim N_4 \gtrsim N_5 \gtrsim N_6.
\end{equation}

\begin{equation}
\label{eq:localizationN22}
N_1 \sim N_2.
\end{equation}

\begin{equation}
\label{eq:localizationN23}
N_1 \gtrsim N.
\end{equation}

In this Big Case, we don't necessarily obtain any cancelation in $M_6'$, so we just write:

\begin{equation}
\label{eq:Bound on M_6'}
|M_6'|\lesssim (\theta(n_1))^2 \lesssim (\theta(N_1))^2 \lesssim \theta(N_1)\theta(N_2).
\end{equation}

Let us now estimate $J_{N_1,N_2,N_3,N_4,N_5,N_6}.$

\vspace{3mm}

Our analysis of this contribution will use techniques similar to those used in \cite{BGT,Z}. As we will see,
when one can't deduce decay estimates just from looking at the Fourier transform in $x$, one can look at the
Fourier transform in $t$.

\vspace{2mm}

We consider two cases:

\vspace{3mm}

\textbf{$\diamond$ Case 1:}
\vspace{1mm}
$N_3,N_4,N_5,N_6 \ll N_1^{\frac{1}{2}}.$

We observe that:

$$J_{N_1,N_2,N_3,N_4,N_5,N_6}=\sum_{n_1+n_2+n_3-n_4-n_5-n_6=0,|n_1| \geq \cdots \geq |n_6|}$$
$$ \int_{\mathbb{R}}
\{M_6'(\chi_{[0,\delta]}\widehat{v_{N_1}}(n_1))\widehat{v_{N_2}}(n_2)\widehat{v_{N_3}}(n_3)
\overline{\widehat{v_{N_4}}(n_4)}\,\overline{\widehat{v_{N_5}}(n_5)}\,\overline{\widehat{v_{N_6}}(n_6)}\}dt=$$

\vspace{2mm}

$$=\sum_{n_1+n_2+n_3-n_4-n_5-n_6=0,|n_1| \geq \cdots \geq |n_6|}\int_{\tau_1+\tau_2+\tau_3-\tau_4-\tau_5-\tau_6=0}$$
\vspace{1mm}
$$M_6' (\chi_{[0,\delta]}v_{N_1})\,\widetilde{\,}\,(n_1,\tau_1)\widetilde{v_{N_2}}(n_2,\tau_2)\widetilde{v_{N_3}}(n_3,\tau_3)
\overline{\widetilde{v_{N_4}}(n_4,\tau_4)}\,\overline{\widetilde{v_{N_5}}(n_5,\tau_5)}\,\overline{\widetilde{v_{N_6}}(n_6,\tau_6)}d \tau_j.$$

\vspace{2mm}

Now, as in \cite{BGT,Z}, we localize in parabolic regions determined by $\langle \tau + n^2 \rangle$.

Namely, given a dyadic integer $L_1$, we let $(\chi_{[0,\delta]}v_{N_1})_{L_1}=(\chi_{[0,\delta]}v)_{N_1,L_1}$ denote
the function obtained from $\chi_{[0,\delta]}v_{N_1}=(\chi_{[0,\delta]}v)_{N_1}$ by restricting its spacetime Fourier transform to the region where $\langle \tau + n^2 \rangle \sim L_1$.
\vspace{1mm}

Likewise, for $j \geq 2$, and for $L_j$ a dyadic integer, we denote by $v_{N_j,L_j}$ the function obtained from $v_{N_j}$ by
localizing its spacetime Fourier transform to $\langle \tau + n^2 \rangle \sim L_j$.

So, now, we want to estimate:

$$J_{\bar{L},\bar{N}}:=$$
$$=\sum_{n_1+n_2+n_3-n_4-n_5-n_6=0,|n_1| \geq \cdots \geq |n_6|}\int_{\tau_1+\tau_2+\tau_3-\tau_4-\tau_5-\tau_6=0}$$
\vspace{1mm}
$$M_6'
\widetilde{(\chi_{[0,\delta]}v)}_{N_1,L_1}(n_1,\tau_1)\widetilde{v_{N_2,L_2}}(n_2,\tau_2)\widetilde{v_{N_3,L_3}}(n_3,\tau_3)
\overline{\widetilde{v_{N_4,L_4}}(n_4,\tau_4)}\,\overline{\widetilde{v_{N_5,L_5}}(n_5,\tau_5)}\,
\overline{\widetilde{v_{N_6,L_6}}(n_6,\tau_6)}d \tau_j.$$

\vspace{3mm}

We have to consider two subcases w.r.t. the $\tau_j$:

\vspace{2mm}
\textbf{Subcase 1:} $|\tau_3|,|\tau_4|,|\tau_5|,|\tau_6| \ll N_1^2.$

\vspace{1mm}
\textbf{Subcase 2:} $max\,\{|\tau_3|,|\tau_4|,|\tau_5|,|\tau_6|\}\gtrsim N_1^2.$

\vspace{2mm}
\textbf{Subcase 1:}

Let us denote by $J^1_{\bar{L},\bar{N}}$ the contribution to $J_{\bar{L},\bar{N}}$ coming from this subcase.

Take $$(n_1,\tau_1) \in supp \, \widetilde{(\chi_{[0,\delta]}v)}_{N_1,L_1},$$
and:

$$(n_2,\tau_2) \in supp \, \widetilde{v_{N_2,L_2}}.$$
keeping in mind the assumptions of the subcase.

We then obtain:

$$L_1+L_2 \gtrsim |\tau_1-n_1^2| + |\tau_2-n_2^2| \geq |\tau_1+\tau_2+n_1^2+n_2^2| \geq |n_1^2+n_2^2|-|\tau_1+\tau_2|=$$
$$=|n_1^2+n_2^2|-|\tau_3-\tau_4-\tau_5-\tau_6|\geq |n_1|^2-|\tau_3|-|\tau_4|-|\tau_5|-|\tau_6|\gtrsim N_1^2$$
In the last inequality, we used the fact that:

$$|n_1|\gtrsim N_1, |\tau_3|,|\tau_4|,|\tau_5|,|\tau_6| \ll N_1^2.$$
In the calculation, we observe the crucial role of the inequality:
$$|n_1^2+n_2^2|\geq|n_1|^2.$$
Since $L_1,L_2 \geq 1$, the previous calculation gives us that:

\begin{equation}
\label{eq:subcase1product}
L_1 L_2 \gtrsim N_1^2.
\end{equation}
We now note that:

$$|J^1_{\bar{L},\bar{N}}|\leq$$
\vspace{2mm}
$$\leq\sum_{n_1+n_2+n_3-n_4-n_5-n_6=0,|n_1| \geq \cdots \geq |n_6|}\int_{\tau_1+\tau_2+\tau_3-\tau_4-\tau_5-\tau_6=0;
|\tau_3|,|\tau_4|,|\tau_5|,|\tau_6|\ll N_1^2}$$
\vspace{1mm}
$$\{|M_6'| |\widetilde{(\chi_{[0,\delta]}v)}_{N_1,L_1}(n_1,\tau_1)||\widetilde{v_{N_2,L_2}}(n_2,\tau_2)||\widetilde{v_{N_3,L_3}}(n_3,\tau_3)|$$
\vspace{1mm}
$$|\overline{\widetilde{v_{N_4,L_4}}(n_4,\tau_4)}|\,|\overline{\widetilde{v_{N_5,L_5}}(n_5,\tau_5)}|\,
|\overline{\widetilde{v_{N_6,L_6}}(n_6,\tau_6)}|\}d \tau_j\leq$$
\vspace{1mm}
$$\leq\sum_{n_1+n_2+n_3-n_4-n_5-n_6=0}\int_{\tau_1+\tau_2+\tau_3-\tau_4-\tau_5-\tau_6=0}$$
\vspace{1mm}
$$\{|M_6'| |\widetilde{(\chi_{[0,\delta]}v)}_{N_1,L_1}(n_1,\tau_1)||\widetilde{v_{N_2,L_2}}(n_2,\tau_2)||\widetilde{v_{N_3,L_3}}(n_3,\tau_3)|$$
\vspace{1mm}
$$|\widetilde{v_{N_4,L_4}}(n_4,\tau_4)|\,|\widetilde{v_{N_5,L_5}}(n_5,\tau_5)|\,
|\widetilde{v_{N_6,L_6}}(n_6,\tau_6)|\}d\tau_j.$$

\vspace{3mm}

Similarly as in Big Case 1, let us define:

$$G_1(x,t):=\sum_n \int e^{inx+it\tau} |\widetilde{(\chi_{[0,\delta]}v)}_{N_1,L_1}(n,\tau)| d\tau.$$
For $j=2,\ldots,6$, we let:

$$G_j(x,t):=\sum_n \int e^{inx+it\tau} |\widetilde{v_{N_j,L_j}}(n,\tau)| d\tau.$$
Arguing as in Big Case 1, using H\"{o}lder's inequality and (\ref{eq:Bound on M_6'}), we get \footnote{Strictly speaking, we should be truncating $G_3,G_4,G_5$, and $G_6$ to $|\tau|\ll N_1^2$, but we ignore this for simplicity of notation since we will later reduce to estimating these factors in $X^{s,b}$ norms, which don't increase if we localize the spacetime Fourier transform.}:

$$|J^1_{\bar{L},\bar{N}}|\lesssim \theta(N_1)\theta(N_2) \|G_1\|_{L^4_tL^2_x} \|G_2\|_{L^4_tL^2_x}
\|G_3\|_{L^4_tL^{\infty}_x}\|G_4\|_{L^4_tL^{\infty}_x} \|G_5\|_{L^{\infty}_{t,x}} \|G_6\|_{L^{\infty}_{t,x}}$$
which is by Sobolev embedding:

$$\lesssim \theta(N_1)\theta(N_2) \|G_1\|_{L^4_tL^2_x} \|G_2\|_{L^4_tL^2_x}
\|G_3\|_{L^4_tH^{\frac{1}{2}+}_x}\|G_4\|_{L^4_tH^{\frac{1}{2}+}_x} \|G_5\|_{X^{\frac{1}{2}+,\frac{1}{2}+}} \|G_6\|_{X^{\frac{1}{2}+,\frac{1}{2}+}}
$$
\vspace{1mm}
Since $supp\,\, \widehat{G_3} \subseteq \{-cN_3,\ldots,cN_3\},supp\,\, \widehat{G_4} \subseteq \{-cN_4,\ldots,cN_4\}$, this expression is:
\vspace{1mm}
$$\lesssim \theta(N_1)\theta(N_2) \|G_1\|_{L^4_tL^2_x} \|G_2\|_{L^4_tL^2_x}
(N_3^{\frac{1}{2}+}\|G_3\|_{L^4_tL^2_x})(N_4^{\frac{1}{2}+}\|G_4\|_{L^4_tL^2_x}) \|G_5\|_{X^{\frac{1}{2}+,\frac{1}{2}+}} \|G_6\|_{X^{\frac{1}{2}+,\frac{1}{2}+}} $$
\vspace{3mm}
which is furthermore by using (\ref{eq:L4tL2x}):
$$\lesssim \theta(N_1)\theta(N_2)N_3^{\frac{1}{2}+}N_4^{\frac{1}{2}+} \|G_1\|_{X^{0,\frac{1}{4}+}} \|G_2\|_{X^{0,\frac{1}{4}+}}
\|G_3\|_{X^{0,\frac{1}{4}+}}\|G_4\|_{X^{0,\frac{1}{4}+}} \|G_5\|_{X^{\frac{1}{2}+,\frac{1}{2}+}} \|G_6\|_{X^{\frac{1}{2}+,\frac{1}{2}+}} =$$
\vspace{1mm}
$$=\theta(N_1)\theta(N_2)N_3^{\frac{1}{2}+}N_4^{\frac{1}{2}+}\|(\chi_{[0,\delta]}v)_{N_1,L_1}\|_{X^{0,\frac{1}{4}+}} \|v_{N_2,L_2}\|_{X^{0,\frac{1}{4}+}}$$
\vspace{1mm}
$$\|v_{N_3,L_3}\|_{X^{0,\frac{1}{4}+}}\|v_{N_4,L_4}\|_{X^{0,\frac{1}{4}+}} \|v_{N_5,L_5}\|_{X^{\frac{1}{2}+,\frac{1}{2}+}} \|v_{N_6,L_6}\|_{X^{\frac{1}{2}+,\frac{1}{2}+}}\lesssim$$
\vspace{1mm}
$$\lesssim N_3^{\frac{1}{2}+}N_4^{\frac{1}{2}+}\|(\mathcal{D}(\chi_{[0,\delta]}v))_{N_1,L_1}\|_{X^{0,\frac{1}{4}+}} \|(\mathcal{D}v)_{N_2,L_2}\|_{X^{0,\frac{1}{4}+}}$$
\vspace{1mm}
$$\|v_{N_3,L_3}\|_{X^{0,\frac{1}{4}+}}\|v_{N_4,L_4}\|_{X^{0,\frac{1}{4}+}} \|v_{N_5,L_5}\|_{X^{\frac{1}{2}+,\frac{1}{2}+}} \|v_{N_6,L_6}\|_{X^{\frac{1}{2}+,\frac{1}{2}+}}$$
\vspace{1mm}
$$\lesssim N_3^{\frac{1}{2}+}N_4^{\frac{1}{2}+}\frac{1}{L_1^{\frac{1}{4}-}}\|(\mathcal{D}(\chi_{[0,\delta]}v))_{N_1,L_1}\|_{X^{0,\frac{1}{2}-}} \frac{1}{L_2^{\frac{1}{4}}}\|(\mathcal{D}v)_{N_2,L_2}\|_{X^{0,\frac{1}{2}+}}$$
\vspace{1mm}
$$\frac{1}{N_3 L_3^{\frac{1}{4}}}\|v_{N_3,L_3}\|_{X^{1,\frac{1}{2}+}}\frac{1}{N_4 L_4^{\frac{1}{4}}}\|v_{N_4,L_4}\|_{X^{1,\frac{1}{2}+}} \frac{1}
{N_5^{\frac{1}{2}-} L_5^{0+}}\|v_{N_5,L_5}\|_{X^{1,\frac{1}{2}+}}\frac{1}
{N_6^{\frac{1}{2}-} L_6^{0+}} \|v_{N_6,L_6}\|_{X^{1,\frac{1}{2}+}}$$
\vspace{1mm}
By (\ref{eq:Xsb localization}) and the definition of the localizations w.r.t. $N_j,L_j$, this quantity is:

\vspace{1mm}
$$\lesssim N_3^{\frac{1}{2}+}N_4^{\frac{1}{2}+}\frac{1}{L_1^{\frac{1}{4}-}}\frac{1}{L_2^{\frac{1}{4}}}\frac{1}{N_3 L_3^{\frac{1}{4}}}
\frac{1}{N_4 L_4^{\frac{1}{4}}}\frac{1}{N_5^{\frac{1}{2}-} L_5^{0+}}\frac{1}{N_6^{\frac{1}{2}-} L_6^{0+}}$$
\vspace{1mm}
$$\|\mathcal{D}v\|_{X^{0,\frac{1}{2}+}}^2 \|v\|_{X^{1,\frac{1}{2}+}}^2 \lesssim$$
\vspace{1mm}
$$\lesssim \frac{1}{(L_1L_2)^{\frac{1}{4}-}}\frac{1}{N_3^{\frac{1}{2}-}N_4^{\frac{1}{2}-}N_5^{\frac{1}{2}-}N_6^{\frac{1}{2}-}}
\frac{1}{L_1^{0+}L_2^{0+}L_3^{\frac{1}{4}}L_3^{\frac{1}{4}}L_5^{0+}L_6^{0+}}\|\mathcal{D}v\|_{X^{0,\frac{1}{2}+}}^2
\|v\|_{X^{1,\frac{1}{2}+}}^4 $$
\vspace{1mm}
From (\ref{eq:subcase1product}) and Proposition \ref{Proposition 3.1}, this is:

\vspace{1mm}
$$\lesssim \frac{1}{N_1^{\frac{1}{2}-}}\frac{1}{N_3^{\frac{1}{2}-}N_4^{\frac{1}{2}-}N_5^{\frac{1}{2}-}N_6^{\frac{1}{2}-}}
\frac{1}{L_1^{0+}L_2^{0+}L_3^{\frac{1}{4}}L_3^{\frac{1}{4}}L_5^{0+}L_6^{0+}}\|\mathcal{D}\Phi\|_{L^2}^2 \|\Phi\|_{H^1}^4\lesssim$$
\vspace{1mm}
\begin{equation}
\label{eq:subcase1bound}
\lesssim \frac{1}{N_1^{\frac{1}{2}-}}\frac{1}{N_2^{0+}N_3^{\frac{1}{2}-}N_4^{\frac{1}{2}-}N_5^{\frac{1}{2}-}N_6^{\frac{1}{2}-}}
\frac{1}{L_1^{0+}L_2^{0+}L_3^{\frac{1}{4}}L_3^{\frac{1}{4}}L_5^{0+}L_6^{0+}}\|\mathcal{D}\Phi\|_{L^2}^2.
\end{equation}
\vspace{2mm}
In order to deduce the last bound, we used the fact that: $N_1\sim N_2$ and $\|\Phi\|_{H^1}\lesssim 1$.

\vspace{2mm}

\textbf{Subcase 2:}

\vspace{2mm}

We recall that in this subcase, one has: 
$$max \{|\tau_3|,|\tau_4|,|\tau_5|,|\tau_6|\} \gtrsim N_1^2.$$
Let us consider the case: $|\tau_3|=max \{|\tau_3|,|\tau_4|,|\tau_5|,|\tau_6|\}.$ We can analogously consider the other
cases, but we have to group the factors in H\"{o}lder's Inequality then \footnote{We take the $L^4_tL^2_x$ norm of the factor with highest $|\tau|$.}. Let us localize as in the previous subcase, and let us denote by $J^2_{\bar{L},\bar{N}}$ the contribution to $J_{\bar{L},\bar{N}}$ coming from this subcase.

\vspace{1mm}
Suppose now that $(\tau_3,n_3) \in supp\,\,\widetilde{v_{N_3,L_3}}$, keeping in mind the assumptions of the subcase. Then:

$$|n_3|\sim N_3\ll N_1^{\frac{1}{2}}, |\tau_3| \gtrsim N_1^2 \Rightarrow |\tau_3+n_3^2|\gtrsim N_1^2-N_1 \gtrsim N_1^2.$$
Consequently:

\begin{equation}
\label{eq:L3bound}
L_3\gtrsim N_1^2.
\end{equation}
Arguing analogously as in the previous subcase, we obtain:

$$|J^2_{\bar{L},\bar{N}}| \lesssim \theta(N_1)\theta(N_2) N_3^{\frac{1}{2}+}N_4^{\frac{1}{2}+}$$
\vspace{1mm}
$$\|(\chi_{[0,\delta]}v)_{N_1,L_1}\|_{X^{0,\frac{1}{4}+}} \|v_{N_2,L_2}\|_{X^{0,\frac{1}{4}+}}
\|v_{N_3,L_3}\|_{X^{0,\frac{1}{4}+}}\|v_{N_4,L_4}\|_{X^{0,\frac{1}{4}+}} \|v_{N_5,L_5}\|_{X^{\frac{1}{2}+,\frac{1}{2}+}}
\|v_{N_6,L_6}\|_{X^{\frac{1}{2}+,\frac{1}{2}+}}\lesssim$$
\vspace{1mm}
$$\lesssim N_3^{\frac{1}{2}+} N_4^{\frac{1}{2}+} \frac{1}{L_1^{\frac{1}{4}-}} \|\mathcal{D}v_{N_1,L_1}\|_{X^{0,\frac{1}{2}+}}
\frac{1}{L_2^{\frac{1}{4}}}\|\mathcal{D}v_{N_2,L_2}\|_{X^{0,\frac{1}{2}+}}$$
\vspace{1mm}
$$\frac{1}{N_3L_3^{\frac{1}{4}}} \|v_{N_3,L_3}\|_{X^{1,\frac{1}{2}+}}
\frac{1}{N_4L_4^{\frac{1}{4}}} \|v_{N_4,L_4}\|_{X^{1,\frac{1}{2}+}} \frac{1}{N_5^{\frac{1}{2}-}L_5^{0+}}
\|v_{N_5,L_5}\|_{X^{1,\frac{1}{2}+}}\frac{1}{N_6^{\frac{1}{2}-}L_6^{0+}}\|v_{N_6,L_6}\|_{X^{1,\frac{1}{2}+}}\lesssim$$
\vspace{1mm}
$$\lesssim \frac{1}{L_3^{\frac{1}{4}-}} \frac{1}{N_3^{\frac{1}{2}-}N_4^{\frac{1}{2}-}N_5^{0+}N_6^{0+}}
\frac{1}{L_1^{\frac{1}{4}-}L_2^{\frac{1}{4}} L_3^{0+} L_4^{\frac{1}{4}}L_5^{0+}L_6^{0+}}
\|\mathcal{D}v\|_{X^{0,\frac{1}{2}+}}^2 \|v\|_{X^{1,\frac{1}{2}+}}^4$$
\vspace{1mm}
which by (\ref{eq:L3bound}) is:

\vspace{1mm}
$$\lesssim \frac{1}{N_1^{\frac{1}{2}-}} \frac{1}{N_3^{\frac{1}{2}-}N_4^{\frac{1}{2}-}N_5^{0+}N_6^{0+}}
\frac{1}{L_1^{\frac{1}{4}-}L_2^{\frac{1}{4}} L_3^{0+} L_4^{\frac{1}{4}}L_5^{0+}L_6^{0+}}
\|\mathcal{D}v\|_{X^{0,\frac{1}{2}+}}^2 \|v\|_{X^{1,\frac{1}{2}+}}^4$$
\begin{equation}
\label{eq:subcase2bound}
\lesssim \frac{1}{N_1^{\frac{1}{2}-}} \frac{1}{N_2^{0+}N_3^{\frac{1}{2}-}N_4^{\frac{1}{2}-}N_5^{0+}N_6^{0+}}
\frac{1}{L_1^{\frac{1}{4}-}L_2^{\frac{1}{4}} L_3^{0+} L_4^{\frac{1}{4}}L_5^{0+}L_6^{0+}}
\|\mathcal{D}\Phi\|_{L^2}^2.
\end{equation}
\vspace{3mm}
\textbf{$\diamond$ Case 2:}
\vspace{1mm}
$N_3 \gtrsim N_1^{\frac{1}{2}}.$
Let us recall that we want to estimate:

$$J_{N_1,N_2,N_3,N_4,N_5,N_6}=$$
\vspace{1mm}
$$\sum_{n_1+n_2+n_3-n_4-n_5-n_6=0,|n_1| \geq \ldots \geq |n_6|}
\int_{\mathbb{R}} M_6' (\chi_{[0,\delta]}\widehat{v_{N_1}}(n_1)) \widehat{v_{N_2}}(n_2)
\widehat{v_{N_3}}(n_3) \overline{\widehat{v_{N_4}}(n_4)}\,\overline{\widehat{v_{N_5}}(n_5)}\,\overline{\widehat{v_{N_6}}(n_6)}dt. $$
\vspace{2mm}
Let us note that: $$|M_6'|=|(\theta(n_1))^2+(\theta(n_2))^2+(\theta(n_3))^2-(\theta(n_4))^2-(\theta(n_5))^2-(\theta(n_6))^2|\lesssim
\theta(N_1)\theta(N_2).$$
We note that this Case is analogous to Case 2 of Big Case 1. Hence, arguing exactly as we did in this Case, we obtain:

\begin{equation}
\label{eq:case2bound}
|J_{N_1,N_2,N_3,N_4,N_5,N_6}|\lesssim \frac{1}{N_1^{\frac{1}{2}-}(N_2 N_3 N_4 N_5 N_6)^{0+}} \|\mathcal{D} \Phi\|_{L^2}^2.
\end{equation}
We combine (\ref{eq:subcase1bound}),(\ref{eq:subcase2bound}),(\ref{eq:case2bound}) and sum in $N_j,L_j$ to deduce that the contribution to $I$ from Big Case 2, which we denoted by $I^{(2)}$ has the property that:

\begin{equation}
\label{eq:Big Case 2}
|I^{(2)}|\lesssim \frac{1}{N^{\frac{1}{2}-}}\|\mathcal{D}\Phi\|_{L^2}^2.
\end{equation}
This gives us a good bound in Big Case 2.
\vspace{4mm}
Combining (\ref{eq:Big Case 1}) and (\ref{eq:Big Case 2}), we finally obtain:

$$|I|=|\|\mathcal{D}u(\delta)\|_{L^2}^2-\|\mathcal{D}u(0)\|_{L^2}^2| \lesssim \frac{1}{N^{\frac{1}{2}-}}\|\Phi\|_{L^2}^2.$$
By construction, the implied constant here depends only on $(s,Energy,Mass)$. Let us denote it by $C=C(s,Energy,Mass)$.
We use Proposition \ref{Proposition 3.1} and the fact that the $H^1$ norm can be bounded by a continuous function of energy and mass to deduce that $C$ is continuous in energy and mass. Lemma \ref{Bigstar1} now follows.

\end{proof}

\subsection{Proof of Theorem \ref{Theorem 1} for $k \geq 3$.}

\vspace{3mm}

We finally note that for $k \geq 3$, we can bound the increment of $\|\mathcal{D}u(t)\|_{L^2}^2$ in an analogous way as we did for $k=2$. Namely, we observe that all the estimates on $M_6,M_6'$ we used depended only on the two highest frequencies and not on how many more frequencies there were.
Furthermore, in the later estimates, when we had to use H\"{o}lder's inequality, we just estimate the $k-2$ extra factors in $L^{\infty}_{t,x}$, and use the fact that $X^{\frac{1}{2}+,\frac{1}{2}+}\hookrightarrow L^{\infty}_{t,x}$. At the end, this only results in a ``0+ loss'' in the dyadic decay factor,
and we get the same increment bound (\ref{eq:Bigstar}) as before.

This finishes the proof of Theorem \ref{Theorem 1} for $k \geq 2$. $\Box$

\vspace{3mm}

\subsection{Remarks on the result of Bourgain.}

\vspace{3mm}

As was mentioned in Section 1.3., in the appendix of \cite{B4}, Bourgain gives a sketch of how one should be able to deduce a better bound in the case $k=2$ though. The methods he indicates there don't seem to apply to the higher nonlinearities $k > 2$. The problem lies in the fact that the inductive procedure from \cite{B4} is linked to the quintic structure of the nonlinearity.

Bourgain starts by defining the following Besov-type norms:

$$\|u\|_{0,p}:=(\int_{\mathbb{R}}(\sum_j |\widetilde{u}(j,j^2+\xi)|^2)^{\frac{p}{2}}d\xi)^{\frac{1}{p}}.$$
This space is similar to the $X^{s,b}$ space we are using, but $X^{s,b}$ spaces were not used in \cite{B4}.
The estimate one starts from is the following Strichartz Estimate:
\vspace{3mm}
Assuming that $supp\,\,\hat{\phi}\,\subseteq \{-N,\ldots,N\}$, one has:

\begin{equation}
\label{eq:L6estimate1}
\|S(t)\phi\|_{L^6_{t,x}}\lesssim N^{0+}\|\phi\|_{L^2_x}.
\end{equation}
Suppose now that $q=q(x,t)$ has the property that $supp\,\,\widehat{q(t)}\subseteq \{-N,\ldots,N\}$.
By writing $u$ as a superposition of modulated free solutions (c.f. Lemma 2.9 in \cite{Tao}), (\ref{eq:L6estimate1})
implies:

\begin{equation}
\label{eq:L6estimate2}
\|q\|_{L^6_{t,x}}\lesssim N^{0+}\|q\|_{0,1}.
\end{equation}
By using H\"{o}lder's inequality, one then deduces:

\begin{equation}
\label{eq:Induction Base}
\int_{\mathbb{R}}\int_{S^1} |q(x,t)|^6 dxdt \lesssim \|q\|_{L^6_{t,x}}^6 \lesssim N^{0+}\|q\|_{0,1}^6.
\end{equation}
The estimate $(\ref{eq:Induction Base})$ is used as the base of the induction in the paper. At each step, the Hamiltonian is modified using a symplectic transformation of the phase space $l^2(\mathbb{Z})$ in such a way that the nonlinearity is reduced to its essential part. In each iteration, it is shown inductively that the analogue of (\ref{eq:Induction Base}) holds for the modified Hamiltonian.

The reason why one doesn't seem to be able to apply these methods to the case $k > 2$ is that the Besov-type norms introduced earlier don't allow us to control the spacetime $L^{\infty}$ norm in a satisfactory way. On the other hand, we recall that for $X^{s,b}$ spaces, we used the bound: $\|u\|_{L^{\infty}_{t,x}}\lesssim \|u\|_{X^{\frac{1}{2}+,\frac{1}{2}+}}$.
It appears that the only estimate, one can use for the spacetime $L^{\infty}$ norm is obtained as follows:

Suppose $q=q(x,t)$ satisfies $supp\,\ \widehat{q(t)} \subseteq \{-N,\ldots,N\}$.

Then:

\begin{equation}
\label{eq:Besovbound}
\|q\|_{L^{\infty}_{t,x}}\lesssim \|q\|_{L^{\infty}_t H^{\frac{1}{2}+}_x}
\lesssim N^{\frac{1}{2}+}\|q\|_{L^{\infty}_tL^2_x}\lesssim N^{\frac{1}{2}+} \|q\|_{0,1}.
\end{equation}
Here, in the first step, we used Sobolev embedding and in the last step, we used the triangle inequality.

From H\"{o}lder's inequality, (\ref{eq:L6estimate2}), and (\ref{eq:Besovbound}) we can deduce that for $k \geq 3$, one has:

$$\int_{\mathbb{R}}\int_{S^1} |q(x,t)|^{2k+2} dxdt \lesssim \|q\|_{L^6_{t,x}}^6 \|q\|_{L^{\infty}_{t,x}}^{2k-4}.$$

\begin{equation}
\label{eq:2k+2bound}
\lesssim N^{0+} \|q\|_{0,1}^6 N^{\frac{2k-4}{2}+} \|q\|_{0,1}^{2k-4}\lesssim N^{(k-2)+}\|q\|_{0,1}^{2k+2}.
\end{equation}
\vspace{2mm}
We observe that this no longer gives us a $N^{0+}$ factor on the right hand side, which was crucial in the proof in \cite{B4}.

\section{Modifications of the Cubic NLS.}

\subsection{Modification 1: Hartree Equation.}

Let us now consider the Hartree equation on $S^1$, i.e. the equation $(\ref{eq:Hartree})$.
The equation (\ref{eq:Hartree}) has the following conserved quantities:

$$M(u(t))=\int |u(x,t)|^2 dx\,\,\,\mbox{\emph{(Mass)}}$$

and

$$E(u(t))=\frac{1}{2}\int |\nabla u(x,t)|^2 dx + \frac{1}{4}\int (V*|u|^2)(x,t) |u(x,t)|^2 dx \,\,\,\mbox{\emph{(Energy)}}$$

The fact that the mass is conserved follows from the fact that $V$ is real-valued.
The fact that the energy is conserved can be checked by using the equation and integrating by parts. The calculation crucially relies on the fact that $V$ is even, see \cite{Ca}.
Furthermore, since $V \geq 0$, we immediately obtain uniform bounds on $\|u(t)\|_{H^1}$.
$M$ is clearly continuous on $H^1$. By using Young's inequality, H\"{o}lder's inequality and Sobolev embedding, it follows that $E$ is also continuous on $H^1$.

\subsubsection{Local-in-time estimates for the Hartree Equation}

\vspace{3mm}

Let $u$ denote a global solution of $(\ref{eq:Hartree})$.
Recalling the definition of the operator $\mathcal{D}$ in (\ref{eq:D operator}), we have:

\begin{proposition}
\label{Proposition 4.1}
Given $t_0 \in \mathbb{R}$, there exists a globally defined function $v:S^1 \times \mathbb{R} \rightarrow \mathbb{C}$ satisfying the properties:

\begin{equation}
\label{eq:properties of v12}
v|_{[t_0,t_0+\delta]}=u|_{[t_0,t_0+\delta]}.
\end{equation}

\begin{equation}
\label{eq:properties of v22}
\|v\|_{X^{1,\frac{1}{2}+}}\leq C(s,E(\Phi),M(\Phi))
\end{equation}

\begin{equation}
\label{eq:properties of v32}
\|\mathcal{D}v\|_{X^{0,\frac{1}{2}+}}\leq C(s,E(\Phi),M(\Phi)) \|\mathcal{D}u(t_0)\|_{L^2}.
\end{equation}
Moreover, $\delta$ and $C$ can be chosen to depend continuously on the energy and mass.

\end{proposition}

\begin{proof}

The proof of Proposition \ref{Proposition 4.1} is analogous to the Proof of Proposition \ref{Proposition 3.1} (see Appendix B). The only modification we have to make is to note that $V \in L^1(S^1)$ implies that $\widehat{V} \in L^{\infty}(S^1)$.
Instead of estimating an expression of the form $||v_{\delta}|^2 v_{\delta}|$ as in the proof of Proposition 3.1, we have to estimate: $|(V*|v_{\delta}|^2)v_{\delta}|.$

However,

$$|(V*|v_{\delta}|^2)v_{\delta}|=\big|\sum_{n_1+n_2+n_3=n} \int_{\tau_1+\tau_2+\tau_3=\tau} d\tau_j
\widehat{V}(n_1+n_2) \widetilde{v_{\delta}}(n_1,\tau_1)\widetilde{\bar{v}_{\delta}}(n_2,\tau_2)\widetilde{v_{\delta}}(n_3,\tau_3)\big| \leq $$
$$\leq \sum_{n_1+n_2+n_3=n} \int_{\tau_1+\tau_2+\tau_3=\tau} d\tau_j
|\widehat{V}(n_1+n_2)| |\widetilde{v_{\delta}}(n_1,\tau_1)||\widetilde{\bar{v}_{\delta}}(n_2,\tau_2)||\widetilde{v_{\delta}}(n_3,\tau_3)| \lesssim$$
$$\lesssim \sum_{n_1+n_2+n_3=n} \int_{\tau_1+\tau_2+\tau_3=\tau} d\tau_j
|\widetilde{v_{\delta}}(n_1,\tau_1)||\widetilde{\bar{v}_{\delta}}(n_2,\tau_2)||\widetilde{v_{\delta}}(n_3,\tau_3)|.$$

This is the same expression that we obtain in the proof of Proposition 3.1. The existence part (i.e. the analogue of properties (\ref{eq:properties of v2}) and (\ref{eq:properties of v3})) now follows in the same way as in the mentioned Proposition. On the other hand, for the uniqueness part (i.e. the analogue of (\ref{eq:properties of v1})), let $v(t),w(t)$ solve (\ref{eq:Hartree}) with the same initial data on the time interval $[0,\delta]$. We also suppose that $\|v(t)\|_{H^1},\|w(t)\|_{H^1}$ are uniformly bounded on this interval . By Minkowski's inequality, and by unitarity of the Linear Schr\"{o}dinger propagator, we obtain, for all $0\leq t \leq \delta$:

$$\|v(t)-w(t)\|_{L^2} \leq \int_0^t \|S(t-t')((V*|v|^2)v(t')-(V*|w|^2)w(t'))\|_{L^2} dt'=$$
$$=\int_0^t \|(V*|v|^2)v(t')-(V*|w|^2)w(t')\|_{L^2} dt'$$

If we combine H\"{o}lder's inequality, Young's inequality and Sobolev embedding, we deduce:

$$\|(V*(u_1 u_2))u_3\|_{L^2} \leq \|V\|_{L^1} \|u_1 u_2\|_{L^{\infty}} \|u_3\|_{L^2} \lesssim
\|u_1\|_{H^1} \|u_2\|_{H^1} \|u_3\|_{L^2}.$$

Similarly:

$$\|(V*(u_1 u_2))u_3\|_{L^2} \leq \|V\|_{L^1}\|u_1 u_2\|_{L^2} \|u_3\|_{L^{\infty}} \lesssim
\|u_1\|_{L^2}\|u_2\|_{H^1}\|u_3\|_{H^1}.$$

Hence:

$$\|v(t)-w(t)\|_{L^2} \lesssim \int_0^t (\|v(t')|_{H^1}+\|w(t')\|_{H^1})^2 \|v(t')-w(t')\|_{L^2} dt'
\lesssim \int_0^t \|v(t')-w(t')\|_{L^2} dt'.$$

Uniqueness now follows from Gronwall's inequality.
\end{proof}

We will now use the method of \emph{higher modified energies} as in \cite{CKSTT3,CKSTT5}. The key is to obtain a better approximation to $\|u(t)\|_{H^s}^2$ than $\|\mathcal{D}u(t)\|_{L^2}^2$ by using a multilinear correction term.

\subsubsection{Introduction of the Higher Modified Energy.}

Before we define the multilinear correction to $E^1(u):=\|\mathcal{D}u(t)\|_{L^2}^2$, let us first find $\frac{d}{dt}\|\mathcal{D}u(t)\|_{L^2}^2$.

$$\frac{d}{dt} \|\mathcal{D}u(t)\|_{L^2}^2 \sim \frac{d}{dt} \big( \sum_{n_1+n_2=0} \widehat{\mathcal{D}u}(n_1) \widehat{\mathcal{D}\bar{u}}(n_2) \big)=$$
\vspace{2mm}
$$=\sum_{n_1+n_2=0} (\theta(n_1)(i \Delta u - i (V*|u|^2)u)\, \widehat{}\, (n_1) \theta(n_2) \widehat{\bar{u}}(n_2)+
\theta(n_1)\widehat{u}(n_1)(-i \Delta \bar{u} + i(V*|u|^2)\bar{u})\, \widehat{} \,(n_2) \theta(n_2)=$$
\vspace{2mm}
$$=\sum_{n_1+n_2=0}\big(-i \big((\theta(n_1))^2 n_1^2 - (\theta(n_2))^2 n_2^2\big)\widehat{u}(n_1)\widehat{\bar{u}}(n_2)$$
$$-i\big((\theta(n_2))^2((V*|u|^2)u)\,\,\widehat{}\,\,(n_1)\widehat{\bar{u}}(n_2)-
(\theta(n_1))^2\widehat{u}(n_1)((V*|u|^2)\bar{u})\,\,\widehat{}\,\,(n_2)\big)\big)=$$
\vspace{2mm}
$$=-i\sum_{n_1+n_2+n_3+n_4=0} ((\theta(n_2))^2 \widehat{V}(n_3+n_4)\widehat{u}(n_1)\widehat{\bar{u}}(n_2)
\widehat{u}(n_3)\widehat{\bar{u}}(n_4)$$
$$-(\theta(n_1))^2\widehat{V}(n_3+n_4)\widehat{u}(n_1)\widehat{\bar{u}}(n_2)\widehat{u}(n_3)\widehat{\bar{u}}(n_4))=$$
\vspace{2mm}
$$=\frac{1}{2}i \sum_{n_1+n_2+n_3+n_4=0}((\theta(n_1))^2\widehat{V}(n_3+n_4)+(\theta(n_3))^2\widehat{V}(n_1+n_2)$$
$$-(\theta(n_2))^2\widehat{V}(n_3+n_4)-(\theta(n_4))^2\widehat{V}(n_1+n_2))\widehat{u}(n_1)
\widehat{\bar{u}}(n_2)\widehat{u}(n_3)\widehat{\bar{u}}(n_4)$$
Since $V$ is even, so is $\widehat{V}$. Hence, when $n_1+n_2+n_3+n_4=0$, we have that:
$\widehat{V}(n_1+n_2)=\widehat{V}(n_3+n_4)$. So, we deduce that:

$$\frac{d}{dt}\|\mathcal{D}u(t)\|_{L^2}^2=ci \sum_{n_1+n_2+n_3+n_4=0}\big((\theta(n_1))^2-(\theta(n_2))^2+(\theta(n_3))^2-(\theta(n_4))^2\big)$$
\begin{equation}
\label{eq:firstcontribution}
\widehat{V}(n_3+n_4)\widehat{u}(n_1)\widehat{\bar{u}}(n_2)\widehat{u}(n_3)\widehat{\bar{u}}(n_4),
\end{equation}
\vspace{2mm}
where $c$ is a real constant.

Recalling the notation from Section 2, we consider the following \emph{higher modified energy}

\begin{equation}
\label{eq:modifiedenergy}
E^2(u):=E^1(u) + \lambda_4(M_4;u).
\end{equation}
The quantity $M_4$ will be determined soon.

The modified energy $E^2$ comes as a ``multilinear correction'' of the modified energy $E^1$ considered earlier:

In order to find $\frac{d}{dt}E^2(u)$, we need to find $\frac{d}{dt}\lambda_4(M_4;u)$. Thus, if we fix a multiplier $M_4$, we obtain:

$$\frac{d}{dt} \lambda_4(M_4;u)=$$
\vspace{2mm}
$$\frac{d}{dt} \big( \sum_{n_1+n_2+n_3+n_4=0} M_4(n_1,n_2,n_3,n_4) \widehat{u}(n_1)
\widehat{\bar{u}}(n_2)\widehat{u}(n_3)\widehat{\bar{u}}(n_4) \big)=$$
\vspace{2mm}
$$=-i\lambda_4(M_4(n_1^2-n_2^2+n_3^2-n_4^2);u)$$
\vspace{2mm}
$$-i\sum_{n_1+n_2+n_3+n_4+n_5+n_6=0}\big[M_4(n_{123},n_4,n_5,n_6)\widehat{V}(n_1+n_2)$$
\vspace{2mm}
$$-M_4(n_1,n_{234},n_5,n_6)\widehat{V}(n_2+n_3)+M_4(n_1,n_2,n_{345},n_6)\widehat{V}(n_3+n_4)$$
\vspace{2mm}
\begin{equation}
\label{eq:secondcontribution}
-M_4(n_1,n_2,n_3,n_{456})\widehat{V}(n_4+n_5)\big]\widehat{u}(n_1)\widehat{\bar{u}}(n_2)\widehat{u}(n_3)\widehat{\bar{u}}(n_4)\widehat{u}(n_5)\widehat{\bar{u}}(n_6)
\end{equation}
From (\ref{eq:firstcontribution}), (\ref{eq:secondcontribution}), it follows that if we take:

\begin{equation}
\label{eq:definitionofM4}
M_4:= \Psi,
\end{equation}
where $\Psi$ is defined by:

$$\Psi: \Gamma_4 \rightarrow \mathbb{R}$$

\begin{equation}
\label{eq:definitionofpsi}
\Psi:=
\begin{cases}
  c\frac{((\theta(n_1))^2-(\theta(n_2))^2+(\theta(n_3))^2-(\theta(n_4))^2)\widehat{V}(n_3+n_4)}
  {n_1^2-n_2^2+n_3^2-n_4^2}  ,\,   \mbox{if }n_1^2-n_2^2+n_3^2-n_4^2 \neq 0 \\
  \,0,\,  \mbox{otherwise. }
\end{cases}
\end{equation}
for an appropriate real constant $c$.
\vspace{2mm}
One then has:

\begin{equation}
\label{eq:incremenentE2}
\frac{d}{dt}E^2(u)=-i \lambda_6(M_6;u).
\end{equation}

\vspace{2mm}

where:

\vspace{2mm}

$$M_6(n_1,n_2,n_3,n_4,n_5,n_6):=
M_4(n_{123},n_4,n_5,n_6)\widehat{V}(n_1+n_2)$$
\vspace{2mm}
$$-M_4(n_1,n_{234},n_5,n_6)\widehat{V}(n_2+n_3)+M_4(n_1,n_2,n_{345},n_6)\widehat{V}(n_3+n_4)$$
\vspace{2mm}
\begin{equation}
\label{eq:M_6_Hartree}
-M_4(n_1,n_2,n_3,n_{456})\widehat{V}(n_4+n_5)
\end{equation}

Heuristically, we expect this expression to be smaller than $\frac{d}{dt}E^1(u)$ since the derivatives are distributed over six factors of $u$ and $\bar{u}$, whereas before we only had four factors.
The key to continue our study of $E^2(u)$ is to deduce bounds on $\Psi$.

\subsubsection{Pointwise bounds on the multiplier $\Psi$.}

As in the previous section, we dyadically localize the frequencies as $|n_j| \sim N_j$.
We then order the $N_j$'s in decreasing order, to obtain
$N_1^* \geq N_2^* \geq N_3^* \geq N_4^*$.
Let us show that the following result holds:

\begin{lemma}
\label{Lemma 4.2}
Under the previous assumptions, one has:
\begin{equation}
\label{eq:psipointwisebound}
\Psi=O\big(\frac{1}{(N_1^*)^2}\theta(N_1^*)\theta(N_2^*)N_3^*N_4^*\big).
\end{equation}
\end{lemma}

\begin{proof}
From the triangle inequality and from the definition of $\theta$, it follows that we need to consider only:

\begin{equation}
\label{eq:N_j^*localization}
N_1^* \sim N_2^* \gtrsim N.
\end{equation}

Furthermore, by construction of $\Psi$, we just need to prove the bound when $n_1^2-n_2^2+n_3^2-n_4^2 \neq 0$.

We recall that:

\begin{equation}
\label{eq:Vhatbound}
|\widehat{V}|\lesssim 1
\end{equation}
Hence, the factor of $\widehat{V}(n_3+n_4)$ will not affect the estimate.

\vspace{2mm}

In the proof of Lemma \ref{Lemma 4.2}, it is crucial to observe that, for $(n_1,n_2,n_3,n_4) \in \Gamma_4$:

$$n_1^2-n_2^2+n_3^2-n_4^2=(n_1-n_2)(n_1+n_2)+(n_3-n_4)(n_3+n_4)=(n_1-n_2)(n_1+n_2)-(n_3-n_4)(n_1+n_2)=$$
\begin{equation}
\label{eq:denominator}
=(n_1+n_2)(n_1-n_2-n_3+n_4)=2(n_1+n_2)(n_1+n_4)
\end{equation}
In particular, when $n_1^2-n_2^2+n_3^2-n_4^2 \neq 0$, one has: $n_1+n_2,n_1+n_4 \neq 0.$

\vspace{1mm}
We must consider several cases:
\vspace{1mm}

$\diamondsuit$\textbf{Case 1:} $N_2^* \gg N_3^*.$

\vspace{1mm}

$\diamondsuit$\textbf{Case 2:} $N_2^* \sim N_3^*.$

\vspace{2mm}

\,\,\textbf{Case 1:} Let's suppose WLOG that: $|n_1| \geq |n_3|, |n_2| \geq |n_4|$, and $|n_1| \sim N_1^*$.

\vspace{1mm}

One needs to consider two Subcases:
\vspace{1mm}

$\diamond$\textbf{Subcase 1:} $|n_2| \sim N_2^*.$

\vspace{1mm}

$\diamond$\textbf{Subcase 2:} $|n_3| \sim N_2^*.$

\vspace{1mm}

\textbf{Subcase 1:}

Since $n_1+n_2+n_3+n_4=0,|n_1|,|n_2|\gg |n_3|,|n_4|$, it follows that $n_1$ and $n_2$ have the opposite sign.

Consequently:

$$|n_1+n_2|=||n_1|-|n_2||.$$
However, $|n_1+n_2|=|n_3+n_4|$ so:

$$||n_1|-|n_2||=|n_3+n_4|.$$
From (\ref{eq:denominator}), one obtains:

\begin{equation}
\label{eq:denominator2}
|n_1^2-n_2^2+n_3^2-n_4^2|=2|(n_1+n_2)(n_1+n_4)| \sim N_1^* |n_3+n_4|.
\end{equation}
In the last estimate, we used the fact that $|n_1|\gg |n_4|$ and $|n_1+n_2|=|n_3+n_4|$.

\vspace{2mm}

Let us now analyze the numerator. We start by observing that \footnote{We are considering $|n_1|\geq |n_2|,|n_1| \geq N$;
it's possible that $|n_2|<N$, but this is accounted for by the $``\leq''$.}:

$$|(\theta(n_1))^2-(\theta(n_2))^2|\leq$$
\begin{equation}
\label{eq:numerator12}
\leq \frac{1}{N^{2s}}(|n_1|^{2s}-|n_2|^{2s}) \lesssim \frac{1}{N^{2s}}|n_1|^{2s-1}||n_1|-|n_2||=\frac{1}{N^{2s}}|n_1|^{2s-1}|n_3+n_4|.
\end{equation}
We now have to consider $(\theta(n_3))^2-(\theta(n_4))^2$.

\vspace{1mm}

One must consider three possibilities:

\vspace{3mm}

\textbf{Sub-subcase 1:} $|n_3|,|n_4|<N.$

\vspace{1mm}

\textbf{Sub-subcase 2:} $|n_4|<N \leq |n_3|$ or $|n_3|<N \leq |n_4|$.

\vspace{1mm}

\textbf{Sub-subcase 3:} $|n_3|,|n_4| \geq N.$

\vspace{3mm}

\textbf{Sub-subcase 1:} In this sub-subcase, one has:$(\theta(n_3))^2-(\theta(n_4))^2=0.$

\vspace{1mm}

\textbf{Sub-subcase 2:} Let's consider WLOG the case when $|n_4|<N \leq |n_3|$. The case $|n_3|<N \leq |n_4|$ is analogous.

\vspace{2mm}

We obtain:

$$|(\theta(n_3))^2-(\theta(n_4))^2|=\frac{1}{N^{2s}}||n_3|^{2s}-N^{2s}|\leq \frac{1}{N^{2s}}||n_3|^{2s}-|n_4|^{2s}|=$$

$$=\frac{1}{N^{2s}}||n_3|^{2s}-|-n_4|^{2s}|\lesssim \frac{1}{N^{2s}}|n_3|^{2s-1}|n_3+n_4|.$$
We note that the first inequality follows from the assumptions of the sub-subcase.

\textbf{Sub-subcase 3:}

We note: 
$$|(\theta(n_3))^2-(\theta(n_4))^2|=\frac{1}{N^{2s}}||n_3|^{2s}-|n_4|^{2s}|.$$

Arguing as in the previous sub-subcase, we obtain:

$$|(\theta(n_3))^2-(\theta(n_4))^2|\lesssim \frac{1}{N^{2s}}|n_3|^{2s-1}|n_3+n_4|.$$
\vspace{4mm}
So, we obtain that in Subcase 1, one has the bound:

\begin{equation}
\label{eq:numerator34}
|(\theta(n_3))^2-(\theta(n_4))^2|\lesssim \frac{1}{N^{2s}}|n_3|^{2s-1}|n_3+n_4| \lesssim \frac{1}{N^{2s}}(N_1^*)^{2s-1}|n_3+n_4|.
\end{equation}

Combining (\ref{eq:numerator12}) and (\ref{eq:numerator34}), one obtains:

\begin{equation}
\label{eq:numeratorA}
|(\theta(n_1))^2-(\theta(n_2))^2+(\theta(n_3))^2-(\theta(n_4))^2| \lesssim \frac{1}{N^{2s}}
(N_1^*)^{2s-1}|n_3+n_4|.
\end{equation}

From (\ref{eq:Vhatbound}), (\ref{eq:denominator2}), and (\ref{eq:numeratorA}), it follows that in Subcase 1:

\begin{equation}
\label{eq:psi1}
\Psi=O\big(\frac{1}{N^{2s}}(N_1^*)^{2s-2}\big)=O\big(\frac{1}{(N_1^*)^2}\frac{(N_1^*)^{2s}}{N^{2s}}\big)=
O\big(\frac{1}{(N_1^*)^2}\theta(N_1^*)\theta(N_2^*)\big).
\end{equation}
\textbf{Subcase 2:}

\vspace{2mm}

Here one has $|n_3| \sim N_2^*.$
In this Subcase, we don't expect to obtain any cancelation in the numerator or in the denominator. We get:

$$|(\theta(n_1))^2-(\theta(n_2))^2+(\theta(n_3))^2-(\theta(n_4))^2|=O((\theta(n_1))^2)=O(\frac{1}{N^{2s}}(N_1^*)^{2s})$$
$$|n_1^2-n_2^2+n_3^2-n_4^2| \sim (N_1^*)^2.$$

So, again using (\ref{eq:Vhatbound}), we deduce:

\begin{equation}
\label{eq:psi2}
\Psi=O\big(\frac{1}{N^{2s}}(N_1^*)^{2s-2}\big)=O\big(\frac{1}{(N_1^*)^2}\theta(N_1^*)\theta(N_2^*)\big).
\end{equation}
\vspace{2mm}
\textbf{Case 2:}

\vspace{1mm}
\textbf{Subcase 1:}

\vspace{2mm}

We first consider the subcase when: $N_1^* \sim N_2^* \sim N_3^* \gg N_4^*.$

Let us assume WLOG that $|n_4|\sim N_4^*.$

Then, by $(\ref{eq:denominator})$, one has:

\begin{equation}
\label{eq:denominatorcase2}
||n_1|^2-|n_2|^2+|n_3|^2-|n_4|^2|=2|(n_1+n_2)(n_1+n_4)|\gtrsim N_1^*.
\end{equation}

Here, we  also used the fact that $|n_1+n_4| \sim N_1^*$ and $|n_1+n_2| \geq 1$. The latter observation follows from the fact that the problem is periodic.

We bound the numerator by:

\begin{equation}
\label{eq:numeratorcase2}
|(\theta(n_1))^2-(\theta(n_2))^2+(\theta(n_3))^2-(\theta(n_4))^2|\lesssim (\theta(n_1))^2 \lesssim \frac{1}{N^{2s}} {(N_1^*)}^{2s}.
\end{equation}

It follows from (\ref{eq:denominatorcase2}), (\ref{eq:numeratorcase2}), and (\ref{eq:Vhatbound}) that:

$$\Psi=O\big(\frac{1}{N_1^*}\frac{(N_1^*)^{2s}}{N^{2s}}\big)=O\big(\frac{1}{(N_1^*)^2}\frac{(N_1^*)^{2s}}{N^{2s}}N_1^*\big)=$$

\begin{equation}
\label{eq:psi3}
=O\big(\frac{1}{(N_1^*)^{2}}\theta(N_1^*)\theta(N_2^*)N_3^*\big).
\end{equation}
\textbf{Subcase 2:}

\vspace{2mm}

In this case, all the $N_j^*$'s are equivalent:

$$N_1^* \sim N_2^* \sim N_3^* \sim N_4^*.$$
By using (\ref{eq:denominator}), and the fact that $|n_1+n_2| \geq 1,|n_1+n_4| \geq 1$, it follows that:

\begin{equation}
\label{eq:denominatorcase3}
||n_1|^2-|n_2|^2+|n_3|^2-|n_4|^2|=2|(n_1+n_2)(n_1+n_4)| \gtrsim 1.
\end{equation}
As before:

\begin{equation}
\label{eq:numeratorcase3}
|(\theta(n_1))^2-(\theta(n_2))^2+(\theta(n_3))^2-(\theta(n_4))^2|\lesssim (\theta(n_1))^2 \lesssim \frac{1}{N^{2s}} {N_1^*}^{2s}.
\end{equation}
(\ref{eq:Vhatbound}), (\ref{eq:numeratorcase3}), and (\ref{eq:denominatorcase3}) now imply:

$$\Psi=O \big(\frac{(N_1^*)^{2s}}{N^{2s}}\big)=O\big(\frac{1}{(N_1^*)^2}\theta(N_1^*)\theta(N_2^*)(N_1^*)^2\big)=$$
\begin{equation}
\label{eq:psi4}
=O\big(\frac{1}{(N_1^*)^2}\theta(N_1^*)\theta(N_2^*)N_3^* N_4^*\big).
\end{equation}
Lemma \ref{Lemma 4.2} now follows from (\ref{eq:psi1}),(\ref{eq:psi2}),(\ref{eq:psi3}), and (\ref{eq:psi4}).
\end{proof}

\vspace{2mm}

\subsubsection{An approximation result for the higher modified energies.}

Let us now show that $E^2(u)$ is a good approximation of $E^1(u)$ in a certain precise sense. The result that we prove is:

\begin{lemma}
\label{Lemma 4.3}
If we take $N$ to be sufficiently large, then:
$$E^2(u) \sim E^1(u),$$
where the implied constant no longer depends on $N$, but depends continuously on energy and mass.
\end{lemma}

\begin{proof}
By construction, we have that: $|E^2(u(t))-E^1(u(t))|=|\lambda_4(M_4;u(t))|$, where $M_4$ has been defined in (\ref{eq:definitionofM4}).Let us WLOG consider the contribution to $\lambda_4(M_4;u(t))$ in which $|n_1| \geq |n_2| \geq |n_3| \geq |n_4|$. The other contributions are bounded analogously. With notation from before, we obtain the following localization:

\begin{equation}
\label{eq:Nj*localization}
N_1^* \geq N_2^* \geq N_3^* \geq N_4^*;\,\,N_1^* \gtrsim N.
\end{equation}
Using Lemma \ref{Lemma 4.2} we note that the corresponding contribution to $|E^2(u)-E^1(u)|$ is:

$$\lesssim \sum_{n_1+n_2+n_3+n_4=0,|n_1|\geq \ldots \geq |n_4|}\,\,\sum_{N_j^* \,\mbox{satisfying}\, (\ref{eq:Nj*localization})}$$
$$\frac{1}{(N_1^*)^2} \theta(N_1^*)\theta(N_1^*)N_3^* N_4^*|\widehat{u_{N_1^*}}(n_1)||\widehat{\overline{u_{N_2^*}}}(n_2)||\widehat{u_{N_3^*}}(n_3)||\widehat{\overline{u_{N_4^*}}}(n_4)|.$$
By taking inverse Fourier transforms, using an $L^2_x,L^2_x,L^{\infty}_x,L^{\infty}_x$ H\"{o}lder inequality, the
$H^{\frac{1}{2}+}_x\hookrightarrow L^{\infty}_x$ Sobolev embedding and the fact that $\|\cdot\|_{L^2_x},\|\cdot\|_{H^{\frac{1}{2}+}_x}$
are invariant under change of sign in the Fourier transform, we obtain that the previous quantity is:

$$\lesssim \sum_{N_j^* \,\mbox{satisfying}\, (\ref{eq:Nj*localization})}
\frac{1}{(N_1^*)^{1-}} \|\theta(N_1^*)u_{N_1^*}\|_{L^2}\|\theta(N_2^*)u_{N_2^*}\|_{L^2}
\|(N_3^*)^{\frac{1}{2}-}u_{N_3^*}\|_{H^{\frac{1}{2}+}}\|(N_4^*)^{\frac{1}{2}-}u_{N_4^*}\|_{H^{\frac{1}{2}+}}
\lesssim$$
$$\lesssim \sum_{N_j^* \,\mbox{satisfying}\, (\ref{eq:Nj*localization})} \frac{1}{(N_1^*)^{1-}}
\|\mathcal{D}u\|_{L^2}^2 \|u\|_{H^1}^2$$
$$\lesssim \frac{1}{N^{1-}} \|\mathcal{D}u\|_{L^2}^2= \frac{1}{N^{1-}}E^1(u).$$
The other contributions are bounded in an analogous way. Hence,

$$|E^2(u)-E^1(u)| \lesssim \frac{1}{N^{1-}}E^1(u).$$
Thus, if we take $N$ sufficiently large, we get for the fixed time $t$:

\begin{equation}
\label{eq:E2E1}
E^2(u(t)) \sim E^1(u(t)).
\end{equation}
The implied constant above doesn't depend on $N$ as long as we choose $N$ to be sufficiently large.
It also doesn't depend on $t$. We see that it depends on the uniform bound on $\|u(t)\|_{H^1}$, hence it depends continuously on energy and mass.

\end{proof}
Hence, in order to bound $E^1(u)$, it suffices to bound $E^2(u).$

\vspace{2mm}

\subsubsection{Estimate on the increment of $E^2(u)$ and proof of Theorem \ref{Theorem 2}.}

For $t_0 \in \mathbb{R}$,we now want to estimate the increment: $$E^2(u(t_0+\delta))-E^2(u(t_0)).$$

\vspace{2mm}

The bound that we will prove is:

\begin{lemma}
\label{Bigstar_Hartree}
For all $t_0 \in \mathbb{R}$, one has: $$|E^2(u(t_0+\delta))-E^2(u(t_0))|\lesssim \frac{1}{N^{2-}}E^2(u(t_0)).$$
\end{lemma}

\vspace{3mm}

Let us observe how Lemma \ref{Bigstar_Hartree} implies Theorem \ref{Theorem 2}:

\begin{proof}(\emph{of Theorem \ref{Theorem 2} assuming Lemma \ref{Bigstar_Hartree}})
We argue similarly as in the proof of Theorem \ref{Theorem 1}. Namely, from Lemma \ref{Bigstar_Hartree}, together with (\ref{eq:bound on D}) and Lemma \ref{Lemma 4.3}, we deduce that:

\begin{equation}
\label{eq:intermediatebound}
E^2(u(T))\lesssim E^2(\Phi) \lesssim E^1(\Phi) = \|\mathcal{D}\Phi\|_{L^2}^2\lesssim \|\Phi\|_{H^s}^2,
\end{equation}
whenever $T \lesssim  N^{2-}.$

So, for such $T$, one has, from (\ref{eq:bound on D}), Lemma \ref{Lemma 4.3}, and (\ref{eq:intermediatebound}):

\begin{equation}
\label{eq:Tbound}
\|u(T)\|_{H^s} \lesssim N^s \sqrt{E^1(u(T))} \lesssim N^s \sqrt{E^2(u(T))}
\lesssim N^s \|\Phi\|_{H^s}.
\end{equation}
Since $T \lesssim N^{2-}$, we can take $N=T^{\frac{1}{2}+}$. Substituting this into (\ref{eq:Tbound}), we obtain:

\begin{equation}
\label{eq:finalboundforTbig}
\|u(T)\|_{H^s} \lesssim T^{\frac{1}{2}s+}\|\Phi\|_{H^s}.
\end{equation}
Here the implied constants depend only on $(s,Energy,Mass)$, and they depend continuously on energy and mass.

Using (\ref{eq:finalboundforTbig}), and arguing as in the proof of Theorem \ref{Theorem 1}, we obtain that, for $s \geq 1$, there exists $C=C(s,Energy,Mass)$, depending continuously on energy and mass such that for all $t \in \mathbb{R}$:
\begin{equation}
\label{eq:HartreeBound}
\|u(t)\|_{H^s}\leq C(1+|t|)^{\frac{1}{2}s+}\|\Phi\|_{H^s}.
\end{equation}
\end{proof}

We now prove Lemma \ref{Bigstar_Hartree}:

\begin{proof}(\emph{of Lemma \ref{Bigstar_Hartree}})
From Proposition \ref{Proposition 4.1}, given $t_0$, we can construct a global function $v$  which agrees with $u$ on $[t_0,t_0+\delta]$ and which satisfies appropriate $X^{s,b}$ bounds. Let's WLOG suppose that $t_0=0.$ We note that all the constants depend only on conserved quantities of the equation, and hence will be independent of $t_0$. From Lemma \ref{Lemma 4.3}, one obtains for $t \in [0,\delta]$:

$$E^2(v) \sim E^1(v).$$
Furthermore, from (\ref{eq:incremenentE2}) and the construction of $v$, we recall for $t \in [0,\delta]$:
$$\frac{d}{dt}E^2(v(t))=-i \lambda_6(M_6;v(t)).$$
We want to estimate $\int_0^{\delta} \frac{d}{dt}E^2(v)(t)dt$.
In order to do this, we just consider the contribution:
$$\int_{0}^{\delta}\sum_{n_1+n_2+n_3+n_4+n_5+n_6=0}
M_4(n_{123},n_4,n_5,n_6)\widehat{V}(n_1+n_2)$$
\begin{equation}
\label{eq:3zvijezde}
\widehat{v}(n_1)\widehat{\bar{v}}(n_2)\widehat{v}(n_3)\widehat{\bar{v}}(n_4)\widehat{v}(n_5)\widehat{\bar{v}}(n_6)dt=:K
\end{equation}
By symmetry, the other contributions are bounded in an analogous way, since as we will see, our argument won't
depend on which factor comes with a complex conjugate, and which factor doesn't.

\vspace{2mm}

Let us now dyadically localize in frequency, with the following localizations:
$$|n_1+n_2+n_3| \sim N_1,\,|n_4| \sim N_2,\,|n_5| \sim N_3,\,|n_6|\sim N_4.$$
\vspace{1mm}
As before, we introduce the dyadic integers $N_1^*,N_2^*,N_3^*,N_4^*$.
\vspace{1mm}
It is then the case that:
\begin{equation}
\label{eq:Njstar}
N_1^* \geq N_2^* \geq N_3^* \geq N_4^*,\,N_1^* \gtrsim N.
\end{equation}
The latter fact follows from the fact that the only nonzero contribution comes from the case where
$(\theta(n_1+n_2+n_3))^2-(\theta(n_4))^2+(\theta(n_5))^2-(\theta(n_6))^2 \neq 0.$
\vspace{2mm}
Let's fix an admissible configuration $(N_1,N_2,N_3,N_4)$ and let's denote its contribution to $K$ by:
$$K_{N_1,N_2,N_3,N_4}:=\int_0^{\delta} \sum_{n_1+n_2+n_3+n_4+n_5+n_6=0}
M_4(n_1+n_2+n_3,n_4,n_5,n_6)$$
$$\widehat{(v\bar{v}v)}_{N_1}(n_1+n_2+n_3)
\widehat{\overline{v_{N_2}}}(n_2)\widehat{v_{N_3}}(n_5)\widehat{\overline{v_{N_4}}}(n_6)dt.$$
We must consider several cases:

\vspace{2mm}

$\diamondsuit$\textbf{Case 1:} $N_1=N_1^*$ or $N_1=N_2^*.$

\vspace{1mm}

$\diamondsuit$\textbf{Case 2:} $N_1=N_3^*$ or $N_1=N_4^*.$

\vspace{2mm}

\textbf{Case 1:}

By symmetry, we consider the case $N_1=N_1^*.$ We will also consider the case when $N_2=N_2^*,N_3=N_3^*,N_4=N_4^*$. The other cases are bounded in a similar way (we just group the terms differently). We obtain:

$$|K_{N_1,N_2,N_3,N_4}|=\big| \int_{\mathbb{R}}\,\,\sum_{n_1+n_2+n_3+n_4+n_5+n_6=0}$$
\vspace{1mm}
$$M_4(n_1+n_2+n_3,n_4,n_5,n_6)\widehat{(v\bar{v}v)}_{N_1}(n_1+n_2+n_3)
(\overline{\chi_{[0,\delta]}v}_{N_2})\,\widetilde{\,}\,(n_4) \widehat{v_{N_3}}(n_5)
\widehat{\overline{v}_{N_4}}(n_6)dt \big|=$$
\vspace{1mm}
$$=\big| \int_{\tau_1+\tau_2+\tau_3+\tau_4+\tau_5+\tau_6=0} \sum_{n_1+n_2+n_3+n_4+n_5+n_6=0} M_4(n_1+n_2+n_3,n_4,n_5,n_6)$$
\vspace{1mm}
$$\widetilde{(v\bar{v}v)}_{N_1}(n_1+n_2+n_3,\tau_1+\tau_2+\tau_3)(\overline{\chi_{[0,\delta]}v}_{N_2})\,\widetilde{\,}\,(n_4,\tau_4)
\widetilde{v_{N_3}}(n_5,\tau_5)\widetilde{\overline{v}_{N_4}}(n_6,\tau_6)d\tau_j \big|$$
\vspace{1mm}
Using the triangle inequality, Lemma \ref{Lemma 4.2}, and $(\ref{eq:Vhatbound})$, this expression is:

\vspace{1mm}
$$\lesssim \int_{\tau_1+\tau_2+\tau_3+\tau_4+\tau_5+\tau_6=0} \sum_{n_1+n_2+n_3+n_4+n_5+n_6=0} \frac{1}{(N_1^*)^2}
\theta(N_1^*) \theta(N_2^*) N_3^* N_4^*$$
\vspace{1mm}
$$|\widetilde{(v\bar{v}v)}_{N_1}(n_1+n_2+n_3,\tau_1+\tau_2+\tau_3)||(\overline{\chi_{[0,\delta]}v}_{N_2})\,\widetilde{\,}\,(n_4,\tau_4)|
|\widetilde{v_{N_3}}(n_5,\tau_5)||\widetilde{\overline{v}_{N_4}}(n_6,\tau_6)|d\tau_j$$
\vspace{1mm}
Since $\theta(N_1^*) \sim \theta(n_1+n_2+n_3)$, by localization, and since $|\widetilde{(v\bar{v}v)}_{N_1}| \leq |\widetilde{v\bar{v}v}|$ by restriction, this expression is:

\vspace{1mm}
$$\lesssim \int_{\tau_1+\tau_2+\tau_3+\tau_4+\tau_5+\tau_6=0} \sum_{n_1+n_2+n_3+n_4+n_5+n_6=0} \frac{1}{(N_1^*)^2}
\theta(n_1+n_2+n_3) \theta(N_2) N_3 N_4$$
\vspace{1mm}
$$|\widetilde{v\bar{v}v}(n_1+n_2+n_3,\tau_1+\tau_2+\tau_3)||(\overline{\chi_{[0,\delta]}v}_{N_2})\,\widetilde{\,}\,(n_4,\tau_4)|
|\widetilde{v_{N_3}}(n_5,\tau_5)||\widetilde{\overline{v}_{N_4}}(n_6,\tau_6)|d\tau_j$$
\vspace{1mm}
$$\lesssim \int_{\tau_1+\tau_2+\tau_3+\tau_4+\tau_5+\tau_6=0} \sum_{n_1+n_2+n_3+n_4+n_5+n_6=0} \frac{1}{(N_1^*)^2}
\theta(n_1+n_2+n_3) \theta(N_2) N_3 N_4$$
\vspace{1mm}
$$|\widetilde{v}(n_1,\tau_1)||\widetilde{\bar{v}}(n_2,\tau_2)||\widetilde{v}(n_3,\tau_3)||(\overline{\chi_{[0,\delta]}v}_{N_2})\,\widetilde{\,}\,(n_4,\tau_4)|
|\widetilde{v_{N_3}}(n_5,\tau_5)||\widetilde{\overline{v}_{N_4}}(n_6,\tau_6)|d\tau_j$$
\vspace{1mm}
Since one has the \emph{``Fractional Leibniz Rule''}: $\theta(n_1+n_2+n_3) \lesssim \theta(n_1)+\theta(n_2)+\theta(n_3)$, we bound this expression by:

\vspace{1mm}
$$\lesssim \int_{\tau_1+\tau_2+\tau_3+\tau_4+\tau_5+\tau_6=0} \sum_{n_1+n_2+n_3+n_4+n_5+n_6=0} \frac{1}{(N_1^*)^2}
(\theta(n_1)+\theta(n_2)+\theta(n_3))|\widetilde{v}(n_1,\tau_1)||\widetilde{\bar{v}}(n_2,\tau_2)||\widetilde{v}(n_3,\tau_3)| $$
\vspace{1mm}
$$(\theta(N_2)|(\overline{\chi_{[0,\delta]}v}_{N_2})\,\widetilde{\,}\,(n_4,\tau_4)|)
(N_3 |\widetilde{v_{N_3}}(n_5,\tau_5)|)(N_4|\widetilde{\overline{v}_{N_4}}(n_6,\tau_6)|)d\tau_j.$$
By symmetry, it suffices to consider:

$$K_{N_1,N_2,N_3,N_4}^1:=\int_{\tau_1+\tau_2+\tau_3+\tau_4+\tau_5+\tau_6=0} \sum_{n_1+n_2+n_3+n_4+n_5+n_6=0} \frac{1}{(N_1^*)^2}
\theta(n_1)|\widetilde{v}(n_1,\tau_1)||\widetilde{\bar{v}}(n_2,\tau_2)||\widetilde{v}(n_3,\tau_3)| $$
\vspace{1mm}
$$(\theta(N_2)|(\overline{\chi_{[0,\delta]}v}_{N_2})\,\widetilde{\,}\,(n_4,\tau_4)|)
(N_3 |\widetilde{v_{N_3}}(n_5,\tau_5)|)(N_4|\widetilde{\overline{v}_{N_4}}(n_6,\tau_6)|)d\tau_j=$$
\vspace{2mm}

We now use an $L^4_{t,x},L^4_{t,x},L^4_{t,x},L^4_{t,x},L^{\infty}_{t,x},L^{\infty}_{t,x}$ H\"{o}lder's inequality, and argue as in previous sections to deduce that this term is:

$$\lesssim \frac{1}{(N_1^*)^2}\|\mathcal{D}v\|_{X^{0,\frac{1}{2}+}} \|v\|_{X^{\frac{1}{2}+,\frac{1}{2}+}}^2
\|\mathcal{D}v\|_{X^{0,\frac{1}{2}+}}\|v\|_{X^{1,\frac{1}{2}+}}^2\leq$$
$$\leq \frac{1}{(N_1^*)^2} \|\mathcal{D}v\|_{X^{0,\frac{1}{2}+}}^2\|v\|_{X^{1,\frac{1}{2}+}}^4$$
which by the $X^{s,b}$ bounds on $v$ is:
$$\lesssim \frac{1}{(N_1^*)^2} \|\mathcal{D}\Phi\|_{L^2}^2 \|\Phi\|_{H^1}^4\lesssim \frac{1}{(N_1^*)^2} \|\mathcal{D}\Phi\|_{L^2}^2.$$
One gets the same bound for the other contributions to $K_{N_1,N_2,N_3,N_4}$ in this Case by symmetry.
\vspace{2mm}

\textbf{Case 2:} We recall that here $N_1=N_3^*$ or $N_1=N_4^*$. By symmetry, we consider the case $N_1=N_3^*$.

\vspace{2mm}

Arguing analogously as in the previous Case, we get the same bound as before. The only difference is that now, in the appropriate bound for
$M_4$, we replace $N_3^*$ by $\langle n_1+n_2+n_3 \rangle$ and we then use the inequality:

$$\langle n_1+n_2+n_3 \rangle \lesssim \langle n_1 \rangle + \langle n_2 \rangle + \langle n_3 \rangle$$
as the ``\emph{Fractional Leibniz Rule}''. So, in any case, we may conclude that:

\begin{equation}
\label{eq:Kbound}
|K_{N_1,N_2,N_3,N_4}|\lesssim \frac{1}{(N_1^*)^2} \|\mathcal{D}\Phi\|_{L^2}^2.
\end{equation}
The implied constant depends only on $(s,Energy,Mass)$.
Using (\ref{eq:Kbound}),(\ref{eq:Njstar}) and summing, it follows that:

$$|K|\lesssim \frac{1}{N^{2-}}\|\mathcal{D}\Phi\|_{L^2}^2=\frac{1}{N^{2-}}E^1(\Phi).$$
By using Lemma \ref{Lemma 4.3}, it follows that:

$$|K|\lesssim \frac{1}{N^{2-}}E^2(\Phi).$$
In an analogous way, we show that the other three terms in $E^2(u(\delta))-E^2(\Phi)$ satisfy this same bound.
The same bound holds for arbitrary $t_0$.
\end{proof}

\begin{remark}
\label{Remark 4.1}
The equation $(\ref{eq:Hartree})$ possesses solutions all of whose Sobolev norms are uniformly bounded in time.
Namely, if we take $n \in \mathbb{Z}$, and $\alpha \in \mathbb{C}$, then:

$$u(x,t):=\alpha e^{-i \widehat{V}(0)|\alpha|^2 t} e^{i(nx-n^2t)}$$
is a solution to (\ref{eq:Hartree}). Since our assumptions on $V$ imply that $\widehat{V}(0)=\int V(x) dx$ is real,
it follows that for all $s,t \in \mathbb{R}$:

$$\|u(t)\|_{H^s}=\|u(0)\|_{H^s}.$$
A similar ansatz was used to show instability of the cubic NLS on $S^1$ in Sobolev spaces of negative index in \cite{BGT1}.
\end{remark}

\begin{remark}
\label{Remark 4.2}

The same bound that we obtain for the Hartree Equation holds also for the Defocusing Cubic NLS on $S^1$ with the same proof. We formally take $V=\delta$. The cubic NLS is, however, completely integrable \cite{ZM}, so we see that the obtained bound is far from optimal. If we consider the defocusing cubic NLS on the real line, in \cite{SoSt1}, we show bounds which allow us to recover uniform bounds on the integral Sobolev norms of a solution, up to a loss of $(1+|t|)^{0+}$. The proof of this result relies on the improved Strichartz estimate and is at the moment possible only on the real line.
\end{remark}

\begin{remark}
\label{Remark 4.3}
The method of higher modified energies doesn't work for the equations we considered in Theorem \ref{Theorem 1}, i.e. if the nonlinearity is $|u|^k$ for $k \geq 2$. The reason why this is so is that the analogue of the multiplier $\Psi$ on $\Gamma_{2k}$, which we again call $\Psi$, is not pointwise bounded. Namely, if we consider the case $k=2$, we should take: $$\psi \sim \frac{(\theta(n_1))^{2s}-(\theta(n_2))^{2s}+(\theta(n_3))^{2s}-(\theta(n_4))^{2s}+(\theta(n_5))^{2s}-(\theta(n_6))^{2s}}
{|n_1|^2-|n_2|^2+|n_3|^2-|n_4|^2+|n_5|^2-|n_6|^2}.$$
Let us assume that $s$ is such that:

$$6^{2s}-2^{2s}+5^{2s}-3^{2s}+1^{2s}-7^{2s}\neq 0.$$
Then, we know that:
$$(n_1,n_2,n_3,n_4,n_5,n_6)=(6N,-2N,5N,-3N,N,-7N) \in \Gamma_6.$$
For this frequency configuration, we have:
$$(\theta(n_1))^{2s}-(\theta(n_2))^{2s}+(\theta(n_3))^{2s}-(\theta(n_4))^{2s}+(\theta(n_5))^{2s}-(\theta(n_6))^{2s}=$$
$$6^{2s}-2^{2s}+5^{2s}-3^{2s}+1^{2s}-7^{2s}\neq 0$$
and
$$|n_1|^2-|n_2|^2+|n_3|^2-|n_4|^2+|n_5|^2-|n_6|^2=36N^2-4N^2+25N^2-9N^2+N^2-49N^2=0.$$
Hence, $\Psi(n_1,n_2,n_3,n_4,n_5,n_6)$ is not well-defined. In particular, in this case, we can no longer prove a pointwise multiplier bound as in Lemma \ref{Lemma 4.2}. A similar construction can be adapted to the case $k>2$, if we just take the remaining 2k-4 frequencies to be equal to zero. We note that the phenomenon that the multiplier $\psi$ is unbounded in the case of the quintic and higher order nonlinearities is linked to the fact that the factorization property $(\ref{eq:denominator})$ no longer holds in this context.
\end{remark}

\subsection{Modification 2: Defocusing Cubic NLS with a Potential.}

Let us now consider the equation $(\ref{eq:potentialcubicnls})$.The equation (\ref{eq:potentialcubicnls}) has conserved mass as before, since $|u|^2+\lambda$ is real-valued.
On the other hand, by integrating by parts, one can check that the quantity:

$$E(u(t)):=\frac{1}{2}\int |\nabla u (x,t)|^2 dx + \frac{1}{4}\int |u(x,t)|^4 dx + \frac{1}{2}\int \lambda(x)|u(x,t)|^2 dx$$
is conserved in time. $E(u(t))$ is the conserved energy. By using H\"{o}lder's inequality and Sobolev embedding, it follows that $E$ is continuous on $H^1$.

\vspace{2mm}

We note that $E$ is not necessarily non-negative and that it doesn't give an a priori bound on $\|u(t)\|_{\dot{H}^1}$.
However, since $\lambda$ is bounded from below, we obtain:

$$\|u(t)\|_{H^1}^2 \lesssim E(u(t))+M(u(t)).$$
Hence $\|u(t)\|_{H^1}$ is uniformly bounded.

\subsubsection{Local-in-time estimates for $(\ref{eq:potentialcubicnls})$}

Let $u$ be a global solution of $(\ref{eq:potentialcubicnls})$.

\begin{proposition}
\label{Proposition 4.4}
Given $t_0 \in \mathbb{R}$, there exists a globally defined function $v:S^1 \times \mathbb{R} \rightarrow \mathbb{C}$ satisfying the properties:

\begin{equation}
\label{eq:properties of v13}
v|_{[t_0,t_0+\delta]}=u|_{[t_0,t_0+\delta]}.
\end{equation}

\begin{equation}
\label{eq:properties of v23}
\|v\|_{X^{1,\frac{1}{2}+}}\leq C(s,E(\Phi),M(\Phi))
\end{equation}

\begin{equation}
\label{eq:properties of v33}
\|\mathcal{D}v\|_{X^{0,\frac{1}{2}+}}\leq C(s,E(\Phi),M(\Phi)) \|\mathcal{D}u(t_0)\|_{L^2}.
\end{equation}
Moreover, $\delta$ and $C$ can be chosen to depend continuously on the energy and mass.
\end{proposition}

\begin{proof}
The proof is similar to the proof of Proposition \ref{Proposition 3.1} and Proposition \ref{Proposition 4.1}.
For the existence part, we argue by a fixed-point method.
Let us take $\delta \in (0,1)$, and let $f \in C^{\infty}_0(\mathbb{R})$ be such that
$f=1$ on $[0,1]$. Let $\mu(x,t):=f(t)\lambda(x).$

With notation as in Appendix B, we consider:

$$Lv:=\chi_{\delta}(t)S(t)\Phi-i \chi_{\delta}(t)\int_{0}^{\delta} S(t-t')(|v_{\delta}|^2 v_{\delta} + \mu v_{\delta})(t')dt'.$$
So:
$$\|Lv\|_{X^{s,b}} \leq c\delta^{\frac{1-2b}{2}} \|\Phi\|_{H^s} + c\delta^{\frac{1-2b}{2}}\||v_{\delta}|^2 v_{\delta}\|_{X^{s,b-1}} + c\delta^{\frac{1-2b}{2}}\|\mu v_{\delta}\|_{X^{s,b-1}}.$$
The new term that we have to estimate now is $\|\mu v_{\delta}\|_{X^{s,b-1}}.$
We argue by duality; let $c=c(n,\tau)$ be such that:
$$\sum_n \int d\tau |c(n,\tau)|^2 \leq 1.$$
By using the Fractional Leibniz Rule:
$$|\sum_n \int d\tau \langle n \rangle^s \langle \tau+n^2 \rangle^{b-1} (\mu v_{\delta})\,\widetilde{}\,(n,\tau)c(n,\tau)|\lesssim$$
\vspace{2mm}
$$\lesssim
\sum_n \int d\tau \big( \sum_{n_1+n_2=n} \int_{\tau_1+\tau_2=\tau} d\tau_j \frac{|c(n,\tau)|}{\langle \tau+n^2 \rangle^{1-b}}
\langle n_1 \rangle^s |\widetilde{\mu}(n_1,\tau_1)||\widetilde{v_{\delta}}(n_2,\tau_2)| \big)$$
\vspace{2mm}
$$+\sum_n \int d\tau \big( \sum_{n_1+n_2=n} \int_{\tau_1+\tau_2=\tau} d\tau_j \frac{|c(n,\tau)|}{\langle \tau+n^2 \rangle^{1-b}}|\widetilde{\mu}(n_1,\tau_1)|\langle n_2 \rangle^s |\widetilde{v_{\delta}}(n_2,\tau_2)| \big)=:I_1+I_2$$
Using Parseval's identity, an $L^2_{t,x},L^4_{t,x},L^4_{t,x}$ H\"{o}lder inequality, and $(\ref{eq:strichartztorus})$, arguing as in the proof of Proposition \ref{Proposition 3.1}, it follows that:
$$I_1 \lesssim \|c\|_{L^2_{\tau}l^2_n} \|\mu\|_{X^{s,\frac{3}{8}}} \|v_{\delta}\|_{X^{0,\frac{3}{8}}}\lesssim \delta^{r_0} \|v\|_{X^{0,b}} \leq \delta^{r_0}\|v\|_{X^{s,b}},$$
for some $r_0>0$. Here, we also used the smoothness of $\mu$ to deduce that $\|\mu\|_{X^{s,\frac{3}{8}}} \lesssim 1$. An analogous argument gives the same bound for $I_2$. The existence part of the proof now follows as in the proof of Proposition \ref{Proposition 3.1}.

For the uniqueness part, suppose that $v$,$w$ are two solutions of (\ref{eq:potentialcubicnls}) on the time interval $[0,\delta]$ with the same initial data and whose $H^1$ norms are uniformly bounded on this interval.
By using Minkowski's inequality and unitarity of the Schr\"{o}dinger operator, we deduce that, for all $t \in [0,\delta]:$

$$\|v(t)-w(t)\|_{L^2} \leq \int_0^{\delta} (\|(|v|^2v-|w|^2w)(t')\|_{L^2}+\|(\lambda v - \lambda w)(t')\|_{L^2})dt'$$
\vspace{2mm}
$$\lesssim \int_0^{\delta} ((\|v\|_{H^1}+\|w\|_{H^1})^2 + \|\lambda\|_{L^{\infty}})\|v-w\|_{L^2}dt'$$
\vspace{2mm}
$$\lesssim \int_0^{\delta} \|v(t')-w(t')\|_{L^2}dt'.$$
Uniqueness now follows from Gronwall's inequality.
\end{proof}

\subsubsection{Definition of $E^2(u)$ for $(\ref{eq:potentialcubicnls})$}

As in the case of the Hartree Equation, we will use \emph{higher modified energies}. Let:

$$E^1(u):=\|\mathcal{D}u\|_{L^2}^2,E^2(u):=E^1(u)+\lambda_4(M_4;u)$$
As before, we have to determine the multiplier $M_4$, so that we cancel the quadrilinear terms in $\frac{d}{dt}E^2(u(t))$.
We note:

$$\frac{d}{dt}E^1(u(t))\sim \frac{d}{dt}\big( \sum_{n_1+n_2=0}\widehat{\mathcal{D}u}(n_1)\widehat{\mathcal{D}\bar{u}}(n_2)\big)=$$
\vspace{2mm}
$$=\frac{1}{2}i \lambda_4((\theta(n_1))^2-(\theta(n_2))^2+(\theta(n_3))^2-(\theta(n_4))^2;u)$$

\begin{equation}
\label{eq:potentialfirstcontribution}
+i \sum_{n_1+n_2+n_3=0}((\theta(n_1))^2-(\theta(n_2))^2)\widehat{u}(n_1)\widehat{\bar{u}}(n_2)\widehat{\lambda}(n_3)
\end{equation}

On the other hand, we compute that:

$$\frac{d}{dt} \lambda_4(M_4;u)= i \lambda_4(M_4(-n_1^2+n_2^2-n_3^2+n_4^2);u)
-i\lambda_6(M_6;u)$$
\vspace{2mm}
$$-i \sum_{n_1+n_2+n_3+n_4=0}M_4\big((\lambda u)\,\widehat{}\,(n_1) \widehat{\bar{u}}(n_2)\widehat{u}(n_3)\widehat{\bar{u}}(n_4)
-\widehat{u}(n_1) (\lambda \bar{u})\,\widehat{}\,(n_2)\widehat{u}(n_3)\widehat{\bar{u}}(n_4)$$
\vspace{2mm}
\begin{equation}
\label{eq:potentialsecondcontribution}
+\widehat{u}(n_1) \widehat{\bar{u}}(n_2)(\lambda u)\,\widehat{}\,(n_3)\widehat{\bar{u}}(n_4)-\widehat{u}\,(n_1) \widehat{\bar{u}}(n_2)\widehat{u}(n_3)(\lambda \bar{u})\,\widehat{}\,(n_4) \big)
\end{equation}
Here:

$$M_6(n_1,n_2,n_3,n_4):=M_4(n_{123},n_4,n_5,n_6)-M_4(n_1,n_{234},n_5,n_6)$$
\begin{equation}
\label{eq:potentialM6}
+M_4(n_1,n_2,n_{345},n_6)-M_4(n_1,n_2,n_3,n_{456}).
\end{equation}

From $(\ref{eq:potentialfirstcontribution})$ and $(\ref{eq:potentialsecondcontribution})$, it follows that we have to choose:

\begin{equation}
\label{eq:definitionofM4potential}
M_4:= \Psi_2.
\end{equation}
where $\Psi_2$ is defined by:

$$\Psi_2: \Gamma_4 \rightarrow \mathbb{R}$$

\begin{equation}
\label{eq:definitionofpsipotential}
\Psi_2:=
\begin{cases}
  c\frac{(\theta(n_1))^2-(\theta(n_2))^2+(\theta(n_3))^2-(\theta(n_4))^2}
  {n_1^2-n_2^2+n_3^2-n_4^2} ,\,   \mbox{if }n_1^2-n_2^2+n_3^2-n_4^2 \neq 0 \\
  \,0,\,  \mbox{otherwise. }
\end{cases}
\end{equation}
for an appropriate real constant $c$.

Hence, for such a choice of $M_4$, we obtain:

$$\frac{d}{dt} E^2(u)=ci \sum_{n_1+n_2+n_3=0}((\theta(n_1))^2-(\theta(n_2))^2)\widehat{u}(n_1)\widehat{\bar{u}}(n_2)\widehat{\lambda}(n_3)$$
\vspace{2mm}
$$-i\lambda_6(M_6;u)-i \sum_{n_1+n_2+n_3+n_4=0}M_4\big[(\lambda u)\,\widehat{}\,(n_1) \widehat{\bar{u}}(n_2)\widehat{u}(n_3)\widehat{\bar{u}}(n_4)
-\widehat{u}(n_1) (\lambda \bar{u})\,\widehat{}\,(n_2)\widehat{u}(n_3)\widehat{\bar{u}}(n_4)$$
\vspace{2mm}
\begin{equation}
\label{eq:potentialcontribution1}
+\widehat{u}(n_1) \widehat{\bar{u}}(n_2)(\lambda u)\,\widehat{}\,(n_3)\widehat{\bar{u}}(n_4)-\widehat{u}\,(n_1) \widehat{\bar{u}}(n_2)\widehat{u}(n_3)(\lambda \bar{u})\,\widehat{}\,(n_4) \big]
\end{equation}

\vspace{2mm}

We dyadically localize the frequencies as $|n_j| \sim N_j$. As before, we define:
$N_1^* \geq N_2^* \geq N_3^* \geq N_4^*$.
The proof of Lemma \ref{Lemma 4.2} gives us that:

\begin{equation}
\label{eq:M_4potentialpointwisebound}
M_4=O(\frac{1}{(N_1^*)^2}\theta(N_1^*)\theta(N_2^*)N_3^*N_4^*).
\end{equation}

As we saw earlier, (\ref{eq:M_4potentialpointwisebound}) implies:

\begin{equation}
\label{eq:E1E2equivalencepotential}
E^2(u)\sim E^1(u).
\end{equation}

\subsubsection{Estimate on the increment of $E^2(u)$ for $(\ref{eq:potentialcubicnls})$ and proof of Theorem \ref{Theorem 3}.}

We want to estimate $E^2(u(t_0+\delta))-E^2(u(t_0))=E^2(v(t_0+\delta))-E^2(v(t_0))$.
The bound that we will prove is:

\begin{lemma}
\label{E2ip}
For all $t_0 \in \mathbb{R}$, one has:
$$|E^2(u(t_0+\delta))-E^2(u(t_0))|\lesssim \frac{1}{N^{1-}}E^2(u(t_0)).$$
\end{lemma}

Arguing as in the proof of Theorem \ref{Theorem 2}, Theorem \ref{Theorem 3} will then follow immediately from Lemma \ref{E2ip}. We now prove Lemma \ref{E2ip}.

\begin{proof}
As before, it suffices to consider $t_0=0$.
We have to consider three possible types of terms that come from integrating over $[0,\delta]$ the right hand side of $(\ref{eq:potentialcontribution1})$.

1) By a slight modification of our work on the Hartree Equation, we have:

\begin{equation}
\label{eq:6linearpotential}
|\int_0^{\delta} \lambda_6(M_6;u)dt| \lesssim \frac{1}{N^{2-}} E^2(\Phi)
\end{equation}

\vspace{3mm}

2) In order to estimate the time integral of the quadrilinear term on the right hand side of (\ref{eq:potentialcontribution1}), it suffices to estimate:

$$|\int_0^{\delta} \sum_{n_1+n_2+n_3+n_4=0}M_4(n_1,n_2,n_3,n_4)(\lambda u)\,\widehat{}\,(n_1)\widehat{\bar{u}}(n_2)
\widehat{u}(n_3)\widehat{\bar{u}}(n_4)dt|$$
Here $M_4$ is the multiplier we defined in (\ref{eq:definitionofM4potential}).
Let $v$ be as in Proposition \ref{Proposition 4.4}, and let $\mu(x,t)=f(t)\lambda(x)$ be as in the proof of Proposition \ref{Proposition 4.4}. Let $\chi=\chi(t):=\chi_{[0,\delta]}(t)$. Then, we want to estimate:

$$\big|\int_{\mathbb{R}} \sum_{n_1+n_2+n_3+n_4=0}M_4(n_1,n_2,n_3,n_4)(\chi \mu v)\,\widehat{}\,(n_1)\widehat{\bar{v}}(n_2)
\widehat{v}(n_3)\widehat{\bar{v}}(n_4)dt\big|$$
Let $N_1,N_2,N_3,N_4$ be dyadic integers. We define:

$$I_{N_1,N_2,N_3,N_4}:=$$
$$\big| \int_{\mathbb{R}} \sum_{n_1+n_2+n_3+n_4=0}M_4(n_1,n_2,n_3,n_4)\widehat{(\chi \mu v)}_{N_1}\,(n_1)\widehat{\bar{v}}_{N_2}(n_2)
\widehat{v}_{N_3}(n_3)\widehat{\bar{v}}_{N_4}(n_4)dt \big|$$
\vspace{2mm}
$$\sim \big| \sum_{n_1+n_2+n_3+n_4=0} \int_{\tau_1+\tau_2+\tau_3+\tau_4=0} d\tau_j$$
\vspace{2mm}
$$\{M_4(n_1,n_2,n_3,n_4)\widetilde{(\chi \mu v)}_{N_1}\,(n_1,\tau_1)\widetilde{\bar{v}}_{N_2}(n_2,\tau_2)
\widetilde{v}_{N_3}(n_3,\tau_3)\widetilde{\bar{v}}_{N_4}(n_4,\tau_4)\}\big|$$
\vspace{2mm}
$$\leq \sum_{n_1+n_2+n_3+n_4=0} \int_{\tau_1+\tau_2+\tau_3+\tau_4=0} d\tau_j$$
\vspace{2mm}
$$\{|M_4(n_1,n_2,n_3,n_4)||\widetilde{(\chi \mu v)}_{N_1}\,(n_1,\tau_1)||\widetilde{\bar{v}}_{N_2}(n_2,\tau_2)|
|\widetilde{v}_{N_3}(n_3,\tau_3)||\widetilde{\bar{v}}_{N_4}(n_4,\tau_4)|\}$$

\vspace{3mm}

We define the dyadic integers $N_j^*$ as before. By using the fact that we are summing over the set where $n_1+n_2+n_3+n_4=0$, and by using the definition of $M_4$, we know that:

\begin{equation}
\label{eq:Njstarpotential}
N_1^* \gtrsim N,\,N_1^* \sim N_2^*
\end{equation}

\vspace{2mm}

We will consider the case when:

$$N_1=N_1^*,N_2=N_2^*,N_3=N_3^*,N_4=N_4^*.$$
The other cases are similar. Namely, in the other cases, we use the Fractional Leibniz Rule differently, as we did in order to bound the term $K$ occurring in (\ref{eq:3zvijezde}).

From $(\ref{eq:M_4potentialpointwisebound})$, it follows that:

$$I_{N_1,N_2,N_3,N_4} \lesssim \sum_{n_1+n_2+n_3+n_4=0} \int_{\tau_1+\tau_2+\tau_3+\tau_4=0} d\tau_j \frac{1}{(N_1^*)^2}
\theta(N_1^*)\theta(N_2^*)N_3^*N_4^*$$
\vspace{2mm}
$$|\widetilde{(\chi \mu v)}_{N_1}\,(n_1,\tau_1)||\widetilde{\bar{v}}_{N_2}(n_2,\tau_2)|
|\widetilde{v}_{N_3}(n_3,\tau_3)||\widetilde{\bar{v}}_{N_4}(n_4,\tau_4)|$$
\vspace{2mm}
$$\lesssim \sum_{n_0+n_1+n_2+n_3+n_4=0} \int_{\tau_0+\tau_1+\tau_2+\tau_3+\tau_4=0} d\tau_j \frac{1}{(N_1^*)^2}
\theta(n_0+n_1)\theta(N_2^*)N_3^*N_4^*$$
\vspace{2mm}
$$|\widetilde{(\chi \mu)}\,(n_0,\tau_0)||\widetilde{v}(n_1,\tau_1)||\widetilde{\bar{v}}_{N_2}(n_2,\tau_2)|
|\widetilde{v}_{N_3}(n_3,\tau_3)||\widetilde{\bar{v}}_{N_4}(n_4,\tau_4)|$$
\vspace{2mm}
$$\lesssim \sum_{n_0+n_1+n_2+n_3+n_4=0} \int_{\tau_0+\tau_1+\tau_2+\tau_3+\tau_4=0} d\tau_j
\frac{1}{(N_1^*)^{2-}}(\theta(n_0)+\theta(n_1))$$
$$|\widetilde{(\chi \mu)}\,(n_0,\tau_0)|
|\widetilde{v}(n_1,\tau_1)|(\frac{1}{(N_2^*)^{0+}}|\widetilde{(\mathcal{D}\bar{v})}_{N_2}(n_2,\tau_2)|)
(\frac{1}{(N_3^*)^{0+}}|\widetilde{\nabla v}_{N_3}(n_3,\tau_3)|)(\frac{1}{(N_4^*)^{0+}}|\widetilde{\nabla \bar{v}}_{N_4}(n_4,\tau_4)|)$$
\vspace{2mm}
$$\lesssim I_{N_1,N_2,N_3,N_4}^{1}+I_{N_1,N_2,N_3,N_4}^{2}$$
Here:

$$I_{N_1,N_2,N_3,N_4}^{1}:=\sum_{n_0+n_1+n_2+n_3+n_4=0} \int_{\tau_0+\tau_1+\tau_2+\tau_3+\tau_4=0} d\tau_j
\frac{1}{(N_1^*)^{2-}}$$
$$|\widetilde{(\chi \mu)}\,(n_0,\tau_0)|
|\widetilde{\mathcal{D}v}(n_1,\tau_1)|(\frac{1}{(N_2^*)^{0+}}|\widetilde{(\mathcal{D}\bar{v})}_{N_2}(n_2,\tau_2)|)
(\frac{1}{(N_3^*)^{0+}}|\widetilde{\nabla v}_{N_3}(n_3,\tau_3)|)(\frac{1}{(N_4^*)^{0+}}|\widetilde{\nabla \bar{v}}_{N_4}(n_4,\tau_4)|)$$
and:

$$I_{N_1,N_2,N_3,N_4}^{2}:=\sum_{n_0+n_1+n_2+n_3+n_4=0} \int_{\tau_0+\tau_1+\tau_2+\tau_3+\tau_4=0} d\tau_j
\frac{1}{(N_1^*)^{2-}}$$
$$|\widetilde{(\chi \mathcal{D}\mu)}\,(n_0,\tau_0)|
|\widetilde{v}(n_1,\tau_1)|(\frac{1}{(N_2^*)^{0+}}|\widetilde{(\mathcal{D}\bar{v})}_{N_2}(n_2,\tau_2)|)
(\frac{1}{(N_3^*)^{0+}}|\widetilde{\nabla v}_{N_3}(n_3,\tau_3)|)(\frac{1}{(N_4^*)^{0+}}|\widetilde{\nabla \bar{v}}_{N_4}(n_4,\tau_4)|)$$
We estimate $I_{N_1,N_2,N_3,N_4}^{1}$. The expression $I_{N_1,N_2,N_3,N_4}^{2}$ is estimated analogously.

\vspace{3mm}

Suppose $F_j:j=0,1,2,3,4$ are such that:

\vspace{3mm}

$$\widetilde{F_0}=|\widetilde{(\chi \mu)}|,\widetilde{F_1}=|\widetilde{\mathcal{D}v}| ,\widetilde{F_2}=|\widetilde{(\mathcal{D}\bar{v})}_{N_2}|,\widetilde{F_3}=|\widetilde{\nabla v}_{N_3}|,
\widetilde{F_4}=|\widetilde{\nabla v}_{N_4}|$$
By Parseval's identity, and then by H\"{o}lder's inequality, we deduce:

$$I_{N_1,N_2,N_3,N_4}^{1}\lesssim \frac{1}{N_1^{2-}N_2^{0+}N_3^{0+}N_4^{0+}} \int \int F_0 F_1 \bar{F_2} F_3 \bar{F_4} dxdt$$
\vspace{2mm}
$$\leq \frac{1}{N_1^{2-}N_2^{0+}N_3^{0+}N_4^{0+}} \|F_0\|_{L^4_{t,x}} \|F_1\|_{L^4_{t,x}}\|F_2\|_{L^6_{t,x}}\|F_3\|_{L^6_{t,x}}\|F_4\|_{L^6_{t,x}}$$
By using $(\ref{eq:strichartztorus})$, $(\ref{eq:strichartztorus2})$, and the construction of the functions $F_j$, this expression is:

$$\lesssim \frac{1}{N_1^{2-}N_2^{0+}N_3^{0+}N_4^{0+}} \|F_0\|_{X^{0,\frac{3}{8}}} \|F_1\|_{X^{0,\frac{3}{8}}}\|F_2\|_{X^{0+,\frac{1}{2}+}}\|F_3\|_{X^{0+,\frac{1}{2}+}}\|F_4\|_{X^{0+,\frac{1}{2}+}}$$
\vspace{2mm}
$$=\frac{1}{N_1^{2-}} \|\chi \mu\|_{X^{0,\frac{3}{8}}} \|\mathcal{D}v\|_{X^{0,\frac{3}{8}}}
(\frac{1}{N_2^{0+}}\|\mathcal{D}v_{N_2}\|_{X^{0+,\frac{1}{2}+}})
(\frac{1}{N_3^{0+}}\|\nabla v_{N_3}\|_{X^{0+,\frac{1}{2}+}})
(\frac{1}{N_4^{0+}}\|\nabla v_{N_4}\|_{X^{0+,\frac{1}{2}+}})$$
Using Lemma \ref{Lemma 2.1}, and Proposition \ref{Proposition 4.4}, we deduce that this expression is:

$$\lesssim \frac{1}{N_1^{2-}} \|\mu\|_{X^{0,\frac{1}{2}+}} \|\mathcal{D}v\|_{X^{0,\frac{1}{2}+}}^2
\|\nabla v\|_{X^{0,\frac{1}{2}+}}^2$$
By the smoothness of $\mu$, this is:
\vspace{2mm}
$$\lesssim \frac{1}{N_1^{2-}} \|\mathcal{D} \Phi\|_{L^2}^2= \frac{1}{N_1^{2-}} E^1(\Phi)$$
\vspace{3mm}
By $(\ref{eq:E1E2equivalencepotential})$, we obtain that the above term is:

$$\lesssim \frac{1}{N_1^{2-}}E^2(\Phi).$$

\vspace{2mm}

An analogous argument shows that $I_{N_1,N_2,N_3,N_4}^2$ is bounded by the same quantity.

Hence:

$$I_{N_1,N_2,N_3,N_4}\lesssim \frac{1}{N_1^{2-}}E^2(\Phi)$$
We sum in the $N_j$ and use $(\ref{eq:Njstarpotential})$ to deduce that:

$$|\int_0^{\delta}\sum_{n_1+n_2+n_3+n_4=0}M_4\big((\lambda u)\,\widehat{}\,(n_1) \widehat{\bar{u}}(n_2)\widehat{u}(n_3)\widehat{\bar{u}}(n_4)
-\widehat{u}(n_1) (\lambda \bar{u})\,\widehat{}\,(n_2)\widehat{u}(n_3)\widehat{\bar{u}}(n_4)$$
\vspace{2mm}
$$+\widehat{u}(n_1) \widehat{\bar{u}}(n_2)(\lambda u)\,\widehat{}\,(n_3)\widehat{\bar{u}}(n_4)-(\lambda u)\,\widehat{}\,(n_1) \widehat{\bar{u}}(n_2)\widehat{u}(n_3)(\lambda \bar{u})\,\widehat{}\,(n_4) \big)dt|$$
\begin{equation}
\label{eq:4linearpotential}
\lesssim \frac{1}{N^{2-}} E^2(\Phi).
\end{equation}

\vspace{3mm}

3) We now estimate the time integral of the bilinear term on the right hand side of (\ref{eq:potentialcontribution1}).
Namely, we bound:

$$\big| \int_0^{\delta} \sum_{n_1+n_2+n_3=0} \big((\theta(n_1))^2-(\theta(n_2))^2\big) \widehat{u}(n_1) \widehat{\bar{u}}(n_2)
\widehat{\lambda}(n_3)dt \big|=$$
\vspace{2mm}
$$\big| \int_{\mathbb{R}} \sum_{n_1+n_2+n_3=0} \big((\theta(n_1))^2-(\theta(n_2))^2\big) (\chi v)\,\widehat{}\,(n_1) \widehat{\bar{v}}(n_2)
\widehat{\mu}(n_3)dt \big|\sim$$
\vspace{2mm}
$$\sim \big| \sum_{n_1+n_2+n_3=0} \int_{\tau_1+\tau_2+\tau_3=0} d\tau_j \big((\theta(n_1))^2-(\theta(n_2))^2\big) (\chi v)\,\widetilde{}\,(n_1,\tau_1) \widetilde{\bar{v}}(n_2,\tau_2)
\widetilde{\mu}(n_3,\tau_3) \big|$$
Given dyadic integers $N_1,N_2,N_3$, we define:

$$J_{N_1,N_2,N_3}:=\sum_{n_1+n_2+n_3=0} \int_{\tau_1+\tau_2+\tau_3=0} d\tau_j |(\theta(n_1))^2-(\theta(n_2))^2| |\widetilde{(\chi v)}_{N_1}(n_1,\tau_1)|| \widetilde{\bar{v}}_{N_2}(n_2,\tau_2)|
|\widetilde{\mu}_{N_3}(n_3,\tau_3)|$$
Let's order the frequencies as before to obtain:

$$N_1^* \geq N_2^* \geq N_3^*.$$
By construction of $\theta$, and by the fact that we are integrating over $n_1+n_2+n_3=0$, we again have that:

\begin{equation}
\label{eq:Njstarpotential2}
N_1^* \gtrsim N, N_1^* \sim N_2^*
\end{equation}
We now consider two cases, depending on the relationship between $N_1^*$ and $N_3$.

\vspace{3mm}

\textbf{Case 1:}\,$N_1^* \sim N_3.$

\vspace{2mm}

In this case, one has:

$$(\theta(n_1))^2-(\theta(n_2))^2=O((\theta(N_3))^2).$$
We find $G_1,G_2,G_3$ such that:

$$\widetilde{G_1}=|\widetilde{(\chi v)}_{N_1}|, \widetilde{G_2}=|\widetilde{v}_{N_2}|,\widetilde{G_3}=
|\widetilde{(\mathcal{D}^2 \mu)}_{N_3}|.$$
So:

$$J_{N_1,N_2,N_3} \lesssim $$
\vspace{2mm}
$$\sum_{n_1+n_2+n_3=0} \int_{\tau_1+\tau_2+\tau_3=0}d\tau_j| \widetilde{(\chi v)}_{N_1}(n_1,\tau_1)||\widetilde{\bar{v}}_{N_2}(n_2,\tau_2)|
(\theta(N_3))^2|\widetilde{\mu}_{N_3}(n_3,\tau_3)|$$
\vspace{2mm}
$$\lesssim \sum_{n_1+n_2+n_3=0} \int_{\tau_1+\tau_2+\tau_3=0}d\tau_j| \widetilde{(\chi v)}_{N_1}(n_1,\tau_1)||\widetilde{\bar{v}}_{N_2}(n_2,\tau_2)|
|\widetilde{(\mathcal{D}^2\mu)}_{N_3}(n_3,\tau_3)|$$
\vspace{2mm}
$$\sim \int \int G_1 \overline{G_2} G_3 dxdt \leq \|G_1\|_{L^4_{t,x}} \|G_2\|_{L^4_{t,x}}\|G_3\|_{L^2_{t,x}}$$
\vspace{2mm}
$$\lesssim \|G_1\|_{X^{0,\frac{3}{8}}}\|G_2\|_{X^{0,\frac{3}{8}}}\|G_3\|_{X^{0,0}}
=\|(\chi v)_{N_1}\|_{X^{0,\frac{3}{8}}}\|v_{N_2}\|_{X^{0,\frac{3}{8}}}\|(\mathcal{D}^2 \mu)_{N_3}\|_{X^{0,0}}$$
\vspace{2mm}
$$\lesssim \|v\|_{X^{0,\frac{1}{2}+}}^2 \frac{1}{N_3^M}\|\mathcal{D}^2 \mu\|_{X^{M,0}}\lesssim \frac{1}{(N_1^*)^M}
\|\mathcal{D}v\|_{X^{0,\frac{1}{2}+}}^2$$
The previous bound holds for all $M>0$, by the smoothness properties of $\mu$.
From $(\ref{eq:E1E2equivalencepotential})$, we have that the contribution from Case 1 is, in particular:

\begin{equation}
\label{eq:2linearpotentialcase1}
\lesssim \frac{1}{N_1^*}E^2(\Phi)
\end{equation}

\vspace{3mm}

\textbf{Case 2:}\,$N_1^* \gg N_3.$

\vspace{2mm}

\textbf{Subcase 1:}\,$N_3 \leq (N_1^*)^{\epsilon}$ ($\epsilon>0$ is small).

\vspace{2mm}

We recall from $(\ref{eq:numerator12})$ that:

$$|(\theta(x))^2-(\theta(y))^2|\leq \frac{1}{N^{2s}}||x|^{2s}-|y|^{2s}|.$$
\vspace{2mm}
Now, in this subcase:

$$|n_1|\sim |n_2|\sim N_1^*,||n_1|-|n_2||=O((N_1^*)^{\epsilon})$$
So, by the Mean Value Theorem:

$$||n_1|^{2s}-|n_2|^{2s}|\lesssim (N_1^*)^{2s-1}(N_1^*)^{\epsilon}$$
Consequently:

$$|(\theta(n_1))^2-(\theta(n_2))^2|\lesssim \frac{1}{(N_1^*)^{1-\epsilon}}\theta(N_1)\theta(N_2).$$
With notation as in Case 1, we obtain that:

$$J_{N_1,N_2,N_3}\lesssim $$
\vspace{2mm}
$$\sum_{n_1+n_2+n_3=0} \int_{\tau_1+\tau_2+\tau_3=0}d\tau_j
\frac{1}{(N_1^*)^{1-\epsilon}}|\widetilde{(\chi \mathcal{D}v)}_{N_1}(n_1,\tau_1)||\widetilde{\mathcal{D}\bar{v}}_{N_2}(n_2,\tau_2)|
|\widetilde{\mu}_{N_3}(n_3,\tau_3)|$$
We now argue, using an $L^4_{t,x},L^4_{t,x},L^2_{t,x}$ H\"{o}lder inequality as in Case 1, to deduce:
\vspace{2mm}
\begin{equation}
\label{eq:2linearpotentialcase21}
J_{N_1,N_2,N_3}\lesssim \frac{1}{(N_1^*)^{1-\epsilon}}\|\mathcal{D}v\|_{X^{0,\frac{1}{2}+}}^2 \|\mu\|_{X^{0,0}} \lesssim
\frac{1}{(N_1^*)^{1-\epsilon}}E^2(\Phi)
\end{equation}

\vspace{2mm}

\textbf{Subcase 2:}\,$N_3>(N_1^*)^{\epsilon}$(for the same $\epsilon>0$ as before).

\vspace{2mm}

In this subcase, we estimate $|(\theta(n_1))^2-(\theta(n_2))^2| \lesssim \theta(N_1) \theta(N_2)$, and hence:

$$J_{N_1,N_2,N_3}\lesssim $$
\vspace{2mm}
$$\sum_{n_1+n_2+n_3=0} \int_{\tau_1+\tau_2+\tau_3=0}d\tau_j |\widetilde{(\chi \mathcal{D}v)}_{N_1}(n_1,\tau_1)||\widetilde{\mathcal{D}\bar{v}}_{N_2}(n_2,\tau_2)|
|\widetilde{\mu}_{N_3}(n_3,\tau_3)|$$
We now argue similarly as in Case 1 to deduce that for all $M>0$:

$$J_{N_1,N_2,N_3}\lesssim \frac{1}{(N_1^*)^{\epsilon M}} \|\mathcal{D}v\|_{X^{0,\frac{1}{2}+}}^2.$$
Hence, if we choose $M$ sufficiently large so that $\epsilon M \geq 1$, we obtain:

\begin{equation}
\label{eq:2linearpotentialcase22}
J_{N_1,N_2,N_3}\lesssim \frac{1}{N_1^*}E^2(\Phi).
\end{equation}

\vspace{2mm}

Combining $(\ref{eq:2linearpotentialcase1})$, $(\ref{eq:2linearpotentialcase21})$, $(\ref{eq:2linearpotentialcase22})$, and summing in the $N_j$, we obtain that:

\vspace{2mm}

\begin{equation}
\label{eq:2linearpotential}
|\int_0^{\delta}\sum_{n_1+n_2+n_3=0}((\theta(n_1))^2-(\theta(n_2))^2)\widehat{u}(n_1)\widehat{\bar{u}}(n_2)\widehat{\lambda}(n_3) dt|\lesssim \frac{1}{N^{1-}}E^2(\Phi).
\end{equation}

\vspace{2mm}

Lemma \ref{E2ip} now follows from $(\ref{eq:6linearpotential})$, $(\ref{eq:4linearpotential})$, and $(\ref{eq:2linearpotential})$.

\end{proof}

\begin{remark}
\label{Remark 4.4}
If $\lambda$ is a constant function, then $u=e^{-i\lambda t}v$, where $v$ is a solution to the cubic $NLS$. Since $\|v(t)\|_{H^s}$ is then uniformly bounded in time, the same holds for $\|u(t)\|_{H^s}$.
If $\lambda$ depends on $x$, one can't argue in this way.
\end{remark}

\begin{remark}
\label{Remark 4.5}
Heuristically, the reason why we get a weaker bound for $(\ref{eq:potentialcubicnls})$ than we did for $(\ref{eq:Hartree})$ is the fact that we have bilinear terms which occur in $\frac{d}{dt}E^2(u)$. Hence, the derivatives have to be distributed among fewer factors of $u$ and $\bar{u}$ than there were before.
\end{remark}

\subsection{Modification 3: Defocusing Cubic NLS with Inhomogeneous Nonlinearity.}

\vspace{2mm}

We now consider the equation $(\ref{eq:inhomogeneouscubicnls})$. The equation $(\ref{eq:inhomogeneouscubicnls})$ has conserved mass. By integration by parts, one can check that energy:

$$E(u(t)):=\frac{1}{2} \int |\nabla u(x,t)|^2 dx + \frac{1}{4} \int \lambda(x) |u(x,t)|^4 dx$$
is conserved in time. Both quantities are continuous on $H^1$. Since $\lambda \geq 0$, conservation of mass and energy gives us uniform bounds on $\|u(t)\|_{H^1}$.

\vspace{2mm}

\subsubsection{Local-in-time estimates for $(\ref{eq:inhomogeneouscubicnls})$}

Let $u$ be a global solution to $(\ref{eq:inhomogeneouscubicnls})$. Let us observe the following fact:

\begin{proposition}
\label{Proposition 4.5}
Given $t_0 \in \mathbb{R}$, there exists a globally defined function $v:S^1 \times \mathbb{R} \rightarrow \mathbb{C}$ satisfying the properties:

\begin{equation}
\label{eq:properties of v14}
v|_{[t_0,t_0+\delta]}=u|_{[t_0,t_0+\delta]}.
\end{equation}

\begin{equation}
\label{eq:properties of v24}
\|v\|_{X^{1,\frac{1}{2}+}}\leq C(s,E(\Phi),M(\Phi))
\end{equation}

\begin{equation}
\label{eq:properties of v34}
\|\mathcal{D}v\|_{X^{0,\frac{1}{2}+}}\leq C(s,E(\Phi),M(\Phi)) \|\mathcal{D}u(t_0)\|_{L^2}.
\end{equation}
Moreover, $\delta$ and $C$ can be chosen to depend continuously on the energy and mass.
\end{proposition}
The proof of Proposition \ref{Proposition 4.5} is analogous to the proof of Proposition \ref{Proposition 3.1}, so we omit the details.

\vspace{2mm}

\subsubsection{Estimate on the increment of $E^1(u)$ for $(\ref{eq:inhomogeneouscubicnls})$ and proof of Theorem \ref{Theorem 4}}

The presence of the inhomogeneity $\lambda$ in the nonlinearity makes it impossible to use $E^2$, as in the case of the previous two equations. The difficulty lies in the fact that the numerators we obtain in the correction terms no longer factorize, so we can't obtain bounds such as $(\ref{eq:psipointwisebound})$. This is analogous to the situation that occurs for the quintic and higher order NLS. For details, see Remark \ref{Remark 4.3}. Hence, we have to work with $E^1$. Theorem \ref{Theorem 4} will follow if we prove that:

\begin{lemma}
\label{inhomogeneousE1increment}
For all $t_0 \in \mathbb{R}$, one has:
$$|E^1(u(t_0+\delta))-E^1(t_0)|\lesssim \frac{1}{N^{\frac{1}{2}-}}E^1(\Phi)$$
\end{lemma}

\begin{proof}

As before, it suffices to consider the case $t_0=0$.
Arguing as in previous sections, we obtain:

$$\frac{d}{dt}E^1(u)=ci \sum_{n_0+n_1+\cdots+n_4=0} \big((\theta(n_1))^2-(\theta(n_2))^2+(\theta(n_3))^2-(\theta(n_4))^2 \big) \widehat{\lambda}(n_0)\widehat{u}(n_1)\widehat{\bar{u}}(n_2)\widehat{u}(n_3)\widehat{\bar{u}}(n_4)$$
Let $N_0,N_1,\ldots,N_4$ be dyadic integers. We define $\mu(x,t):=f(t)\lambda(x)$ as in the proof of Proposition \ref{Proposition 4.4}. The expression we want to estimate is:
$$I_{N_0,N_1,N_2,N_3,N_4}:= \big| \int_0^{\delta} \sum_{n_0+n_1+\cdots+n_4=0}\big((\theta(n_1))^2-(\theta(n_2))^2+(\theta(n_3))^2-(\theta(n_4))^2\big)$$
\vspace{2mm}
$$\widehat{\mu}_{N_0}(n_0)\widehat{v}_{N_1}(n_1)\widehat{\bar{v}}_{N_2}(n_2)\widehat{v}_{N_3}(n_3)
\widehat{\bar{v}}_{N_4}(n_4)dt \big|$$
If $\chi=\chi(t)=\chi_{[0,\delta]}(t)$, then $I_{N_0,N_1,N_2,N_3,N_4}$ is:
$$\lesssim \big| \sum_{n_0+n_1+\cdots+n_4=0}\int_{\tau_0+\tau_1+\cdots \tau_4=0} d\tau_j
\big((\theta(n_1))^2-(\theta(n_2))^2+(\theta(n_3))^2-(\theta(n_4))^2\big)$$
\vspace{2mm}
$$\widetilde{\mu}_{N_0}(n_0,\tau_0)\widetilde{(\chi v)}_{N_1}(n_1,\tau_1)\widetilde{\bar{v}}_{N_2}(n_2,\tau_2)\widetilde{v}_{N_3}(n_3,\tau_3)
\widetilde{\bar{v}}_{N_4}(n_4,\tau_4) \big|$$
We define $N_j^*$ for $j=1,\ldots 5$ to be the ordering of $\{N_0,N_1,N_2,N_3,N_4\}$. With this notation, we have the following bounds:
\begin{equation}
\label{eq:Njstarpotential3}
N_1^* \gtrsim N, N_1^* \sim N_2^*
\end{equation}
We consider two cases:

\vspace{3mm}

\textbf{Case 1:}\,$N_0 \gtrsim (N_1^*)^{\epsilon}$ (Here $\epsilon>0$ is small.)

\vspace{2mm}

We use the fact that the multiplier is $O((\theta(N_1^*))^2)$, and an $L^{\infty}_{t,x}, L^4_{t,x}, L^4_{t,x}, L^4_{t,x}, L^4_{t,x}$ H\"{o}lder's inequality to deduce that:
$$I_{N_0,N_1,N_2,N_3,N_4} \lesssim (\theta(N_1^*))^2 \|\mu_{N_0}\|_{X^{\frac{1}{2}+,\frac{1}{2}+}}\|v_{N_1}\|_{X^{0,\frac{1}{2}+}}
\|v_{N_2}\|_{X^{0,\frac{1}{2}+}}\|v_{N_3}\|_{X^{0,\frac{1}{2}+}}\|v_{N_4}\|_{X^{0,\frac{1}{2}+}}.$$
Considering separately the cases when $N_0 \sim N_1^*$ and when $N_0 \ll N_1^*$, this expression is:
\vspace{2mm}
$$\lesssim \|\mathcal{D}\mu_{N_0}\|_{X^{\frac{1}{2}+, \frac{1}{2}+}}\|\mathcal{D}v\|_{X^{0,\frac{1}{2}+}}\|v\|_{X^{0,\frac{1}{2}+}}^3
+ \|\mu_{N_0}\|_{X^{\frac{1}{2}+,\frac{1}{2}+}}\|\mathcal{D}v\|_{X^{0,\frac{1}{2}+}}^2\|v\|_{X^{0,\frac{1}{2}+}}^2$$
\vspace{2mm}
$$\lesssim (\|\mathcal{D}\mu_{N_0}\|_{X^{\frac{1}{2}+,\frac{1}{2}+}}+\|\mu_{N_0}\|_{X^{\frac{1}{2}+,\frac{1}{2}+}})
\|\mathcal{D}v\|_{X^{0,\frac{1}{2}+}}^2\|v\|_{X^{0,\frac{1}{2}+}}^2$$
For $M>0$, this quantity is:
\vspace{2mm}
$$\lesssim \frac{1}{(N_0)^M}(\|\mathcal{D}\mu_{N_0}\|_{X^{M+\frac{1}{2}+,\frac{1}{2}+}}+\|\mu_{N_0}\|_{X^{M+\frac{1}{2}+,\frac{1}{2}+}}) \|\mathcal{D}v\|_{X^{0,\frac{1}{2}+}}^2 \|v\|_{X^{0,\frac{1}{2}+}}^2$$
\vspace{2mm}
$$\lesssim \frac{1}{(N_1^*)^{\epsilon M}} \|\mathcal{D} \Phi\|_{L^2}^2= \frac{1}{(N_1^*)^{\epsilon M}} E^1(\Phi).$$
In particular, if we choose $M$ sufficiently large so that $\epsilon M \geq \frac{1}{2}$, we get:
\begin{equation}
\label{eq:inhomogeneouscase1}
I_{N_0,N_1,N_2,N_3,N_4} \lesssim \frac{1}{(N_1^*)^{\frac{1}{2}}} E^1(\Phi).
\end{equation}

\vspace{2mm}

\textbf{Case 2:}\,$N_0 \ll (N_1^*)^{\epsilon}$ (for the same $\epsilon$ as before)

\vspace{2mm}

If we take $\epsilon<\frac{1}{2}$, we note that the same arguments we used to prove Theorem \ref{Theorem 1} allow us to deduce that in Case 2:

\begin{equation}
\label{eq:inhomogeneouscase2}
I_{N_0,N_1,N_2,N_3,N_4}\lesssim \frac{1}{(N_1^*)^{\frac{1}{2}}}E^1(\Phi).
\end{equation}
More precisely, we recall the proof of Lemma \ref{Bigstar1}. The only place in which one can't immediately adapt the proof of Lemma \ref{Bigstar1} to $(\ref{eq:inhomogeneouscubicnls})$ is in Case 1 of Big Case 1. If one has the additional assumption that $N_0 \ll (N_1^*)^{\frac{1}{2}}$, the proof then follows as before.

Using $(\ref{eq:inhomogeneouscase1})$, $(\ref{eq:inhomogeneouscase2})$, and summing in the $N_j$, the Lemma follows.
\end{proof}

\vspace{3mm}

\begin{remark}
\label{Remark 4.6}
If $\lambda$ is constant, we can obtain $(\ref{eq:inhomogeneouscubicnls})$ by rescaling the cubic NLS, so Theorem \ref{Theorem 4} can be improved in this case.
\end{remark}

\subsection{Comments on $(\ref{eq:Hartree})$, $(\ref{eq:potentialcubicnls})$, and $(\ref{eq:inhomogeneouscubicnls})$}

The reason why we considered the three equations in this section was because they were obtained from the cubic NLS by breaking the complete integrability. Different ways of breaking the complete integrability of the cubic NLS manifested themselves in the bounds we obtained, and the methods we could use to obtain them. As we saw, the least drastic change happened when we added the convolution potential in the case of the Hartree Equation, whereas the most drastic change happened when we multiplied the nonlinearity with the inhomogeneity in $(\ref{eq:inhomogeneouscubicnls})$.

\section{Appendix A: Proof of Lemma 2.1.}

\begin{proof}
We argue by duality.
Let us consider $v$ s.t.$\|v\|_{X^{-s,-b}}\leq 1$. We want to prove that:

\begin{equation}
\label{eq:dualXsb}
|\int_c^d\int_{S^1} u(x,t) \overline{v(x,t)} dxdt| \lesssim \|u\|_{X^{s,b+}} \|v\|_{X^{-s,-b}}.
\end{equation}
We observe: $$\chi_{[c,d]}(t)=\frac{sign(t-c)-sign(t-d)}{2}.$$
By symmetry, we just need to get the bound:

\begin{equation}
\label{eq:signc}
|\int_{\mathbb{R}}\int_{S^1} sign(t-c)u(x,t) \overline{v(x,t)} dxdt| \lesssim \|u\|_{X^{s,b+}} \|v\|_{X^{-s,-b}}.
\end{equation}
Let us first prove, the claim when $c=0$, i.e.

\begin{equation}
\label{eq:sign}
|\int_{\mathbb{R}}\int_{S^1} sign(t)u(x,t) \overline{v(x,t)} dxdt| \lesssim \|u\|_{X^{s,b+}} \|v\|_{X^{-s,-b}}.
\end{equation}
The key to prove (\ref{eq:sign}) is to use the Hilbert transform in the time variable.

\vspace{2mm}

We recall that the Hilbert transform on the real line is defined by:

\begin{equation}
\label{eq:Hilberttransform1}
Hf:=cf *(p.v. \frac{1}{x}).
\end{equation}
The constant $c$ is chosen so that $H$ is an isometry on $L^2$.
It can be shown \cite{D}, that one then has the identity:

\begin{equation}
\label{eq:Hilberttransform2}
\widehat{Hf}(\xi) \sim -isign(\xi)\hat{f}(\xi).
\end{equation}
From Parseval's identity and from (\ref{eq:Hilberttransform1}),(\ref{eq:Hilberttransform2}), we obtain:

$$\int_{\mathbb{R}}\int_{S^1} sign(t)u(x,t) \overline{v(x,t)} dxdt \sim
\int_{\mathbb{R}} \sum_n \widetilde{u}(n,\tau)(p.v.\int_{\mathbb{R}}
\overline{\widetilde{v}(n,\tau')}\frac{1}{\tau-\tau'}d \tau')d \tau=:J.$$
\vspace{2mm}
Let us consider three cases:

\vspace{3mm}

\textbf{Case 1:} $\langle \tau + n^2 \rangle \sim \langle \tau' + n^2 \rangle.$

\vspace{1mm}

\textbf{Case 2:} $\langle \tau + n^2 \rangle \gg \langle \tau' + n^2 \rangle.$

\vspace{1mm}

\textbf{Case 3:} $\langle \tau + n^2 \rangle \ll \langle \tau' + n^2 \rangle.$

\vspace{2mm}
Let $J_1,J_2,J_3$ denote the contributions to $J$ coming from the three cases respectively. We estimate these contributions separately.

\vspace{3mm}

\textbf{Case 1:} In this case, we perform a dyadic decomposition.
Let $\widetilde{u_k}$,$\widetilde{v_k}$ respectively denote the localizations of $\widetilde{u}$, $\widetilde{v}$ to $\langle \tau + n^2 \rangle \sim \langle \tau' + n^2 \rangle \sim 2^k$. Then, since in this case $|j-k|=O(1)$, we get:

$$|J_1|= |\sum_{|j-k|=O(1)} \int_{\mathbb{R}}\, \sum_n
\widetilde{u_j}(n,\tau)(p.v.\int_{\mathbb{R}}
\overline{\widetilde{v_k}(n,\tau')}\frac{1}{\tau-\tau'}d \tau')d \tau| \sim$$
$$\sim |\sum_{|j-k|=O(1)} \int_{\mathbb{R}}\, \sum_n
\widetilde{u_j}(n,\tau)H_{\tau}\overline{\widetilde{v_k}}(n,\tau)d \tau|$$
$$\leq \sum_{|j-k|=O(1)} |\int_{\mathbb{R}}\, \sum_n
\langle n \rangle ^s \widetilde{u_j}(n,\tau)\langle n \rangle ^{-s} H_{\tau} \overline{\widetilde{v_k}}(n,\tau)d \tau|.$$
Here, we denoted by $H_{\tau}(\cdot)$ the Hilbert transform in the $\tau$ variable. We then use the Cauchy-Schwarz inequality in $(n,\tau)$ to see that the previous expression is:

$$\leq \sum_{|j-k|=O(1)} \|\langle n \rangle^s\widetilde{u_j}\|_{l^2_n L^2_{\tau}} \|\langle n \rangle^{-s}H_{\tau}\overline{\widetilde{v_k}}\|_{l^2_n L^2_{\tau}}.$$
We then recall that the Hilbert transform is bounded on $L^2$ by (\ref{eq:Hilberttransform2}) to deduce that:

$$|J_1| \lesssim \sum_{|j-k|=O(1)} \|\langle n \rangle^s\widetilde{u_j}\|_{l^2_n L^2_{\tau}} \|\langle n \rangle^{-s}\overline{\widetilde{v_k}}\|_{l^2_n L^2_{\tau}}$$
\vspace{1mm}
Since $|j-k|=O(1)$, and by definition of $u_j,v_k$, this is:

\vspace{1mm}

$$\lesssim \sum_{|j-k|=O(1)} \|\langle n \rangle^s \langle \tau + n^2 \rangle^b \widetilde{u_j}\|_{l^2_n L^2_{\tau}} \|\langle n \rangle^{-s} \langle \tau + n^2 \rangle^{-b} \overline{\widetilde{v_k}}\|_{l^2_n L^2_{\tau}}$$
\vspace{1mm}
We use the Cauchy-Schwarz inequality in the sum of $j,k$ to bound this by:

\vspace{1mm}

$$\lesssim \|u\|_{X^{s,b}} \|v\|_{X^{-s,-b}}\leq \|u\|_{X^{s,b+}} \|v\|_{X^{-s,-b}}.$$

\vspace{3mm}

\textbf{Case 2:}
\vspace{2mm}
Since in this case $\langle \tau + n^2 \rangle \gg \langle \tau' + n^2 \rangle$, we have that:

$$|\tau - \tau'| \sim \langle \tau + n^2 \rangle \gg \langle \tau' + n^2 \rangle.$$
It follows that for all $\theta \in [0,1]$, one has:
$$\frac{1}{|\tau - \tau'|} \lesssim \frac{1}{ \langle \tau + n^2 \rangle^{\theta} \langle \tau' + n^2 \rangle^{1-\theta}}.$$
We deduce that:

$$|J_2|\lesssim \int_{\mathbb{R}}\int_{\mathbb{R}} \sum_n |\widetilde{u}(n,\tau)||\widetilde{v}(n,\tau')|\frac{1}{ \langle \tau + n^2 \rangle^{\theta} \langle \tau' + n^2 \rangle^{1-\theta}} d\tau d\tau'=$$

$$=\sum_n (\int_{\mathbb{R}} |\widetilde{u}(n,\tau)| \langle \tau + n^2 \rangle^{\frac{1}{2}+\delta-\theta}\langle \tau + n^2 \rangle^{-\frac{1}{2}-\delta}
\langle n \rangle^s d\tau)
(\int_{\mathbb{R}} |\widetilde{v}(n,\tau')| \langle \tau' + n^2 \rangle^{\frac{1}{2}+\delta-(1-\theta)}\langle \tau' + n^2 \rangle^{-\frac{1}{2}-\delta}
\langle n \rangle^{-s} d\tau').$$
Here, $\delta>0$ was arbitrary.
Now, we first use the Cauchy-Schwarz inequality in $\tau, \tau'$, together with the fact that:
$$\|\langle \tau + n^2\rangle^{-\frac{1}{2}-\delta}\|_{l^{\infty}_n L^2_{\tau}} \lesssim 1$$
and
$$\|\langle \tau' + n^2\rangle^{-\frac{1}{2}-\delta}\|_{l^{\infty}_n L^2_{\tau'}} \lesssim 1$$
followed by the Cauchy-Schwarz inequality in $n$ to deduce that:

$$|J_2|\lesssim \|u\|_{X^{s,\frac{1}{2}+ \delta -\theta}}\|v\|_{X^{-s,\frac{1}{2}+ \delta -(1-\theta)}}.$$
Let us take $\delta>0$ sufficiently small so that $b + \delta <\frac{1}{2}$. We then take:
$\theta:=\frac{1}{2}-b-\delta$, which is positive, and $b+:=b+2\delta$. With such a choice, we get that:
$$|J_2|\lesssim \|u\|_{X^{s,b+}}\|v\|_{X^{-s,-b}}.$$

\textbf{Case 3:}

In this case, we again have: $|\tau- \tau'| \gtrsim \langle \tau + n^2 \rangle, \langle \tau' + n^2 \rangle,$
and we argue to get the same bound as in the previous case.

\vspace{2mm}
The bound (\ref{eq:sign}) now follows.

\vspace{2mm}
Let us now observe that the bound (\ref{eq:sign}) implies (\ref{eq:signc}).

\vspace{3mm}
Let $M_a$ denote the \emph{modulation operator} $$M_af (x)=e^{i a x}f(x).$$
Then, one obtains that:

$$(M_a H M_{-a}f)\,\,\widehat{}\,\,(\xi) \sim -i sign(\xi-a) \widehat{f}(\xi).$$
Let $\Phi^{-1}(\cdot)$ denote the inverse spacetime Fourier transform. Then, by Parseval's Identity, we obtain:

$$\int_{\mathbb{R}}\int_{S^1}sign(t-c)u(x,t)\overline{v(x,t)} dxdt \sim$$
$$\sim \int_{\mathbb{R}}\sum_n (\Phi^{-1}u) (n,\tau) \overline{M_cHM_{-c}(\Phi^{-1}v)(n,\tau)}d\tau \sim$$
$$\sim \int_{\mathbb{R}}\sum_n (\Phi^{-1}u) (n,\tau) e^{-i c \tau} p.v. ( \int_{\mathbb{R}} \frac{e^{i c \tau'} \overline{(\Phi^{-1}v)(n,\tau')}}{\tau-\tau'}d\tau')d\tau \sim$$
$$\sim \int_{\mathbb{R}}\sum_n \widetilde{u}(n,\tau) e^{i c \tau} p.v. ( \int_{\mathbb{R}} \frac{e^{-i c \tau'} \overline{\widetilde{v}(n,\tau')}}{\tau-\tau'}d\tau')d\tau.$$
In the last step, we use the Fourier inversion formula which gives us that:

$$\Phi^{-1}w(n,\tau)\sim \widetilde{w}(-n,-\tau).$$
Multiplication by the unimodular factors $e^{i c \tau},e^{i c \tau'}$ doesn't change the rest of the argument used to derive (\ref{eq:sign}). Hence, the proof of (\ref{eq:signc}) follows as before.

\vspace{1mm}

\end{proof}

\begin{remark}
We deduce from the proof that none of the implied constants depend on $c$ and $d$.
\end{remark}

\section{Appendix B: Proofs of Propositions \ref{Proposition 3.1}, \ref{Proposition 3.2}, and \ref{Proposition 3.3}:}

In order to prove Proposition \ref{Proposition 3.1}, we recall several facts. One of the key ingredients of the proof is the following set of localization estimates in $X^{s,b}$ spaces. We start with $f \in
C^{\infty}_0(\mathbb{R}),\delta>0$ arbitrary, and we assume that
$b>\frac{1}{2}$. Let $S(t)$ denote the linear Schr\"{o}dinger
propagator. Then, there exists a constant $C>0$ depending only on
$f,s,b$ such that:

\begin{equation}
\label{eq:loc1} \|f(\frac{t}{\delta})S(t)\Phi\|_{X^{s,b}}\leq C
\delta^{\frac{1-2b}{2}} \|\Phi\|_{H^s}.
\end{equation}

\begin{equation}
\label{eq:loc2}
\|f(\frac{t}{\delta})h\|_{X^{s,b}}\leq C
\delta^{\frac{1-2b}{2}}\|h\|_{X^{s,b}}.
\end{equation}

\begin{equation}
\label{eq:loc3}
\|f(\frac{t}{\delta})\int_{0}^{t}S(t-t')w(t')d
t'\|_{X^{s,b}}\leq C \delta^{\frac{1-2b}{2}}\|w\|_{X^{s,b-1}}.
\end{equation}

The analogous fact is proved for the $X^{s,b}$ spaces corresponding to the
Korteweg-de Vries equation in \cite{KPV2} in the non-periodic case.
However, all the bounds for the periodic Schr\"{o}dinger equation
follow in the same way, because we are estimating the integral in the
variable dual to time. These bounds for our equation can also be found
in \cite{CafE}. We also note that in $(\ref{eq:loc3})$, we can
translate time so that our initial time is arbitrary $t_0$ and not
necessarily 0.

If, on the other hand $b'<\frac{1}{2}$,one has:

\begin{equation}
\label{eq:loc4}
\|f(\frac{t}{\delta})w\|_{X^{s,b'}}\lesssim_f \|w\|_{X^{s,b'}}.
\end{equation}
We observe that the implied constant is independent of $\delta>0$.

For the proof of the inequality (\ref{eq:loc4}), one should consult Lemma 1.2. in \cite{G}. We note that the proof
from the paper holds if $b=b'$ in the given notation. One can also refer to Lemma 2.11 in \cite{Tao}

\begin{proof} (\emph{of Proposition \ref{Proposition 3.1}})

Let us WLOG assume that $t_0=0$ for simplicity of notation. Later, we will see that the $\delta$ we obtain is indeed independent of time. Throughout the proof, we take $\delta>0$ small which we will determine later. Let $b=\frac{1}{2}+=\frac{1}{2}+\epsilon$ for $\epsilon$
sufficiently small which we also determine later.

Let us start by taking $\chi,\phi,\psi \in C^{\infty}_0(\mathbb{R})$,
with $0 \leq \chi,\phi,\psi \leq 1$, such that:

\begin{equation}
\label{eq:chi} \chi=1\,\, \mbox{on}\,\, [-1,1]\,\,,\chi=0
\,\,\mbox{outside}\,\, [-2,2].
\end{equation}
\begin{equation}
\label{eq:phi} \phi=1\,\, \mbox{on}\,\,  [-2,2]\,\,,\phi=0\,\,
\mbox{on}\,\, [-4,4].
\end{equation}
\begin{equation}
\label{eq:psi} \psi=1\,\, \mbox{on}\,\,  [-4,4]\,\,,\psi=0\,\,
\mbox{on}\,\, [-8,8].
\end{equation}
We let:
\begin{equation}
\label{eq:localizedbumpfunctions}
\chi_{\delta}:=\chi(\frac{\cdot}{\delta}),\phi_{\delta}:=\phi(\frac{\cdot}{\delta}),\psi_{\delta}
:=\psi(\frac{\cdot}{\delta}).
\end{equation}
Then:
\begin{equation}
\label{eq:chidelta} \chi_{\delta}=1\,\, \mbox{on}\,\,
[-\delta,\delta]\,\,,\chi_{\delta}=0 \,\,\mbox{outside}\,\,
[-2\delta,2\delta].
\end{equation}
\begin{equation}
\label{eq:phidelta} \phi_{\delta}=1\,\, \mbox{on}\,\,
[-2\delta,2\delta]\,\,,\phi_{\delta}=0\,\, \mbox{outside}\,\,
[-4\delta,4\delta].
\end{equation}
\begin{equation}
\label{eq:psidelta} \psi_{\delta}=1\,\, \mbox{on}\,\,
[-4\delta,4\delta]\,\,,\psi_{\delta}=0\,\, \mbox{outside}\,\,
[-8\delta,8\delta].
\end{equation}
For $v:S^1 \times \mathbb{R} \mapsto \mathbb{C}$, we define:
$$Lv:=\chi_{\delta}(t)S(t)\Phi-i\chi_{\delta}(t)\int_{0}^{t}S(t-t')|v|^4v(t')dt'.$$
By $(\ref{eq:chidelta})$ and $(\ref{eq:phidelta})$, and denoting
$\phi_{\delta}v$ by $v_{\delta}$,we obtain:
$$Lv=\chi_{\delta}(t)S(t)\Phi-i\chi_{\delta}(t)\int_{0}^{t}S(t-t')|v_{\delta}|^4v_{\delta}(t')dt.'$$
Using $(\ref{eq:loc1})$ and $(\ref{eq:loc3})$, we obtain:
\begin{equation}
\label{eq:boundonL}
\|Lv\|_{X^{s,b}}\leq c
\delta^{\frac{1-2b}{2}}\|\Phi\|_{H^s}+c
\delta^{\frac{1-2b}{2}}\||v_{\delta}|^4 v_{\delta} \|_{X^{s,b-1}}.
\end{equation}
\vspace{3mm}
We estimate the quantity $\||v_{\delta}|^4 v_{\delta}
\|_{X^{s,b-1}}$ by duality. Let us take: $$c:\mathbb{Z} \times \mathbb{R} \rightarrow
\mathbb{C},\,\mbox{such that}\,\,\sum_n\int d\tau |c(n,\tau)|^2=1.$$
Let us consider the quantity:
\begin{equation}
\label{eq:zvijezda}
\,\,\sum_n \,\, \int \,d\tau (1+|n|)^s\,
(1+|\tau+n^2|)^{b-1}\,\widetilde{(|v_{\delta}|^4v_{\delta})}(n,\tau)\,c(n,\tau)\,\,=:I
\end{equation}
Since we know:
$$\widetilde{(|v_{\delta}|^4v_{\delta})}(n,\tau)=
\sum_{n_1-n_2+n_3-n_4+n_5=n}\int_{\tau_1-\tau_2+\tau_3-\tau_4+\tau_5=\tau}d \tau_j
\widetilde{v_{\delta}}(n_1,\tau_1)\overline{\widetilde{v_{\delta}}(n_2,\tau_2)}\widetilde{v_{\delta}}(n_3,\tau_3)
\overline{\widetilde{v_{\delta}}(n_4,\tau_4)}\widetilde{v_{\delta}}(n_5,\tau_5).$$
it follows that:
$$|I| \leq \sum_n\sum_{n_1-n_2+n_3-n_4+n_5=n}\int_{\tau_1-\tau_2+\tau_3-\tau_4+\tau_5=\tau}d\tau_j
\{(1+|n|)^s(1+|\tau+n^2|)^{b-1}|c(n,\tau)|$$
$$|\widetilde{v_{\delta}}(n_1,\tau_1)||\overline{\widetilde{v_{\delta}}(n_2,\tau_2)}||\widetilde{v_{\delta}}(n_3,\tau_3)|
|\overline{\widetilde{v_{\delta}}(n_4,\tau_4)}||\widetilde{v_{\delta}}(n_5,\tau_5)|\}.$$
Since $n=n_1-n_2+n_3-n_4+n_5$, it follows that:

$$|n|^s \lesssim \max \{|n_1|^s,|n_2|^s,|n_3|^s,|n_4|^s,|n_5|^s\}.$$
By symmetry, it suffices to bound the expression:
$$I_1:=\sum_n\sum_{n_1-n_2+n_3-n_4+n_5=n}\int_{\tau_1-\tau_2+\tau_3-\tau_4+\tau_5=\tau}d\tau_j d\tau
\{\frac{|c(n,\tau)|}{(1+|\tau+n^2|)^{1-b}}$$
$$(1+|n_1|)^s|\widetilde{v_{\delta}}(n_1,\tau_1)||\overline{\widetilde{v_{\delta}}(n_2,\tau_2)}||\widetilde{v_{\delta}}(n_3,\tau_3)|
|\overline{\widetilde{v_{\delta}}(n_4,\tau_4)}||\widetilde{v_{\delta}}(n_5,\tau_5)|\}=$$
$$=\sum_n\sum_{n_1+n_2+n_3+n_4+n_5=n}\int_{\tau_1+\tau_2+\tau_3+\tau_4+\tau_5=\tau}d\tau_j d\tau
\{\frac{|c(n,\tau)|}{(1+|\tau+n^2|)^{1-b}}$$
$$(1+|n_1|)^s|\widetilde{v_{\delta}}(n_1,\tau_1)||\overline{\widetilde{v_{\delta}}(-n_2,-\tau_2)}||\widetilde{v_{\delta}}(n_3,\tau_3)|
|\overline{\widetilde{v_{\delta}}(-n_4,-\tau_4)}||\widetilde{v_{\delta}}(n_5,\tau_5)|\}.$$
\vspace{3mm}
Let us now define the following functions:
\vspace{3mm}
\begin{equation}
\label{eq:F}
F(x,t):=\sum_{n}\int d\tau \{
\frac{|c(n,\tau)|}{(1+|\tau+n^2|)^{1-b}} e^{inx+it\tau}\}.
\end{equation}
\begin{equation}
\label{eq:G}
G(x,t):=\sum_{n}\int d\tau \{ (1+|n|)^s
|\widetilde{v_{\delta}}(n,\tau)| e^{inx+it\tau}\}.
\end{equation}
\begin{equation}
\label{eq:H}
H(x,t):=\sum_{n}\int d\tau \{
|\widetilde{v_{\delta}}(n,\tau)| e^{inx+it\tau}\}.
\end{equation}
Consequently, by using Parseval's identity, one obtains:

$$I_1 \lesssim \int\int F\bar{G}H\bar{H}H\bar{H} dxdt=|\int\int  F\bar{G}H\bar{H}H\bar{H} dxdt|,$$
which by H\"{o}lder's inequality is:

\begin{equation}
\label{eq:holderproduct}
\leq \|F\|_{L^4_{t,x}} \|G\|_{L^4_{t,x}}
\|H\|^2_{L^4_{t,x}} \|H\|^2_{L^{\infty}_{t,x}}.
\end{equation}
Recalling $(\ref{eq:strichartztorus})$, and using the fact that
$b=\frac{1}{2}+$, we have \footnote{In the following calculation, and later on, we crucially use the fact that
one doesn't change the $X^{s,b}$ norm of a function when one takes absolute values in its Spacetime Fourier Transform.}:
\begin{equation}
\label{eq:boundF}
\|F\|_{L^4_{t,x}}\lesssim
\|F\|_{X^{0,\frac{3}{8}}}\leq \|F\|_{X^{0,1-b}}=\|c\|_{l^2_k
L^2_{\tau}}=1.
\end{equation}
$$ \|G\|_{L^4_{t,x}}\lesssim
\|G\|_{X^{0,\frac{3}{8}}}=\|(1+|n|)^s|\widehat{v_{\delta}}(n,\tau)|(1+|\tau+n^2|)^{\frac{3}{8}}\|_{l^2_n
L^2_{\tau}}=$$
\begin{equation}
\label{eq:boundG}
=\|v_{\delta}\|_{X^{s,\frac{3}{8}}}
\lesssim\|v\|_{X^{s,\frac{3}{8}}}\leq\|v\|_{X^{s,b}}.
\end{equation}
The implied constant in the above inequality is independent of $\delta$ by (\ref{eq:loc4}). Also:
$$\|H\|_{L^4_{t,x}}\lesssim
\|v_{\delta}\|_{X^{0,\frac{3}{8}}}$$
We interpolate between $X^{0,0}$ and
$X^{0,b}$ for an appropriate $\theta \in (0,1)$ to deduce that this is:
$$\lesssim
\|v_{\delta}\|_{X^{0,0}}^{\theta}\|v_{\delta}\|_{X^{0,b}}^{1-\theta}.$$
We estimate $\|v_{\delta}\|_{X^{0,0}}$ by:
$$\|v_{\delta}\|_{X^{0,0}}=\|v_{\delta}\|_{L^2_{x,t}}$$
which by the support properties of $\psi_{\delta}$ is:
$$=\|v_{\delta}\psi_{\delta}\|_{L^2_{x,t}} \leq
\|\psi_\delta\|_{L^4_t}\|v_{\delta}\|_{L^4_tL^2_x}$$
$$\lesssim \delta^{\frac{1}{4}} \|v_{\delta}\|_{X^{0,\frac{1}{4}+}} \lesssim  \delta^{\frac{1}{4}}
\|v_{\delta}\|_{X^{0,b}}.$$
\vspace{2mm}
Here, we have used (\ref{eq:L4tL2x}).

\vspace{2mm}

Hence:
$$\|H\|_{L^4_{t,x}} \lesssim
(\delta^{\frac{1}{4}}\|v_{\delta}\|_{X^{0,b}})^{\theta}(\|v_{\delta}\|_{X^{0,b}})^{1-\theta}=$$
\begin{equation}
\label{eq:boundH}
=\delta^{\frac{\theta}{4}}\|v_{\delta}\|_{X^{0,b}}\lesssim
\delta^{\frac{\theta}{4}+\frac{1-2b}{2}}\|v\|_{X^{0,b}}.
\end{equation}
In the last step, we used $(\ref{eq:loc2})$.
\vspace{2mm}

Furthermore, by Sobolev embedding:
\begin{equation}
\label{eq:boundH2}
\|H\|_{L^{\infty}_{t,x}} \lesssim
\|H\|_{X^{\frac{1}{2}+,\frac{1}{2}+}}=\|v_{\delta}\|_{X^{\frac{1}{2}+,\frac{1}{2}+}}
\leq \|v_{\delta}\|_{X^{1,b}}\lesssim
\delta^{\frac{1-2b}{2}}\|v\|_{X^{1,b}}.
\end{equation}

$\underline{\mbox{We calculate}\,\, \theta:}$

\vspace{2mm}

We know:
$$\frac{3}{8}=0 \cdot \theta + b \cdot (1-\theta)$$
So:
\begin{equation}
\label{eq:theta}
\theta=\frac{b-\frac{3}{8}}{b}=\frac{1+8\epsilon}{4+8\epsilon}.
\end{equation}
Combining $(\ref{eq:holderproduct})-(\ref{eq:boundH2})$, it follows
that:
$$ \||v_{\delta}|^4 v_{\delta}\|_{X^{s,b-1}}\lesssim
\|v\|_{X^{s,b}}({\delta^{\frac{\theta}{4}+\frac{1-2b}{2}}\|v\|_{X^{0,b}}})^2
(\delta^{\frac{1-2b}{2}}\|v\|_{X^{1,b}})^2\leq$$
\begin{equation}
\label{eq:nonlinearitybound} \leq \delta^{\hspace{1mm}
\theta_0+2(1-2b)}(\|v\|_{X^{1,b}})^4\|v\|_{X^{s,b}}.
\end{equation}
Here:

\begin{equation}
\label{eq:theta0}
\theta_0:=\frac{\theta}{2}=\frac{1+8\epsilon}{8+16\epsilon}.
\end{equation}
Hence, from $(\ref{eq:boundonL})$ and
$(\ref{eq:nonlinearitybound})$,we obtain:

\begin{equation}
\label{eq:boundonL2} \|Lv\|_{X^{s,b}}\leq c
\delta^{\frac{1-2b}{2}}\|\Phi\|_{H^s}+c_1
\delta^{\hspace{1mm}\theta_0
+\frac{5}{2}(1-2b)}(\|v\|_{X^{1,b}})^4\|v\|_{X^{s,b}}.
\end{equation}
Here $c,c_1>0$ depend on $s$.

\vspace{2mm}

If we take $c,c_1$ possibly even smaller, and if we repeat the
previous argument in the special case $s=1$, it follows that:

\begin{equation}
\label{eq:boundonL2_1} \|Lv\|_{X^{1,b}}\leq c
\delta^{\frac{1-2b}{2}}\|\Phi\|_{H^1}+c_1
\delta^{\hspace{1mm}\theta_0 +\frac{5}{2}(1-2b)}(\|v\|_{X^{1,b}})^5.
\end{equation}
Now, we estimate $\|Lv-Lw\|_{X^{1,b}}$. In order to do this, we note
that:

\vspace{3mm}

$|v|^4v-|w|^4w=\,\,$Sum of quintic terms, each of which contains at
least one factor of $v-w$ or $\overline{v-w}$. By the above proof, since
the estimates $(\ref{eq:boundonL2})$ depended only on bounds on
spacetime norms in $x,t$, we can put complex conjugates in the appropriate factors (so if
$v-w$ comes with a conjugate, it doesn't matter). Furthermore, by
the triangle inequality, we know:
$\|v-w\|_{X^{1,b}}\leq\|v\|_{X^{1,b}}+\|w\|_{X^{1,b}}$. Thus,
arguing as before, we can obtain, for some $c_2>0:$

\begin{equation}
\label{eq:boundonL3}
\|Lv-Lw\|_{X^{1,b}}\leq c_2
\delta^{\hspace{1mm}\theta_0
+\frac{5}{2}(1-2b)}(\|v\|_{X^{1,b}}+\|w\|_{X^{1,b}})^4\|v-w\|_{X^{1,b}}.
\end{equation}

\vspace{4mm}

Let

\begin{equation}
\label{eq:Gamma}
\Gamma:=\{v:\|v\|_{X^{s,b}}\leq
2c\delta^{\frac{1-2b}{2}}\|\Phi\|_{H^s},\|v\|_{X^{1,b}}\leq
2c\delta^{\frac{1-2b}{2}}\|\Phi\|_{H^1}\}.
\end{equation}
Let us give $\Gamma$ the metric $d(v,w):=\|v-w\|_{X^{1,b}}$. Then, by
Proposition \ref{Proposition 3.2}, $(\Gamma,d)$ is a Banach space.

From $(\ref{eq:boundonL2})$, we have for all $v\in \Gamma$

$$\|Lv\|_{X^{s,b}}\leq c \delta^{\frac{1-2b}{2}} \|\Phi\|_{H^s} +
c_1 \delta^{\theta_0+\frac{5}{2}(1-2b)}(2c\delta^{\frac{1-2b}{2}}\|\Phi\|_{H^1})^4
\hspace{1mm}2c\delta^{\frac{1-2b}{2}}\|\Phi\|_{H^s}=$$

\begin{equation}
\label{eq:Lvbound_s}
=c\delta^{\frac{1-2b}{2}}\|\Phi\|_{H^s}(1+32c_1c^4\delta^{\hspace{1mm}\theta_0+\frac{9}{2}(1-2b)}\|\Phi\|_{H^1}^4).
\end{equation}
Analogously, from $(\ref{eq:boundonL2_1})$ :

\begin{equation}
\label{eq:Lvbound_1} \|Lv\|_{X^{1,b}}\leq
c\delta^{\frac{1-2b}{2}}\|\Phi\|_{H^1}(1+32c_1c^4\delta^{\hspace{1mm}\theta_0+\frac{9}{2}(1-2b)}\|\Phi\|_{H^1}^4).
\end{equation}
Finally, if $v,w \in \Gamma,\,\,(\ref{eq:boundonL3})$ implies that:

$$\|Lv-Lw\|_{X^{1,b}}\leq c_2
\delta^{\hspace{1mm}\theta_0+\frac{5}{2}(1-2b)}(4c\delta^{\frac{1-2b}{2}}\|\Phi\|_{H^1})^4
\|v-w\|_{X^{1,b}}$$

\begin{equation}
\label{eq:Lv-Lw bound} \leq 256
c_2c^4\delta^{\hspace{1mm}\theta_0+\frac{9}{2}(1-2b)}
\|\Phi\|_{H^1}^4 \|v-w\|_{X^{1,b}}.
\end{equation}

\vspace{3mm}

We recall that $\theta_0=\frac{1+8\epsilon}{8+16\epsilon}, b=\frac{1}{2}+\epsilon$.
We observe that for $\epsilon>0$ sufficiently small, one has

\begin{equation}
\label{eq:epsilonuvjet}
\frac{1+8\epsilon}{8+16\epsilon}-9\epsilon>0
\end{equation}
From now, let us fix $\epsilon$ to satisfy the condition $(\ref{eq:epsilonuvjet})$. In other words, we have:

\begin{equation}
\label{eq:epsiloncondition}
\theta_0+\frac{9}{2}(1-2b)>0.
\end{equation}

\vspace{3mm}

Hence, we can choose $\delta>0$ sufficiently small such that:

\begin{equation}
\label{eq:contraction1}
32c_1c^4\delta^{\hspace{1mm}\theta_0+\frac{9}{2}(1-2b)}\|\Phi\|_{H^1}^4\leq
1.
\end{equation}

\begin{equation}
\label{eq:contraction2} 256
c_2c^4\delta^{\hspace{1mm}\theta_0+\frac{9}{2}(1-2b)}\|\Phi\|_{H^1}^4\leq
\frac{1}{2}.
\end{equation}
From $(\ref{eq:contraction1}),(\ref{eq:contraction2})$, the
preceding bounds and the fact that $(\Gamma,d)$ is a Banach Space,
it follows that \emph{L has a fixed point} $v\in \Gamma.$

\vspace{3mm}

By construction of L, for this $v$, we know:

\begin{itemize}

\item $v(t_0)=\Phi.$

\item $iv_t+\Delta v=|v|^4v \mbox{\,\,for\,\,} t\in
[t_0-\delta,t_0+\delta]$, and hence by uniqueness (which is proved
by an application of Gronwall's inequality), it follows that:

$$v=u \mbox{\,\,for\,\,}t \in [t_0-\delta,t_0+\delta].$$

\item $\|v\|_{X^{s,b}}\leq 2 c\delta^{\frac{1-2b}{2}} \|\Phi\|_{H^s}=2 c\delta^{\frac{1-2b}{2}} \|u_0\|_{H^s}.$

\end{itemize}

It just remains to address the issue of choosing $\delta$ uniformly
in $t_0$. However, from
$(\ref{eq:contraction1}),(\ref{eq:contraction2})$, it follows that
we just want $\delta$ to satisfy:
\begin{equation}
\label{eq:conditionondelta}
\delta^{\theta_0+\frac{9}{2}(1-2b)} \|\Phi\|_{H^1}^4\lesssim
1.
\end{equation}
By the fact that:
$$\|\Phi\|_{H^1}\lesssim_{Mass(u),Energy(u)}1$$
it follows that we can choose $$\delta \sim_{Mass(u),Energy(u)}1$$ which is
uniform in time, so the previous procedure can be iterated with
fixed increment $\delta$.

This proves (\ref{eq:properties of v1}) and (\ref{eq:properties of v2}). We now have to prove (\ref{eq:properties of v3}).

Let us recall that the function $v$ that we have constructed satisfies:
\begin{equation}
\label{eq:X1bbound}
\|v\|_{X^{1,b}}\leq c \delta^{\frac{1-2b}{2}} \|\Phi\|_{H^1}.
\end{equation}

\begin{equation}
\label{eq:Xsbbound}
\|v\|_{X^{s,b}}<\infty.
\end{equation}
and

$$Lv=\chi_{\delta}(t)S(t)\Phi-i\chi_{\delta}(t)\int_{0}^{t}S(t-t')|v_{\delta}|^4v_{\delta}(t')dt'.$$
We take $\mathcal{D}$'s in the previous equation, and since $\mathcal{D}$ acts only on the spatial variables (as a Fourier multiplier), we obtain:

$$\mathcal{D}v=\chi_{\delta}(t)S(t)\mathcal{D}\Phi - i\chi_{\delta}(t)\int_0^t S(t-t') \mathcal{D}(|v_{\delta}|^4 v_{\delta}(t'))dt'.$$
We know that:

$$\forall\, m,n \in \mathbb{Z},\,\theta(m+n)\lesssim_s \theta(m)+\theta(n).$$
From this \emph{``Fractional Leibniz Rule''}, we deduce that for $n=n_1-n_2+n_3-n_4+n_5$, one has:

$$\theta(n)\lesssim_s max \{\theta(n_1),\theta(n_2),\theta(n_3),\theta(n_4),\theta(n_5)\}.$$
So, arguing analogously as earlier (c.f. (\ref{eq:boundonL})), we obtain:

$$\|\mathcal{D}v\|_{X^{0,b}}\leq c_1 \delta^{\frac{1-2b}{2}} \|\mathcal{D}\Phi\|_{L^2}+
c_2 \delta^{\frac{1-2b}{2}}\|\mathcal{D}(|v_{\delta}|^4v_{\delta})\|_{X^{0,b-1}}\leq$$
$$\leq  c_1 \delta^{\frac{1-2b}{2}} \|\mathcal{D}\Phi\|_{L^2}+
c_3 \delta^{\theta_0+\frac{5(1-2b)}{2}}\|v\|_{X^{1,b}}^4 \|\mathcal{D}v\|_{X^{0,b}}.$$
By using (\ref{eq:X1bbound}), we get:

$$\|\mathcal{D}v\|_{X^{0,b}}\leq  c_1 \delta^{\frac{1-2b}{2}} \|\mathcal{D}\Phi\|_{L^2}+
c_4 \delta^{\theta_0+\frac{9(1-2b)}{2}}\|\Phi\|_{H^1}^4 \|\mathcal{D}v\|_{X^{0,b}}.$$

By using $(\ref{eq:epsiloncondition})$, we can choose $\delta>0$ (possibly smaller than the one chosen before), such that:

\begin{equation}
\label{eq:conditionondelta1}
c_4 \delta^{\theta_0 + \frac{9}{2}(1-2b)} \|\Phi\|_{H^1}^4 \leq \frac{1}{2}.
\end{equation}
Observe that then $\delta=\delta(s,Energy,Mass).$ Also, we note that choosing $\delta$ to be even smaller than
the one chosen in the proof of (\ref{eq:properties of v1}),(\ref{eq:properties of v2}), yet still depending only on
$(s,Energy,Mass)$ doesn't create problems with the estimates on $\|v\|_{X^{1,b}},\|v\|_{X^{s,b}}$ we had earlier.

Note that:

$$\|\mathcal{D}v\|_{X^{0,b}}\leq \|v\|_{X^{s,b}}<\infty.$$
where in the last inequality, we were using (\ref{eq:Xsbbound}).

Hence:

$$\|\mathcal{D}v\|_{X^{0,b}}\leq c_1\delta^{\frac{1-2b}{2}}\|\mathcal{D}\Phi\|_{L^2}+\frac{1}{2}\|\mathcal{D}v\|_{X^{0,b}}.$$
implies:

$$\|\mathcal{D}v\|_{X^{0,b}}\leq 2c_1 \|\mathcal{D}\Phi\|_{L^2}.$$
In other words, we obtain:

$$\|\mathcal{D}v\|_{X^{0,\frac{1}{2}+}}\lesssim \|\mathcal{D}\Phi\|_{L^2}.$$
with the explicit constant depending only on $(s,Energy,Mass)$.

\vspace{2mm}

We may now conclude that (\ref{eq:properties of v3}) holds.

\vspace{2mm}

It remains to see the continuity of $\delta,C$ in the energy and mass. We recall from the construction of $\delta$ (c.f. (\ref{eq:conditionondelta}),(\ref{eq:conditionondelta1})) that we want, for some $\gamma>0$:

$$\delta \lesssim \|\Phi\|_{H^1}^{-\gamma}.$$
Since $\|\Phi\|_{H^1}^2 \lesssim M(\Phi) + E(\Phi)$, we take:

\begin{equation}
\label{eq:deltachoice}
\delta \sim (M(\Phi)+E(\Phi))^{-\frac{\gamma}{2}}
\end{equation}
Such a $\delta$ depends continuously on the energy and mass. We notice that the $C$ is obtained as a continuous function of $\delta$, and the bounds on the $H^1$ norm of a solution, so it also depends continuously on energy and mass.

This proves Proposition \ref{Proposition 3.1} in the case $k=2$.

\vspace{3mm}
If we are considering the general case $k \geq 2$, we have to modify the previous proof to consider the map:

$$Lv:=\chi_{\delta}(t)S(t)\Phi-i\chi_{\delta}(t)\int_{0}^{t}S(t-t')|v|^{2k}v(t')dt'.$$
Arguing as in  (\ref{eq:boundonL}), we deduce:

$$\|Lv\|_{X^{s,b}}\leq c
\delta^{\frac{1-2b}{2}}\|\Phi\|_{H^s}+c
\delta^{\frac{1-2b}{2}}\||v_{\delta}|^{2k} v_{\delta} \|_{X^{s,b-1}}.$$

\vspace{2mm}

One then estimates the quantity $\||v_{\delta}|^{2k} v_{\delta} \|_{X^{s,b-1}}$ by duality.

\vspace{3mm}

The extra $k-2$ terms that are obtained here are estimated in $\|\cdot\|_{L^{\infty}_{t,x}}$ after an application
of H\"{o}lder's inequality and we again use the fact that:
$X^{\frac{1}{2}+,\frac{1}{2}+}\hookrightarrow L^{\infty}_{t,x}$. The proof then follows similarly as in the case $k=2$. We omit the details.

\end{proof}

We now present the proof of Proposition \ref{Proposition 3.2}, by which we can iterate our
construction without changing the size of the increment:

\vspace{3mm}

\begin{proof}(\emph{of Proposition \ref{Proposition 3.2}})

The proof of this remarkable fact uses the special structure of the
$X^{s,b}$ spaces. The main ingredient is the following fact, taken
from \cite{Ca}:

\newtheorem*{Theorem 1.2.5}{Theorem 1.2.5}
\begin{Theorem 1.2.5}``Consider two Banach spaces $X \hookrightarrow
Y$ and $1<p,q\leq \infty$ and an open interval $I \subseteq
\mathbb{R}$ (which can equal $\mathbb{R})$. Let $(f_n)_{n \geq 0}$ be
a bounded sequence in $L^q(I,Y)$, and let $f:I \rightarrow Y$ be
such that: $f_n(t) \rightharpoonup f(t)$ in $Y$ as $n \rightarrow
\infty$ for a.e. $t \in I$. If $(f_n)_{n \geq 0}$ is bounded in
$L^p(I,X)$ and if $X$ is reflexive, then $f \in L^p(I,X)$ and
$\|f\|_{L^p(I,X)}\leq \lim \inf \|f_n\|_{L^p(I,X)}.''$
\end{Theorem 1.2.5}

We now work on the Fourier transform side. For $\sigma \geq 0$, we define:

$$h_n^{\sigma}:=\{(b_n)_{n \in \mathbb{Z}}:(\sum_n
(1+|n|)^{2\sigma}|b_n|^2)^{\frac{1}{2}}< \infty\}$$

$$\|b\|_{h^n_{\sigma}}:=(\sum_n
(1+|n|)^{2\sigma}|b_n|^2)^{\frac{1}{2}}.$$
In this way, we get a Hilbert space, which is in particular a
reflexive Banach space.
The set $ B:=\{v:\|v\|_{X^{s,b}}\leq R\}$ is identified with the set:

$$E:=\{\tilde{v}:\mathbb{Z}\times \mathbb{R} \rightarrow \mathbb{C}:\,
\sum_n \int d \tau
(1+|\tau+n^2|)^{2b}(1+|n|)^{2s}|\tilde{v}(n,\tau)|^2\leq R^2\}.,$$
with the metric given by:

$$d(\tilde{v},\tilde{w}):=(\sum_n \int d \tau
(1+|\tau+n^2|)^{2b}(1+|n|)^{2s}|\tilde{v}(n,\tau)-\tilde{w}(n,\tau)|^2)^{\frac{1}{2}}.$$
We will now apply Theorem 1.2.5 from \cite{Ca} with:

$$X=h_n^s,Y=h_n^1,p=q=2,I=\mathbb{R}.$$

\vspace{3mm}

Let us now start with $(u_r)_{r \geq 0}$ a sequence in $B$ such
that: $u_r \rightarrow u$ as $r \rightarrow \infty$ in $X^{1,b}$ and
we want to argue that $u \in B.$

\vspace{2mm}

Let us take:

$$f_r(n,\tau):=(1+|\tau+n^2|)^b \widetilde{u_r}(n,\tau).$$
Then:

$$\|f_r\|_{L^2_{\tau}h_n^s} \leq R.$$
The claim we want to prove is:

$$\|f\|_{L^2_{\tau}h_n^s} \leq R\,\,\mbox{where}\,\,f(n,\tau):=(1+|\tau+n^2|)^b \widetilde{v}(n,\tau).$$
\vspace{2mm}
We know that:

$$\|u_r-u\|_{X^{1,b}} \rightarrow 0.$$
Thus:

$$\|\|f_r(\tau)-f(\tau)\|_{h_n^1}\|_{L^2_{\tau}} \rightarrow 0.$$
Hence:

$$\|f_r(\tau)-f(\tau)\|_{h_n^1} \rightarrow 0\,\,\mbox{in measure as
a function of}\,\,\tau.$$
Thus, we can pass to a subsequence of $(f_r)_{r \geq 0}$ which we
again call $(f_r)$ such that:

$$\|f_r(\tau)-f(\tau)\|_{h_n^1} \rightarrow 0\,\,\mbox{pointwise almost everywhere as a function of}\,\,\tau.$$
In particular:

$$f_r(\tau) \rightarrow f(\tau) \,\,\mbox{in}\,\,
h_n^1=Y,\,\,\mbox{for almost every}\,\,\tau.$$
So:

$$f_r(\tau) \rightharpoonup f(\tau) \,\,\mbox{in}\,\,
h_n^1=Y,\,\,\mbox{for almost every}\,\,\tau.$$
Now, Theorem 1.2.5 from \cite{Ca} implies that:

\begin{equation}
\label{eq:closedballbound}
\|f\|_{L^2_{\tau}h_n^s} \leq \lim \inf
\|f_r\|_{L^2_{\tau}h_n^s} \leq R.
\end{equation}

\end{proof}

We now prove the Approximation Lemma.

\begin{proof}(\emph{of Proposition \ref{Proposition 3.3}})

\vspace{2mm}

With notation as in the statement of the Proposition, we consider $n$ sufficiently large so that:
$$M(\Phi_n) \sim M(\Phi),E(\Phi_n) \sim E(\Phi), \|\Phi_n\|_{H^s} \sim \|\Phi\|_{H^s}.$$
Let us denote $N(f):=|f|^{2k}f.$

\vspace{2mm}

With notation as in the proof of Proposition \ref{Proposition 3.1}, we define:

$$Lv:= \chi_{\delta}(t)S(t)\Phi-i\chi_{\delta}(t)\int_0^t S(t-t')N(v)(t')dt'.$$

$$L_n v^{(n)}:= \chi_{\delta}(t)S(t)\Phi_n-i\chi_{\delta}(t)\int_0^t S(t-t')N(v^{(n)})(t')dt'.$$
From our earlier arguments, we can choose $\delta=\delta(s,E(\Phi),M(\Phi))>0$ sufficiently small so that $L$ has a fixed point $v$ that coincides with $u$ for $t \in [0,\delta]$, and which satisfies:

\begin{equation}
\label{eq:lwp}
\|v\|_{X^{s,b'}}\leq C(s,E(\Phi),M(\Phi))\delta^{\frac{1-2b'}{2}}\|\Phi\|_{H^s}, b'=\frac{1}{2}+
\end{equation}

\vspace{2mm}

Let us fix $T>\delta$. By just iterating the local well-posedness bound, we get that for all $t\in [0,T]$:

\begin{equation}
\label{eq:C2bound}
\|u(t)\|_{H^s} \leq C(s,E(\Phi),M(\Phi))\|\Phi\|_{H^s}e^{C_1(s,E(\Phi),M(\Phi))T}=:C_2.
\end{equation}

Hence $C_2=C_2(s,E(\Phi),M(\Phi),\|\Phi\|_{H^s},T)>0.$

\vspace{2mm}

We can repeat the same for $L_n$ to obtain a fixed point $v^{(n)}$ which coincides with $u^{(n)}$ for $t\in [t_0,t_0+\delta]$.
The $\delta$ and $C_2$ will remain equivalent to the ones chosen earlier.

Then:

$$\|v-v^{(n)}\|_{X^{s,b'}}=\|Lv-L_n v^{(n)}\|_{X^{s,b'}} \leq$$

$$\|\chi_{\delta}(t)S(t)\Phi - \chi_{\delta}(t)S(t)\Phi_n\|_{X^{s,b'}}+
\|\chi_{\delta}(t)\int_0^t S(t-t')(N(v)-N(v^{(n)}))(t')dt'\|_{X^{s,b'}}\leq$$

$$c \delta^{\frac{1-2b'}{2}} \|\Phi-\Phi_n\|_{H^s} + c\delta^{r_0}(P(\|v\|_{X^{s,b'}})+P(\|v^{(n)}\|_{X^{s,b'}}))\|v-v^{(n)}\|_{X^{s,b'}}.$$
Here, $c>0$ is a universal constant, $P$ is a polynomial of fixed degree such that $P(0)=0$, and $r_0>0$ is fixed (independent of $b'$).

Hence, by (\ref{eq:lwp}), it follows that:

$$\|v-v^{(n)}\|_{X^{s,b'}} \leq c\delta^{\frac{1-2b'}{2}}\|\Phi-\Phi_n\|_{H^s}+
c \delta^{r_0}(P(\delta^{\frac{1-2b'}{2}}\|\Phi\|_{H^s})+P(\delta^{\frac{1-2b'}{2}}\|\Phi_n\|_{H^s}))\|v-v^{(n)}\|_{X^{s,b'}} $$

\begin{equation}
\label{eq:v-vn}
\leq c\delta^{\frac{1-2b'}{2}}\|\Phi-\Phi_n\|_{H^s}+ \tilde{c} \delta^{r_0}P(\delta^{\frac{1-2b'}{2}}C_2)\|v-v^{(n)}\|_{X^{s,b'}}.
\end{equation}
The last inequality was obtained by combining (\ref{eq:lwp}) and (\ref{eq:C2bound}).

\vspace{2mm}

We now choose $\delta$ even smaller such that:

$$\tilde{c} \delta^{r_0}P(\delta^{\frac{1-2b'}{2}}C_2)\leq \frac{1}{2}.$$
By choosing $\delta$ even smaller, the previous estimate (\ref{eq:lwp}), and all the subsequent estimates will remain otherwise unchanged. The new $\delta=\delta(s,E(\Phi),M(\Phi),C_2)>0$ now also depends on $C_2$.

\vspace{2mm}

We obtain:

\vspace{2mm}

$$\|v-v^{(n)}\|_{X^{s,b'}}\leq 2c \delta^{\frac{1-2b'}{2}}\|\Phi-\Phi_n\|_{H^s}.$$
By using $(\ref{eq:C2bound})$, it follows that, this bound can be iterated on $\sim \frac{T}{\delta}$ time intervals, with the same $\delta$. Namely, in the definition of $L,L_n$, we just have to consider $\chi_{\delta}(\cdot-r)$ for an appropriate time translation $r$, and instead of $\Phi,\Phi_n$ as initial data, we consider $u(r),u^{(n)}(r)$ respectively.

Furthermore, let us use the fact that: $X^{s,b'} \hookrightarrow L^{\infty}_t H^s_x$ to deduce that, for the large enough $n$ we are considering, one has:

\vspace{2mm}

$$\|u(t)-u^{(n)}(t)\|_{H^s}\leq C\|\Phi-\Phi_n\|_{H^s}.$$
\vspace{2mm}
Here, $C=C(s,E(\Phi),M(\Phi),\|\Phi\|_{H^s},T)>0.$

\vspace{2mm}

The claim now follows.

\end{proof}

\begin{remark}
The proof of Proposition \ref{Proposition 3.3} is also valid for $(\ref{eq:Hartree})$, $(\ref{eq:potentialcubicnls})$, and $(\ref{eq:inhomogeneouscubicnls})$.
\end{remark}

\end{document}